\pdfoutput=1

\documentclass[pdflatex,sn-mathphys-num]{sn-jnl}

\usepackage{graphicx}
\usepackage{amsmath,amssymb,amsfonts}
\usepackage{enumitem}
\usepackage{tikz}
\usetikzlibrary{angles,quotes,calc,arrows.meta,patterns,positioning,shapes.geometric}
\usepackage{cleveref}
\hypersetup{hidelinks}

\theoremstyle{thmstyleone}
\newtheorem{theorem}{Theorem}[section]
\newtheorem{lemma}[theorem]{Lemma}
\newtheorem{proposition}[theorem]{Proposition}
\newtheorem{corollary}[theorem]{Corollary}

\theoremstyle{thmstylethree}
\newtheorem{definition}[theorem]{Definition}

\theoremstyle{thmstyletwo}
\newtheorem{remark}[theorem]{Remark}

\begin{document}

\title[Robin Non-Localization]{Uniform Non-Localization under Robin Boundary Perturbations:
Spectral Splitting and Bounded Eigenspace Complexity}

\author*[1,2]{\fnm{Binh T.} \sur{Nguyen}}\email{ngtbinh@hcmus.edu.vn}

\affil*[1]{\orgdiv{Faculty of Mathematics and Computer Science},
\orgname{University of Science, Vietnam National University Ho Chi Minh City},
\orgaddress{\city{Ho Chi Minh City}, \country{Vietnam}}}

\affil[2]{\orgname{AISIA Lab},
\orgaddress{\city{Ho Chi Minh City}, \country{Vietnam}}}

\abstract{We study the high-energy spectral effect of Robin boundary perturbations on complete Laplace eigenspaces and its consequences for eigenfunction non-localization. Highly degenerate Neumann levels generally split under Robin perturbation, but this splitting need not produce a simple spectrum: distinct modal classes may still contribute to the same Robin eigenvalue. We prove that, for every fixed positive Robin parameter, the number of modal classes contributing to any one eigenvalue is uniformly bounded over the spectrum on the equilateral triangle and on every rectangle with rational squared aspect ratio.

On the equilateral triangle, for all sufficiently small Robin parameters,
the Neumann complete-eigenspace observation estimate persists with a
constant uniform both in the eigenvalue and in the boundary parameter.  For
an arbitrary fixed positive Robin parameter, a high-frequency shell-splitting
analysis gives a spectrum-wide bound on the number of modal classes
contributing to any one Robin eigenvalue.  In the nonsquare rectangular
case, the corresponding splitting is governed by a strongly convex profile
on weighted quadratic shells.  Since each modal class has uniformly bounded
plane-wave complexity, the spectral complexity bound implies observation on
every measurable set of positive measure through a multidimensional
Tur\'an--Nazarov inequality, without a frequency-separation assumption.
Thus, spectral simplicity is not required for complete-eigenspace
non-localization: uniformly bounded spectral coincidence complexity is
sufficient.}



\maketitle


\section{Introduction}
\label{sec:introduction}

Let $\Omega\subset\mathbb R^2$ be a bounded domain, and let
$E_\lambda$ denote an eigenspace of a self-adjoint realization of the
Laplacian on $\Omega$.  Given a measurable set $V\subset\Omega$ with
$|V|>0$, we consider the complete-eigenspace observation problem
\begin{equation}
    \inf_{\lambda}
    \inf_{0\ne u\in E_\lambda}
    \frac{\|u\|_{L^2(V)}^2}
         {\|u\|_{L^2(\Omega)}^2}
    >0.
\label{eq:intro_unl}
\end{equation}
The infimum over the entire eigenspace is essential.  If an eigenvalue
is multiple, an arbitrary eigenfunction may be a linear combination
of several modes, so estimates for a distinguished eigenbasis do not
in general imply \eqref{eq:intro_unl}.

Robin boundary conditions provide a natural perturbation of the Neumann
spectrum in which both eigenvalues and eigenspace multiplicities may change.
At high energy this raises a question not visible at the level of individual
eigenvalue shifts: can the splitting and recombination of spectral branches
produce eigenspaces of increasing complexity, allowing eigenfunctions to
concentrate away from a fixed region of the domain?  The problem therefore
links the high-energy Robin spectrum to the spatial structure of complete
eigenspaces.

This paper studies the stability of \eqref{eq:intro_unl} under Robin
boundary perturbations.  Our principal model is the equilateral
triangle $T$, with
\begin{equation}
    -\Delta u=\lambda u
    \quad\text{in }T,
    \qquad
    \partial_\nu u+\sigma u=0
    \quad\text{on }\partial T,
    \qquad
    \sigma\geq0.
\label{eq:intro_robin_bc}
\end{equation}
The Neumann problem corresponds to $\sigma=0$.  The question is
whether complete-eigenspace observation remains uniform under this
perturbation, and what spectral information is sufficient to retain
such an estimate when the Robin spectrum is not simple.

Two distinct forms of uniformity arise.  Near the Neumann endpoint, the
observation constant must be uniform simultaneously in the spectral level
and in the Robin parameter.  For an arbitrary fixed $\sigma>0$, no
uniformity with respect to $\sigma$ is asserted.  The relevant
spectrum-wide question is instead whether the number of distinct modal
classes contributing to a single Robin eigenspace remains uniformly
bounded over the spectral level.  The two cases therefore require
different spectral mechanisms.

Our first result is uniform up to the Neumann endpoint.  There exists
$\sigma_0>0$ such that, for every measurable $V\subset T$ with
$|V|>0$,
\begin{equation}
    \inf_{0\leq\sigma\leq\sigma_0}
    \inf_{\lambda\in\operatorname{Spec}(-\Delta_\sigma)}
    \inf_{0\ne u\in E_{\lambda,\sigma}}
    \frac{\|u\|_{L^2(V)}^2}
         {\|u\|_{L^2(T)}^2}
    >0.
\label{eq:intro_small_robin}
\end{equation}
The small-parameter result is perturbative only in the boundary
parameter: the estimate is uniform over the entire, unbounded spectrum.
Pointwise convergence of individual Robin eigenvalues or eigenfunctions as
$\sigma\downarrow0$ is therefore insufficient, and it does not control
complete eigenspaces of nontrivial multiplicity.  We instead compare the
Robin modal subspaces directly with their Neumann limits and obtain a bound
uniform in the modal indices.  In particular, the relevant subspace gap is
$O(\sqrt{\sigma})$.  This rate is sharp for the full family of modal indices,
since the lowest Robin mode already exhibits square-root behavior.

For an arbitrary fixed $\sigma>0$, the problem is no longer
perturbative.  Eigenvalue displacement alone does not control complete
eigenspaces: distinct perturbed modal classes may still meet at the same
Robin eigenvalue.  The fixed-parameter problem is therefore one of spectral
coincidence complexity.  Let
$\kappa_\sigma(\lambda)$ denote the number of desymmetrized modal
classes contributing to the Robin eigenvalue $\lambda$, and set
\begin{equation}
    K_\sigma
    :=
    \sup_{\lambda\in\operatorname{Spec}(-\Delta_\sigma)}
    \kappa_\sigma(\lambda).
\label{eq:intro_complexity_condition}
\end{equation}
For a fixed $\sigma>0$, let
\[
K_\sigma
:=
\sup_{\lambda\in\operatorname{Spec}(-\Delta_\sigma)}
\#\{\text{modal classes contributing to the eigenspace at }\lambda\}.
\]
The principal fixed-parameter spectral statement is
\begin{equation}
    K_\sigma<\infty
    \qquad
    \text{for every fixed }\sigma>0.
\label{eq:intro_global_complexity}
\end{equation}
Thus, spectral simplicity is not required.  Accidental coincidences may
occur, but their complexity remains uniformly bounded over the spectrum.

To prove \eqref{eq:intro_global_complexity}, we first obtain
high-frequency splitting within the arithmetic Neumann shells.  Apart
from finitely many low-index configurations, the Robin correction
strictly orders the relevant modal classes on each shell.  A uniform
Robin--Neumann displacement bound then restricts a given Robin
eigenvalue to a bounded number of neighboring reference shells.  The
resulting bound is independent of the spectral level, although it may
depend on the fixed parameter $\sigma$.

The spectral bound has a direct consequence for observation.  We prove
that a subspace spanned by a uniformly bounded number of plane waves
satisfies an $L^2$ observation inequality on every measurable set of
positive measure, with a constant independent of the frequencies and
of their separation.  Since every Robin modal class has uniformly
bounded plane-wave complexity, \eqref{eq:intro_global_complexity}
implies a uniform Fourier-complexity bound for the complete
eigenspaces.  Hence, for every fixed $\sigma>0$ and every measurable
$V\subset T$ with $|V|>0$,
\begin{equation}
    \inf_{\lambda\in\operatorname{Spec}(-\Delta_\sigma)}
    \inf_{0\ne u\in E_{\lambda,\sigma}}
    \frac{\|u\|_{L^2(V)}^2}
         {\|u\|_{L^2(T)}^2}
    >0.
\label{eq:intro_fixed_robin}
\end{equation}
This yields the structural principle underlying the fixed-parameter
theory.  Observation of complete eigenspaces does not require spectral
simplicity.  It is enough to have a uniform bound on spectral coincidence
complexity, provided the individual modal classes have uniformly bounded
Fourier complexity.  In this sense, bounded eigenspace complexity replaces
simplicity as the relevant spectral condition.

The same principle is realized in a second spectral setting.  We prove
the fixed-parameter result for the square and, more generally, for every
rectangle
\[
    R_{a,b}=(0,a)\times(0,b)
\]
whose squared aspect ratio satisfies $a^2/b^2\in\mathbb Q$.  For every
fixed $\sigma>0$,
\begin{equation}
    K_{\sigma,a,b}^{\mathrm{rect}}<\infty,
\label{eq:intro_rect_complexity}
\end{equation}
and the corresponding complete-eigenspace observation estimate holds
on every measurable subset of $R_{a,b}$ of positive measure.  The two
geometries exhibit different spectral mechanisms: the triangle is governed
by a nonseparable secular system, whereas rational rectangles are separable
and lead to weighted arithmetic shells.  In the nonsquare case, the Neumann
reference levels are weighted quadratic shells.  The leading Robin correction is strongly convex with a unique
interior minimum, and the discrete spacing of admissible shell points
gives strict ordering on the two sides of this minimum.  Separate edge
estimates complete the same-shell coincidence bound.

The results above have three complementary components.  The first is
perturbative: on the equilateral triangle, complete-eigenspace observation
is uniform simultaneously in the spectral level and in the Robin parameter
for $0\leq\sigma\leq\sigma_0$.  The second is spectral: for every fixed
$\sigma>0$, the number of Robin modal classes contributing to a single
eigenvalue is uniformly bounded over the spectrum, without any simplicity
assumption.  The third is an observation principle: uniformly bounded modal
coincidence complexity, together with uniformly bounded plane-wave
complexity of the individual modal classes, yields measurable-set
observation with no frequency-separation hypothesis.  The fixed-parameter
spectral and observation results hold both on the equilateral triangle and
on every rectangle with rational squared aspect ratio.

The quantifiers in the perturbative and fixed-parameter results are
different.  The estimate \eqref{eq:intro_small_robin} is uniform for
$0\leq\sigma\leq\sigma_0$.  For an arbitrary fixed $\sigma>0$, the
constant in \eqref{eq:intro_fixed_robin} may depend on $\sigma$, and in the
rectangular case also on the fixed aspect ratio.  No uniform lower bound is
asserted over all $\sigma>0$ or over the family of rational aspect ratios.

\section{Related work}

The problem considered here lies at the intersection of high-energy Robin
spectral theory, eigenfunction localization, polygonal spectral geometry,
and observability.  A distinction relevant throughout is that between
results on individual eigenvalue branches or selected eigenfunctions and
estimates that hold uniformly for every vector in every eigenspace.  The
latter distinction becomes essential in the presence of spectral
multiplicity, since bounds for a chosen eigenbasis need not be preserved
under arbitrary linear combinations.

Localization of Laplace eigenfunctions and its dependence on the geometry of
the domain have been studied extensively; see the survey of Grebenkov and
Nguyen~\cite{GrebenkovNguyen2013}.  Explicit high-frequency localization
mechanisms, including whispering-gallery, bouncing-ball, and focusing modes,
were analyzed by Nguyen and Grebenkov~\cite{NguyenGrebenkov2013} in circular,
spherical, elliptical, and related separable geometries.  Localization can
also be generated by singular or degenerating geometry.  Nazarov, P\'erez,
and Taskinen~\cite{NazarovPerezTaskinen2016} studied localization of
Dirichlet eigenfunctions in thin nonsmooth domains; G\'omez, Nazarov, and
P\'erez-Mart\'inez~\cite{GomezNazarovPerezMartinez2021} obtained localization
effects for Dirichlet problems in domains surrounded by thin stiff and heavy
bands; and Cardone, Nazarov, and
Taskinen~\cite{CardoneNazarovTaskinen2024} proved localization and exponential
decay for eigenfunctions in a thin-walled Dirichlet beaker.  Further
perspectives on localization under geometric degeneration are given by
van den Berg and Bucur~\cite{VanDenBergBucur2025}, while
Steinerberger~\cite{Steinerberger2023} studied related questions concerning
growth and concentration of Laplace eigenfunctions on compact manifolds.

Quantitative non-localization has received comparatively less attention.
In one dimension, Liard, Lissy, and
Privat~\cite{LiardLissyPrivat2018} obtained uniform non-localization estimates
for Sturm--Liouville eigenfunctions, including observation on measurable
subsets.  For rational polygons, Marklof and
Rudnick~\cite{MarklofRudnick2012} proved configuration-space
equidistribution along a density-one subsequence of Dirichlet eigenfunctions.
This gives asymptotic nonconcentration for almost all eigenfunctions but
allows exceptional subsequences, and therefore does not yield an estimate
uniform over the complete spectrum.  In a different direction, Ceki\'c,
Georgiev, and Mukherjee~\cite{CekicGeorgievMukherjee2020} proved uniform
lower mass for Dirichlet and Neumann eigenfunctions near the singular
boundary of convex polyhedral billiards.  Their observation region is tied
to the singular set of the billiard, whereas the problem considered here
allows an arbitrary fixed measurable subset of positive measure.  Uniform
complete-eigenspace non-localization for the Dirichlet Laplacian on the
integrable polygons---rectangles, isosceles right triangles, equilateral
triangles, and hemi-equilateral triangles---is established in the companion
work~\cite{NguyenIntegrablePolygons2026}.  The estimates there are uniform
over the spectral level and over the complete eigenspace.  Related spectral
information for integrable polygons was obtained by M{\aa}rdby and
Rowlett~\cite{MardbyRowlett2025}.  The present problem differs in that the
geometry is fixed while the boundary condition is perturbed.

Reflection of the equilateral triangle across its sides connects the
Dirichlet and Neumann problems to Fourier analysis and observability on flat
tori.  Bourgain, Burq, and Zworski~\cite{BourgainBurqZworski2013} established
control results for Schr\"odinger operators on two-dimensional tori with
rough potentials, and Burq and Zworski~\cite{BurqZworski2019} obtained
observability for the free Schr\"odinger equation on rectangular
two-dimensional tori from arbitrary measurable sets of positive measure.
Semiclassical measures for the Schr\"odinger equation on tori and their
consequences for propagation and observability were studied by Anantharaman
and Maci\`a~\cite{AnantharamanMacia2014}.  For the equilateral triangle,
Alphonse and Lafontaine~\cite{AlphonseLafontaine2025} proved a stronger
Schr\"odinger observability theorem for Dirichlet and Neumann boundary
conditions.  Their dynamical result implies the stationary
complete-eigenspace estimate needed at the Neumann endpoint here.  We
nevertheless give a self-contained stationary proof in the form adapted to
the Robin perturbation argument.

The Fourier mechanism behind this reference estimate is closely related to
the work of Burq, Germain, Sorella, and
Zhu~\cite{BurqGermainSorellaZhu2026} on trace and observability inequalities
for Laplace eigenfunctions on flat tori.  In two dimensions, their analysis
uses bounded local clustering of lattice points on frequency circles.
After unfolding the equilateral triangle, the same two-dimensional lattice
geometry appears on the dual triangular lattice.  Although an arithmetic
eigenspace may contain an unbounded total number of frequencies, the relevant
local clusters have uniformly bounded size.  This permits a spectrum-uniform
estimate at the level of the complete eigenspace rather than at the level of
a selected Fourier mode.  For the fixed-Robin problem, a different
finite-dimensional mechanism is used.  Once the number of modal classes
contributing to one eigenspace is uniformly bounded, a multidimensional
Tur\'an--Nazarov inequality gives observation for the resulting plane-wave
space without requiring a lower bound on the separation of its frequencies.

There is also a substantial literature on observation from measurable sets
and quantitative unique continuation.  Apraiz, Escauriaza, Wang, and
Zhang~\cite{ApraizEscauriazaWangZhang2014} established observability
inequalities for the heat equation from sets of positive measure and proved
a Lebeau--Robbiano spectral inequality on suitable bounded domains.
Scale-free unique-continuation estimates for spectral projectors were
developed by Naki\'c, T\"aufer, Tautenhahn, and
Veseli\'c~\cite{NakicTauferTautenhahnVeselic2018}, while Egidi and
Veseli\'c~\cite{EgidiVeselic2020} proved scale-free
Logvinenko--Sereda and unique-continuation estimates on tori.
Davey~\cite{Davey2020} obtained quantitative unique continuation for
Schr\"odinger operators with singular lower-order terms, and Dicke, Rose,
Seelmann, and Tautenhahn~\cite{DickeRoseSeelmannTautenhahn2023} extended
scale-free estimates to spectral subspaces for classes of singular
potentials.  Related spectral inequalities for confining Schr\"odinger
operators and sparse sensor sets were obtained by Dicke, Seelmann, and
Veseli\'c~\cite{DickeSeelmannVeselic2024}.  Burq and
Moyano~\cite{BurqMoyano2023} developed propagation-of-smallness estimates
from sets of positive measure with applications to heat control, and
Kukavica and Li~\cite{KukavicaLi2025} recently proved measurable-set
observability for Gevrey functions and, in particular, for finite sums of
Laplace eigenfunctions on compact Gevrey manifolds.  These results apply in
far broader geometric and operator-theoretic settings than those considered
here.  Their quantitative constants, however, generally retain a dependence
on the spectral scale or on quantitative geometric parameters of the
observation set.  The estimates sought in \eqref{eq:intro_unl} instead concern
an arbitrary fixed measurable set of positive measure and require a lower
bound uniform over exact eigenspaces throughout the spectrum.  In the arithmetic geometries considered here, this stronger spectral
uniformity follows from uniform control of the Fourier complexity of the
exact eigenspaces.

The Robin part of the problem is closely related to the high-energy spectral
analysis of Rudnick, Wigman, and Yesha~\cite{RudnickWigmanYesha2021}, who
studied Robin--Neumann eigenvalue gaps on planar domains.  For rectangles,
Rudnick and Wigman~\cite{RudnickWigmanRectangles2021} analyzed the separable
Robin spectrum, its arithmetic multiplicities, and the dependence on the
aspect ratio.  For the equilateral triangle,
McCartin~\cite{McCartin2004Robin} obtained the complete Robin eigenstructure
in terms of symmetric and antisymmetric trigonometric modes governed by a
coupled system of transcendental secular equations.  Rudnick and
Wigman~\cite{RudnickWigman2022} subsequently studied Robin--Neumann
displacement, spectral gaps, multiplicities, and simplicity of the
desymmetrized spectrum for sufficiently small positive Robin parameter.
These works concern eigenvalue displacement, level structure, and
multiplicity.  The fixed-parameter question considered here is different:
for an arbitrary fixed positive Robin parameter, can the number of distinct
modal classes contributing to a single eigenspace grow without bound along
the spectrum?  We prove that it cannot on the equilateral triangle or on any
rectangle with rational squared aspect ratio.

More general Robin perturbation and asymptotic results include the
effective-Hamiltonian analysis of Pankrashkin and
Popoff~\cite{PankrashkinPopoff2016}, the large-coupling spectral analysis of
Belgacem et al.~\cite{BelgacemBelHadjAliBenAmorThabet2018}, the
Robin-to-Dirichlet asymptotics of Ognibene~\cite{Ognibene2025}, and the
quantitative compact-resolvent stability results of Bisterzo and
Siclari~\cite{BisterzoSiclari2025}.

The preceding spectral perturbation results do not by themselves give
uniform control of complete eigenspaces.  Near the Neumann endpoint, such
control requires comparison of the Robin and Neumann modal subspaces
uniformly in the modal indices, since the spectral level is unbounded as
$\sigma\downarrow0$.  For a fixed positive Robin parameter, the issue is
different: distinct modal classes may contribute to the same eigenvalue.
We prove that their number remains uniformly bounded over the spectrum, both
on the equilateral triangle and on rectangles with rational squared aspect
ratio.  Together with the frequency-uniform observation estimate for
finite plane-wave spaces, this yields complete-eigenspace observation
without requiring spectral simplicity.


\section{Mathematical setting and main results}
\label{sec:setting}

\subsection{The equilateral triangular cavity}
\label{subsec:triangle}

Let $T\subset\mathbb{R}^2$ be the open equilateral triangle of side
length $L>0$. Up to a rigid motion, we take its vertices to be
\begin{equation}
P_1=(0,0),\qquad
P_2=(L,0),\qquad
P_3=\left(\frac{L}{2},\frac{\sqrt{3}L}{2}\right),
\label{eq:triangle}
\end{equation}
and write $T$ for the interior of the triangle determined by
$P_1,P_2,P_3$; see Figure~\ref{fig:triangle_geometry}.

\begin{figure}[t]
\centering
\begin{tikzpicture}[scale=0.78]
    \coordinate (P1) at (0,0);
    \coordinate (P2) at (6,0);
    \coordinate (P3) at (3,{3*sqrt(3)});

    \draw[->] (-0.65,0) -- (6.85,0) node[right] {$x$};
    \draw[->] (0,-0.55) -- (0,5.85) node[above] {$y$};

    \fill[gray!10] (P1) -- (P2) -- (P3) -- cycle;
    \draw[thick] (P1) -- (P2) -- (P3) -- cycle;

    \fill (P1) circle (1.7pt);
    \fill (P2) circle (1.7pt);
    \fill (P3) circle (1.7pt);

    \node[below left] at (P1) {$P_1=(0,0)$};
    \node[below right] at (P2) {$P_2=(L,0)$};
    \node[above right] at (P3)
      {$P_3=\left(\frac{L}{2},\frac{\sqrt3 L}{2}\right)$};
    \node at (3,1.75) {$T$};
    \node[below] at (3,0) {$L$};
\end{tikzpicture}
\caption{The equilateral triangle $T$ and the coordinate convention used
throughout the paper. The shaded region denotes the open triangle; its
boundary is drawn only for geometric reference.}
\label{fig:triangle_geometry}
\end{figure}
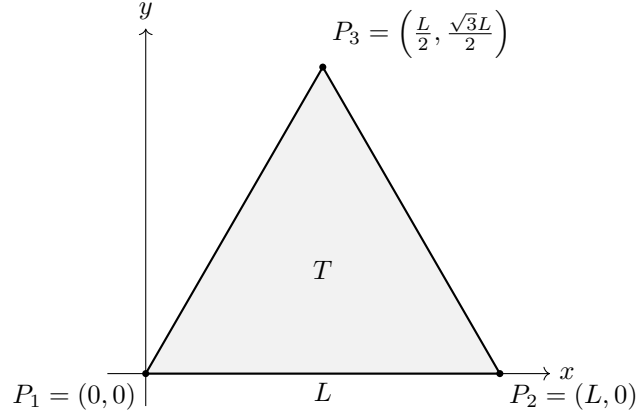

We consider the eigenvalue problem
\begin{equation}
-\Delta u=\lambda u
\qquad\text{in }T
\label{eq:eigenproblem}
\end{equation}
with Dirichlet, Neumann, or Robin boundary conditions:
\begin{align}
u&=0
&&\text{on }\partial T,
\label{eq:dirichlet_bc}
\\
\partial_\nu u&=0
&&\text{on }\partial T,
\\
\partial_\nu u+\sigma u&=0
&&\text{on }\partial T,
\qquad \sigma\geq0,
\label{eq:robin_bc}
\end{align}
where $\nu$ denotes the outward unit normal. The Robin problem at
$\sigma=0$ coincides with the Neumann problem.

We denote by
$-\Delta_D$, $-\Delta_N$, and $-\Delta_\sigma$
the corresponding nonnegative self-adjoint realizations in $L^2(T)$,
with the convention
\begin{equation}
-\Delta_0=-\Delta_N.
\label{eq:robin_neumann_convention}
\end{equation}
Equivalently, $-\Delta_D$ is associated with the quadratic form
$$
    q_D[u]
    =
    \int_T |\nabla u|^2\\,dx,
    \qquad
    \mathcal D(q_D)=H_0^1(T),
$$
whereas $-\Delta_\sigma$, $\sigma\geq0$, is associated with
\begin{equation}
q_\sigma[u]
=
\int_T |\nabla u|^2\,dx
+
\sigma\int_{\partial T}|u|^2\,ds,
\qquad
\mathcal D(q_\sigma)=H^1(T).
\label{eq:robin_quadratic_form}
\end{equation}
In particular, $q_0$ is the Neumann form.

Since $T$ is bounded and Lipschitz, these forms define self-adjoint
operators with compact resolvent. Consequently, each spectrum is
discrete, every eigenvalue has finite multiplicity, and the only
accumulation point of the spectrum is $+\infty$.
For $B\in\{D,N\}$ and
$\lambda\in\operatorname{Spec}(-\Delta_B)$, we write
\begin{equation}
E_\lambda^B
:=
\ker(-\Delta_B-\lambda).
\label{eq:eigenspace_DN}
\end{equation}
For $\sigma\geq0$ and
$\lambda\in\operatorname{Spec}(-\Delta_\sigma)$, we similarly set
\begin{equation}
E_{\lambda,\sigma}
:=
\ker(-\Delta_\sigma-\lambda).
\label{eq:eigenspace_R}
\end{equation}
Thus,
\begin{equation}
E_{\lambda,0}=E_\lambda^N.
\label{eq:neumann_robin_eigenspace}
\end{equation}
Throughout the paper, all eigenfunction estimates are understood at
the level of the complete eigenspace; no distinguished eigenbasis is
fixed.

\subsection{Observation ratios and uniform non-localization}
\label{subsec:unl_definition}

Let $V\subset T$ be measurable with
$|V|>0$.
For $0\neq u\in L^2(T)$, define the observation ratio
\begin{equation}
\mathcal{O}_V(u)
:=
\frac{\|u\|_{L^2(V)}^2}
     {\|u\|_{L^2(T)}^2}.
\label{eq:observation_ratio}
\end{equation}
For an eigenspace $E\subset L^2(T)$, set
\begin{equation}
\mathfrak{c}_V(E)
:=
\inf_{0\neq u\in E}\mathcal{O}_V(u).
\label{eq:eigenspace_observation}
\end{equation}
Thus, $\mathfrak{c}_V(E)$ measures the smallest fraction of
$L^2$-mass visible in $V$ among all normalized elements of $E$.
This distinction is important when $\dim E>1$. In particular,
an estimate of the form
\[
\mathcal{O}_V(\phi_j)\geq c
\]
for each member of a chosen eigenbasis does not by itself imply
\[
\mathfrak{c}_V(E)\geq c,
\]
since linear combinations of basis elements may exhibit additional
cancellation on $V$.

\begin{definition}[Uniform non-localization]
\label{def:unl}
A Laplacian $-\Delta_B$ is said to satisfy uniform
non-localization on $V$ if there exists a constant $c_V>0$ such that
\begin{equation}
\mathfrak{c}_V(E_\lambda^B)\geq c_V
\label{eq:unl_definition}
\end{equation}
for every eigenvalue $\lambda$ of $-\Delta_B$.
\end{definition}
Equivalently,
\begin{equation}
\inf_{\lambda\in\operatorname{Spec}(-\Delta_B)}
\inf_{0\neq u\in E_\lambda^B}
\frac{\|u\|_{L^2(V)}^2}
     {\|u\|_{L^2(T)}^2}
>0.
\label{eq:unl_equivalent}
\end{equation}
For Robin boundary conditions, we write
\begin{equation}
c_V(\sigma)
:=
\inf_{\lambda\in\operatorname{Spec}(-\Delta_\sigma)}
\mathfrak{c}_V(E_{\lambda,\sigma}).
\label{eq:robin_unl_constant}
\end{equation}
The central problem is therefore to determine when
\begin{equation}
c_V(\sigma)>0,
\label{eq:positive_robin_constant}
\end{equation}
and, more strongly, whether this lower bound remains positive
uniformly as $\sigma$ varies.

\subsection{Finite Fourier complexity and measurable-set observation}
\label{subsec:finite_complexity}

The fixed-Robin theory reduces the observability problem to a statement
about eigenspaces with uniformly bounded Fourier complexity.  Since this
reduction is independent of the particular geometry of the equilateral
triangle, we first formulate and prove the corresponding observation
principle in an abstract setting.

Let $\Omega\subset\mathbb R^d$ be bounded and measurable with
$|\Omega|>0$.  For a finite-dimensional subspace
$E\subset L^2(\Omega)$, define its \emph{plane-wave complexity} by
\begin{equation}
\mathfrak F_\Omega(E)
:=
\inf\left\{
N\in\mathbb N:
E\subset
\operatorname{span}\{
e^{i\xi_1\cdot x},\ldots,e^{i\xi_N\cdot x}
\}\big|_\Omega
\text{ for some }\xi_j\in\mathbb R^d
\right\},
\label{eq:plane_wave_complexity}
\end{equation}
with the convention $\mathfrak F_\Omega(E)=\infty$ if no such finite
representation exists.  The frequencies may depend on $E$; only their
number is measured.

\begin{lemma}[Finite Fourier-complexity observation]
\label{lem:finite_complexity_observation}
Let $\Omega\subset\mathbb R^d$ be bounded and measurable with
$|\Omega|>0$, and let $V\subset\Omega$ be measurable with $|V|>0$.
For every $N\in\mathbb N$, there exists a constant
\[
    \gamma_{\Omega,V,N}>0
\]
such that
\begin{equation}
    \|p\|_{L^2(V)}^2
    \geq
    \gamma_{\Omega,V,N}
    \|p\|_{L^2(\Omega)}^2,
\label{eq:finite_complexity_observation}
\end{equation}
for every exponential polynomial
\begin{equation}
    p(x)
    =
    \sum_{j=1}^{M}a_j e^{i\xi_j\cdot x},
    \qquad
    M\leq N,
\label{eq:finite_exp_sum}
\end{equation}
with $a_j\in\mathbb C$ and $\xi_j\in\mathbb R^d$ arbitrary.
In particular, $\gamma_{\Omega,V,N}$ is independent of the magnitudes,
mutual separations, and spectral variation of the frequencies.
\end{lemma}

\begin{proof}
Let $K\subset\mathbb R^d$ be a fixed compact convex body containing
$\Omega$.  We use the multidimensional Tur\'an--Nazarov inequality of
Brudnyi and Yomdin~\cite{BrudnyiYomdin2016}.  For every Borel set
$Z_0\subset K$ of positive measure, their estimate gives
\begin{equation}
    \|p\|_{L^\infty(K)}
    \leq
    \left(\frac{C_d|K|}{|Z_0|}\right)^{M-1}
    \|p\|_{L^\infty(Z_0)},
\label{eq:turan_nazarov_general}
\end{equation}
where $C_d>0$ depends only on the dimension.  Here, the exponents are
the linear functionals $f_j(x)=i\xi_j\cdot x$, so
\[
    \operatorname{Re}\bigl(f_j(x)-f_j(y)\bigr)=0
    \qquad (x,y\in K).
\]
Thus, the exponential growth factor in the general
Tur\'an--Nazarov estimate is equal to one.
Set
\[
    A:=\left(\frac{2}{|V|}\right)^{1/2}\|p\|_{L^2(V)},
    \qquad
    Z:=\{x\in V:|p(x)|\leq A\}.
\]
Chebyshev's inequality gives $|Z|\geq |V|/2$.  By inner regularity, there
is a compact set $Z_0\subset Z$ with
\[
    |Z_0|\geq\frac{|V|}{4}.
\]
Hence,
\[
    \|p\|_{L^\infty(Z_0)}
    \leq
    \left(\frac{2}{|V|}\right)^{1/2}\|p\|_{L^2(V)},
\]
and \eqref{eq:turan_nazarov_general} yields
\[
    \|p\|_{L^\infty(K)}
    \leq
    \left(\frac{4C_d|K|}{|V|}\right)^{M-1}
    \left(\frac{2}{|V|}\right)^{1/2}
    \|p\|_{L^2(V)}.
\]
Since $M\leq N$ and
$\|p\|_{L^2(\Omega)}
\leq|\Omega|^{1/2}\|p\|_{L^\infty(K)}$,
we obtain
\[
    \|p\|_{L^2(\Omega)}
    \leq
    C_{\Omega,V,N}\|p\|_{L^2(V)}
\]
for a finite constant depending only on the fixed sets
$\Omega,V$, the dimension, and $N$.  Taking
$\gamma_{\Omega,V,N}=C_{\Omega,V,N}^{-2}$ proves the claim.
\end{proof}

\begin{proposition}[Finite-complexity eigenspace observation principle]
\label{prop:finite_complexity_eigenspace_observation}
Let $\Omega\subset\mathbb R^d$ be bounded and measurable with
$|\Omega|>0$, and let $\{E_\lambda\}_{\lambda\in\Sigma}$ be any family
of finite-dimensional subspaces of $L^2(\Omega)$.  If
\begin{equation}
    \sup_{\lambda\in\Sigma}
    \mathfrak F_\Omega(E_\lambda)
    \leq N<\infty,
\label{eq:uniform_eigenspace_frequency_complexity}
\end{equation}
then for every measurable $V\subset\Omega$ with $|V|>0$,
\begin{equation}
    \|u\|_{L^2(V)}^2
    \geq
    \gamma_{\Omega,V,N}
    \|u\|_{L^2(\Omega)}^2,
\label{eq:abstract_eigenspace_observation}
\end{equation}
for every $\lambda\in\Sigma$ and every $u\in E_\lambda$.
\end{proposition}

\begin{proof}
For each $\lambda$, the definition
\eqref{eq:plane_wave_complexity} and
\eqref{eq:uniform_eigenspace_frequency_complexity} place every vector
of $E_\lambda$ in the span of at most $N$ plane waves.  Lemma
\ref{lem:finite_complexity_observation} then gives
\eqref{eq:abstract_eigenspace_observation} with the same constant for
all $\lambda$.
\end{proof}

\begin{corollary}[Modal complexity implies observation]
\label{cor:modal_complexity_observation}
Suppose that for every $\lambda\in\Sigma$,
\begin{equation}
    E_\lambda
    =
    \sum_{\mathfrak c\in\mathcal I(\lambda)}
    \mathcal R_{\mathfrak c},
\label{eq:abstract_modal_decomposition}
\end{equation}
where
\[
    \#\mathcal I(\lambda)\leq K
    \qquad\text{and}\qquad
    \mathfrak F_\Omega(\mathcal R_{\mathfrak c})
    \leq N_{\mathrm{mod}}
\]
uniformly in $\lambda$ and $\mathfrak c$.  Then
\begin{equation}
    \mathfrak F_\Omega(E_\lambda)
    \leq N_{\mathrm{mod}}K
\label{eq:abstract_modal_to_frequency_complexity}
\end{equation}
for every $\lambda$, and consequently
\begin{equation}
    \|u\|_{L^2(V)}^2
    \geq
    \gamma_{\Omega,V,N_{\mathrm{mod}}K}
    \|u\|_{L^2(\Omega)}^2
\label{eq:abstract_modal_observation}
\end{equation}
for every measurable $V\subset\Omega$ of positive measure,
every $\lambda$, and every $u\in E_\lambda$.
\end{corollary}

\begin{proof}
Taking the union of the frequency sets needed for the modal spaces in
\eqref{eq:abstract_modal_decomposition} gives
\eqref{eq:abstract_modal_to_frequency_complexity}; coincident
frequencies can only reduce the count.  Proposition
\ref{prop:finite_complexity_eigenspace_observation} then gives
\eqref{eq:abstract_modal_observation}.
\end{proof}

\begin{remark}
\label{rem:finite_complexity_no_frequency_separation}
The mechanism above requires no lower bound on
$|\xi_j-\xi_k|$ and no upper bound on $|\xi_j|$.  The only uniform
quantity is the number of frequencies.  This is precisely why a bound
on spectral coincidence complexity can be converted into an
observation constant that is uniform over arbitrarily high spectral
levels.  No openness or thickness assumption on $V$ is used; positive
Lebesgue measure is sufficient.
\end{remark}

\subsection{Dirichlet and Neumann boundary conditions}
\label{subsec:DN_results}

Reflection across the sides of the equilateral triangle unfolds the Dirichlet and Neumann problems to the flat torus associated with the triangular reflection lattice. Although each individual reflection orbit of a wave vector is finite, a complete eigenspace may contain arbitrarily many such arithmetic orbits. The uniform observation estimate therefore requires control of the full spectral shell rather than a uniform bound on its number of plane-wave frequencies. This is achieved below by a two-dimensional lattice-clustering argument.

\begin{theorem}[Dirichlet uniform non-localization]
\label{thm:dirichlet_unl}
Let $V\subset T$ be measurable with $|V|>0$. Then, there exists a
constant
$c_{V,D}>0$
such that
\begin{equation}
\|u\|_{L^2(V)}^2
\geq
c_{V,D}\|u\|_{L^2(T)}^2
\label{eq:dirichlet_unl}
\end{equation}
for every Dirichlet eigenvalue $\lambda$ and every
$u\in E_\lambda^D$.

In particular,
\begin{equation}
\inf_{\lambda\in\operatorname{Spec}(-\Delta_D)}
\mathfrak{c}_V(E_\lambda^D)>0.
\label{eq:dirichlet_unl_inf}
\end{equation}
\end{theorem}

\begin{theorem}[Neumann uniform non-localization]
\label{thm:neumann_unl}
Let $V\subset T$ be measurable with $|V|>0$. Then, there exists
$c_{V,N}>0$
such that
\begin{equation}
\|u\|_{L^2(V)}^2
\geq
c_{V,N}\|u\|_{L^2(T)}^2
\label{eq:neumann_unl}
\end{equation}
for every Neumann eigenvalue $\lambda$ and every
$u\in E_\lambda^N$.
Hence,
\begin{equation}
\inf_{\lambda\in\operatorname{Spec}(-\Delta_N)}
\mathfrak{c}_V(E_\lambda^N)>0.
\label{eq:neumann_unl_inf}
\end{equation}
\end{theorem}
The constants in
Theorems \ref{thm:dirichlet_unl} and \ref{thm:neumann_unl}
are allowed to depend on $V$, but not on the eigenvalue or on the
choice of vector within a multiple eigenspace.

\subsection{Stability under small Robin perturbations}
\label{subsec:small_robin}

We next consider the Robin boundary condition, which is
$
\partial_\nu u+\sigma u=0.
$
The Neumann problem is recovered at $\sigma=0$. Our main
stability result asserts that the non-localization estimate persists
uniformly for sufficiently small Robin perturbations.

\begin{theorem}[Uniform small-Robin stability]
\label{thm:small_robin}
Let $V\subset T$ be measurable with $|V|>0$. There exist constants
$\sigma_0>0$ and 
$c_{V,R}>0$
such that
\begin{equation}
\|u\|_{L^2(V)}^2
\geq
c_{V,R}\|u\|_{L^2(T)}^2,
\label{eq:small_robin_bound}
\end{equation}
for every
$0\leq\sigma\leq\sigma_0$, 
$\lambda\in\operatorname{Spec}(-\Delta_\sigma)$,
and 
$u\in E_{\lambda,\sigma}$.
Equivalently,
\begin{equation}
\inf_{0\leq\sigma\leq\sigma_0}
\inf_{\lambda\in\operatorname{Spec}(-\Delta_\sigma)}
\inf_{0\neq u\in E_{\lambda,\sigma}}
\frac{\|u\|_{L^2(V)}^2}
     {\|u\|_{L^2(T)}^2}
>0.
\label{eq:small_robin_uniform}
\end{equation}
The threshold $\sigma_0$ may depend on the observation set $V$
(and on the fixed triangle $T$), but is independent of the spectral level.
\end{theorem}

The essential point in the inequality (\ref{eq:small_robin_uniform}) is the
simultaneous uniformity in both $\lambda$ and $\sigma$. Continuity of
a fixed eigenbranch at $\sigma=0$ is not sufficient for this
conclusion.

\subsection{Arbitrary Robin parameters}
\label{subsec:arbitrary_robin}

We next fix $\sigma>0$. Distinct modal classes may then contribute to
the same Robin eigenvalue. It is therefore necessary to bound, uniformly
over the spectrum, the number of modal classes contained in a single
eigenspace.

For
   $ \lambda\in\operatorname{Spec}(-\Delta_\sigma)$,
let
$
    \mathcal{I}_\sigma(\lambda)
$
denote the set of distinct Robin modal index classes that contribute to
$E_{\lambda,\sigma}$, where modal representatives related by the
systematic symmetries described in
\Cref{sec:robin_structure} are identified. Define
\begin{equation}
\kappa_\sigma(\lambda)
:=
\#\mathcal{I}_\sigma(\lambda)
\label{eq:kappa}
\end{equation}
and
\begin{equation}
K_\sigma
:=
\sup_{\lambda\in\operatorname{Spec}(-\Delta_\sigma)}
\kappa_\sigma(\lambda).
\label{eq:Ksigma}
\end{equation}
The quantity $K_\sigma$ concerns the global coincidence structure of
the Robin spectrum. It is not simply the multiplicity of $\lambda$:
$\kappa_\sigma(\lambda)$ counts distinct modal index classes after the
systematic symmetry multiplicities within each class have been factored
out.

\begin{theorem}[Global bounded coincidence complexity]
\label{thm:global_robin_complexity}
Fix $\sigma>0$. There exist constants $B_\sigma,N_\sigma<\infty$,
independent of the spectral level, such that every Robin eigenvalue
receives contributions from at most $B_\sigma$ desymmetrized modal
classes over each Neumann shell and from at most $N_\sigma$ Neumann
shells. In particular,
\begin{equation}
    K_\sigma\le B_\sigma N_\sigma<\infty.
\label{eq:global_Ksigma_finite}
\end{equation}
\end{theorem}

For later reference, define the fixed-parameter observation constant
by
\begin{equation}
c_V(\sigma)
:=
\inf_{\lambda\in\operatorname{Spec}(-\Delta_\sigma)}
\inf_{0\neq u\in E_{\lambda,\sigma}}
\frac{\|u\|_{L^2(V)}^2}
{\|u\|_{L^2(T)}^2}.
\label{eq:fixed_robin_constant}
\end{equation}

\begin{theorem}[Uniform non-localization for every fixed Robin parameter]
\label{thm:fixed_robin}
Fix $\sigma>0$. Then, for every measurable set $V\subset T$ with
$|V|>0$, there exists a constant $c_{V,\sigma}>0$ such that
\begin{equation}
\|u\|_{L^2(V)}^2
\geq
c_{V,\sigma}\,\|u\|_{L^2(T)}^2
\label{eq:fixed_robin_unl}
\end{equation}
for every
$\lambda\in\operatorname{Spec}(-\Delta_\sigma)$ and every
$0\neq u\in E_{\lambda,\sigma}$. Equivalently,
\begin{equation}
c_V(\sigma)>0.
\label{eq:fixed_robin_positive}
\end{equation}
\end{theorem}

The constant in Theorem \ref{thm:fixed_robin} may depend on $\sigma$.
In particular, no uniformity with respect to the Robin parameter is
asserted for arbitrary $\sigma>0$. This distinguishes
Theorem \ref{thm:fixed_robin} from the small-Robin result, where the
observation constant is uniform for
$0\leq\sigma\leq\sigma_0$.

\subsection{A second geometry: the Robin square}
\label{subsec:square_main_results}

Let $S=(0,1)^2$ and impose the Robin condition
$\partial_\nu u+\sigma u=0$ on $\partial S$.  For a Robin eigenvalue
$\Lambda$ of $S$, let $\kappa_\sigma^{\square}(\Lambda)$ denote the
number of distinct unordered modal classes $\{n,m\}$ contributing to its
eigenspace, and set
\begin{equation}
 K_\sigma^{\square}
 :=
 \sup_{\Lambda\in\operatorname{Spec}(-\Delta_{\sigma,S})}
 \kappa_\sigma^{\square}(\Lambda).
\label{eq:square_Ksigma_def}
\end{equation}

\begin{theorem}[Fixed-parameter Robin non-localization on the square]
\label{thm:square_fixed_robin}
For every fixed $\sigma>0$,
\begin{equation}
 K_\sigma^{\square}<\infty.
\label{eq:square_global_complexity_main}
\end{equation}
Consequently, for every measurable $V\subset S$ with $|V|>0$, there exists
$c_{V,\sigma}^{\square}>0$ such that
\begin{equation}
 \|u\|_{L^2(V)}^2
 \geq
 c_{V,\sigma}^{\square}\|u\|_{L^2(S)}^2
\label{eq:square_observation_main}
\end{equation}
for every Robin eigenvalue of $S$ and every $u$ in the corresponding
eigenspace.
\end{theorem}

The square exhibits the same arithmetic shell structure as the rational
rectangle family, with the additional exchange symmetry
$n\leftrightarrow m$.
For a rectangle
\[
    R_{a,b}=(0,a)\times(0,b),
    \qquad
    \frac{a^2}{b^2}\in\mathbb Q,
\]
let $\kappa_{\sigma,a,b}^{\mathrm{rect}}(\Lambda)$ denote the number of
distinct modal classes contributing to the eigenspace associated with
$\Lambda$, and set
\begin{equation}
    K_{\sigma,a,b}^{\mathrm{rect}}
    :=
    \sup_{\Lambda\in
    \operatorname{Spec}(-\Delta_{\sigma,R_{a,b}})}
    \kappa_{\sigma,a,b}^{\mathrm{rect}}(\Lambda).
\label{eq:rectangle_Ksigma_def}
\end{equation}

\begin{theorem}[Fixed-parameter Robin non-localization on rational rectangles]
\label{thm:rational_rectangle_fixed_robin}
Let
\[
    R_{a,b}=(0,a)\times(0,b),
    \qquad
    \frac{a^2}{b^2}\in\mathbb Q.
\]
For every fixed $\sigma>0$,
\begin{equation}
    K_{\sigma,a,b}^{\mathrm{rect}}<\infty.
\label{eq:rectangle_global_complexity_main}
\end{equation}
Consequently, for every measurable set $V\subset R_{a,b}$ with $|V|>0$,
there exists $c_{V,\sigma,a,b}>0$ such that
\begin{equation}
    \|u\|_{L^2(V)}^2
    \geq
    c_{V,\sigma,a,b}
    \|u\|_{L^2(R_{a,b})}^2
\label{eq:rational_rectangle_observation_main}
\end{equation}
for every Robin eigenvalue of $R_{a,b}$ and every $u$ in the corresponding
eigenspace.  The constants are not asserted to be uniform in $\sigma$ or
in the aspect ratio.
\end{theorem}

Theorem~\ref{thm:rational_rectangle_fixed_robin} includes the square,
already covered by Theorem~\ref{thm:square_fixed_robin}.  For a rational
squared aspect ratio, the Neumann spectrum is organized by weighted
quadratic shells.  In the nonsquare case, the proof uses a discrete
splitting argument on these shells near the unique stationary point of
the leading Robin correction.

\begin{remark}
The rationality assumption on the squared aspect ratio is used to organize
the Neumann spectrum into exact weighted quadratic shells, which serve as
the reference sets for the splitting argument.  If
$a^2/b^2\notin\mathbb Q$, distinct modal pairs have distinct Neumann
eigenvalues, but no corresponding arithmetic shell decomposition is
available and the gaps between different reference levels need not admit
a uniform lower bound.  The Robin coincidence problem is then governed by
Diophantine properties of the aspect ratio rather than by splitting within
a fixed arithmetic shell.  We do not consider this case here.
\end{remark}

\subsection{Structure of the proof}
\label{subsec:proof_structure}

Arguments for different boundary conditions require different
treatments.

For the Dirichlet and Neumann problems, reflection unfolds each
eigenspace onto a fixed flat torus.  Although every reflection orbit is
finite, a spectral shell may contain arbitrarily many arithmetic orbits,
so that no uniform bound on the number of frequencies in a complete
eigenspace is available. Instead, the argument uses a Jarn\'ik-type
lattice lemma to decompose each shell into uniformly small frequency
clusters.  A two-dimensional Fourier-cluster estimate then gives the
required observation inequality.

For Robin boundary conditions with a small parameter, the Neumann estimate
is stable under perturbation.  McCartin's secular equations give
$O(\sqrt{\sigma})$ control of the Robin modal spaces, uniformly in the
modal indices, while the small-parameter spectral separation prevents
distinct desymmetrized classes from coinciding.  The observation estimate
is therefore stable under the resulting perturbation of the modal
subspaces, with a constant uniform for
$0\leq \sigma\leq \sigma_0$.

The situation is different when $\sigma>0$ is fixed. Here, distinct modal
classes may contribute to the same Robin eigenspace, and the problem is to
bound their number uniformly over the spectrum.  We first establish a
general shell-comparison principle which reduces this question to
high-frequency ordering within a reference shell, together with a uniform
bound on the displacement from the reference spectrum.

For the equilateral triangle, McCartin's secular equations admit, with
$\sigma$ fixed, a high-frequency expansion whose leading nonconstant
correction on a Neumann shell is a positive multiple of
\[
    F(m,n)
    =
    \frac{1}{m^2}
    +\frac{1}{n^2}
    +\frac{1}{(m+n)^2}.
\]
The ordering induced by this term, combined with separate estimates in
the bulk and near the edges of the shell, gives an eventual strict ordering
of the Robin eigenvalues within each Neumann shell. Thus, only a uniformly
bounded set of low transverse indices can contribute additional
coincidences.  The uniform displacement of the Robin spectrum from the
Neumann spectrum also bounds the number of Neumann shells that can
contribute to a single Robin eigenvalue.  It follows that
$K_\sigma<\infty$.

This spectral bound has a direct consequence for observation.  Each modal
class is contained in the span of a uniformly bounded number of plane
waves, and therefore $K_\sigma<\infty$ gives a uniform bound on the Fourier
complexity of every complete Robin eigenspace.  The multidimensional
Tur\'an--Nazarov inequality then yields the observation estimate on every
measurable set of positive measure.  No separation condition on the
frequencies is required.

The same principle applies to rectangles, although the spectral splitting
takes a different form.  On the square, separation of variables reduces
the problem to the one-dimensional Robin frequency sequence.  Its
fixed-$\sigma$ asymptotic expansion splits the sum-of-two-squares shells.
Together with the bounded exceptional part of the spectrum and the
uniform Robin--Neumann displacement, this gives
$K_\sigma^{\square}<\infty$.  Since each unordered modal class has
uniformly bounded plane-wave complexity, the measurable-set observation
estimate follows as above.

For a rectangle satisfying $a^2/b^2\in\mathbb Q$, write
$a^2=p\ell^2$ and $b^2=q\ell^2$.  The corresponding Neumann shells are
\[
    qn^2+pm^2=J.
\]
The leading Robin correction on such a shell is described by a strongly
convex angular profile with a unique interior minimum.  On either side
of this minimum, the profile is strictly monotone.  The spacing of the
admissible lattice points on the shell, together with separate estimates
near the two coordinate axes, gives at most two high-frequency
coincidences on each shell.  After the finitely many exceptional indices
are included, the number of modal pairs contributing to any Robin
eigenvalue is therefore uniformly bounded:
\[
    K_{\sigma,a,b}^{\mathrm{rect}}<\infty.
\]
Each separated product mode is a linear combination of at most four plane
waves, so the same Fourier-complexity argument gives an observation on every
measurable subset of a positive measure.

Thus, spectral multiplicity plays different roles in the three cases.
For Dirichlet and Neumann boundary conditions, the possibly unbounded
arithmetic multiplicity is treated directly by Fourier clustering.  Near
the Neumann endpoint, the observation estimate persists by perturbative
stability of the modal subspaces.  For a fixed positive Robin parameter,
spectral simplicity is replaced by a uniform bound on the number of modal
classes contained in a single eigenspace.  The triangle and rational
rectangles obtain this bound by different shell-splitting arguments,
showing that the observation principle depends on bounded spectral
complexity rather than on the particular separation structure of either
geometry.


\section{Dirichlet and Neumann non-localization}
\label{sec:dirichlet_neumann}

The proofs of Theorems~\ref{thm:dirichlet_unl} and
\ref{thm:neumann_unl} are based on the reflection structure of the
equilateral triangle.  Unfolding identifies each Dirichlet or Neumann
eigenfunction with a Laplace eigenfunction on the flat torus associated
with the triangular lattice.  At a fixed eigenvalue, the corresponding
Fourier frequencies lie on a single Euclidean circle.

The number of points of the hexagonal dual lattice on such a circle is
not uniformly bounded with respect to the eigenvalue.  Consequently,
the finite-exponential Tur\'an--Nazarov estimate used later for the
fixed-Robin problem does not give a uniform observation bound in this
setting, since its constant depends on the number of frequencies.
We instead work directly with complete toral eigenspaces.  A
Jarn\'ik-type decomposition of the lattice points into uniformly small
frequency clusters, followed by a two-dimensional Fourier-cluster
estimate, yields the required observation inequality.

\subsection{The triangular lattice and the unfolded torus}
\label{subsec:unfolding}

Let
\begin{equation}
v_1
=
\left(\frac{3L}{2},\frac{\sqrt{3}L}{2}\right),
\qquad
v_2
=
\left(0,\sqrt{3}L\right),
\label{eq:lattice_generators}
\end{equation}
and define the translation lattice of the affine reflection tiling by
\begin{equation}
\Lambda
=
\mathbb{Z}v_1+\mathbb{Z}v_2.
\label{eq:triangular_lattice}
\end{equation}
The vectors $v_1$ and $v_2$ are translations in the affine reflection
group generated by the three side reflections.  Its fundamental
parallelogram has area
\begin{equation}
|\det(v_1,v_2)|
=
\frac{3\sqrt{3}}{2}L^2
=
6|T|,
\label{eq:reflection_lattice_covolume}
\end{equation}
so it consists, up to boundary sets of measure zero, of six reflected
copies of the triangular alcove $T$.  We denote by
\begin{equation}
\mathbb{T}_{\Lambda}^{2}
=
\mathbb{R}^{2}/\Lambda
\label{eq:flat_torus}
\end{equation}
the corresponding flat torus.
The dual lattice is
\begin{equation}
\Lambda^{*}
=
\left\{
\xi\in\mathbb{R}^{2}:
\xi\cdot a\in2\pi\mathbb{Z}
\text{ for every }a\in\Lambda
\right\}.
\label{eq:dual_lattice}
\end{equation}
A convenient dual basis, characterized by
$b_j\cdot v_k=2\pi\delta_{jk}$, is
\begin{equation}
b_1
=
\left(\frac{4\pi}{3L},0\right),
\qquad
b_2
=
\left(-\frac{2\pi}{3L},\frac{2\sqrt{3}\pi}{3L}\right),
\label{eq:dual_basis}
\end{equation}
so that
\[
\Lambda^{*}
=
\mathbb{Z}b_1+\mathbb{Z}b_2.
\]
For $\xi\in\Lambda^{*}$, the function
\[
e_{\xi}(x)=e^{i\xi\cdot x}
\]
is a toral eigenfunction satisfying
\begin{equation}
-\Delta e_{\xi}
=
|\xi|^{2}e_{\xi}.
\label{eq:torus_plane_wave}
\end{equation}
Consequently, the eigenspace of the toral Laplacian corresponding to
$\lambda\geq0$ is
\begin{equation}
\mathcal{H}_{\lambda}
=
\operatorname{span}
\left\{
e^{i\xi\cdot x}:
\xi\in\Lambda^{*},
\ |\xi|^{2}=\lambda
\right\}.
\label{eq:torus_eigenspace}
\end{equation}
In lattice coordinates, if
\[
\xi=mb_1+nb_2,
\qquad m,n\in\mathbb{Z},
\]
then
\begin{equation}
|\xi|^{2}
=
\frac{16\pi^{2}}{9L^{2}}
\left(m^{2}-mn+n^{2}\right).
\label{eq:hexagonal_quadratic_form}
\end{equation}
Thus, the multiplicity of a toral eigenvalue is governed by the number
of representations of an integer by the quadratic form
\[
Q(m,n)=m^{2}-mn+n^{2}.
\]

\begin{remark}
\label{rem:unbounded_multiplicity}
The cardinality of a level set of $Q$ is not uniformly bounded with
respect to the spectral level.  Thus a complete Dirichlet or Neumann
eigenspace need not have uniformly bounded Fourier complexity, even
though each individual reflection orbit is finite.  In particular, the
finite-exponential observation principle used for fixed positive Robin
parameter does not apply uniformly in this setting, since its constant
depends on the number of frequencies.

The argument below instead treats the complete toral eigenspace
$\mathcal{H}_{\lambda}$ directly.  The lattice points on each spectral
circle are decomposed into uniformly small frequency clusters, and the
resulting two-dimensional Fourier-cluster estimate is independent of
the total number of points on the circle.
\end{remark}

\subsection{Reflection extension}
\label{subsec:reflection_extension}

Let $r_j$, $j=1,2,3$, denote reflection across the three sides of
$T$. Successive reflections tile the plane by congruent copies of
$T$. A function on $T$ satisfying homogeneous Dirichlet or Neumann
boundary conditions can therefore be extended across each side by
odd or even reflection, respectively.

\begin{lemma}[Reflection extension]
\label{lem:reflection_extension}
Let $B\in\{D,N\}$ and let
\[
u\in E_{\lambda}^{B}.
\]
Then, there exists a nonzero function
\[
U\in\mathcal{H}_{\lambda}
\]
such that
\begin{equation}
U|_{T}=u.
\label{eq:extension_restriction}
\end{equation}
Moreover, $U$ is odd with respect to reflection across each side of
$T$ when $B=D$, and even with respect to reflection across each side
when $B=N$.
\end{lemma}
\begin{proof}
Let $W_{\mathrm{aff}}$ denote the affine reflection group generated by
the reflections $r_1,r_2,r_3$ in the three sides of $T$. The images
$w(T)$, $w\in W_{\mathrm{aff}}$, form the equilateral triangular tiling
of $\mathbb R^2$, and the action is simply transitive on the interiors
of the triangular alcoves.
Define
\[
    \chi_D(w):=\det(Dw)\in\{-1,1\},
    \qquad
    \chi_N(w):=1,
    \qquad
    w\in W_{\mathrm{aff}},
\]
where $Dw$ denotes the linear part of the affine isometry $w$. Since
each generating reflection has determinant $-1$, the character
$\chi_D$ records precisely the parity of the number of reflections and
is independent of the choice of a reflection word representing $w$.
For $B\in\{D,N\}$ define, on the interior of every reflected alcove,
\begin{equation}
    U(wx):=\chi_B(w)u(x),
    \qquad
    x\in T,
    \quad
    w\in W_{\mathrm{aff}}.
\label{eq:group_reflection_extension}
\end{equation}
Because distinct alcove interiors are disjoint, this defines $U$
unambiguously almost everywhere in $\mathbb R^2$.

The extensions across adjacent sides are compatible with the boundary
conditions. In the Dirichlet case, $u\in H_0^1(T)$ has vanishing trace
on each side, so the odd reflection belongs locally to $H^1$ across
the reflecting line. In the Neumann case, the even reflection belongs
locally to $H^1$, and the weak Neumann condition gives the corresponding
flux compatibility. Consequently,
\[
    U\in H^1_{\mathrm{loc}}(\mathbb R^2).
\]
Let $\varphi\in C_c^\infty(\mathbb R^2)$. Its support meets only
finitely many reflected triangles. Applying the weak eigenvalue
equation on each such triangle and summing, the contributions from
every common edge cancel pairwise, by odd reflection in the Dirichlet
case and by even reflection together with the Neumann condition in the
Neumann case. The vertices contribute no boundary term. Hence,
\[
    \int_{\mathbb R^2}\nabla U\cdot\nabla\overline\varphi\\,dx
    =
    \lambda\int_{\mathbb R^2}U\overline\varphi\\,dx.
\]
Thus, $-\Delta U=\lambda U$ in the sense of distributions on
$\mathbb R^2$. The standard interior elliptic regularity then implies that
$U$ is smooth in $\mathbb R^2$.

The translation subgroup of $W_{\mathrm{aff}}$ is the lattice
$\Lambda=\mathbb Zv_1+\mathbb Zv_2$ defined in
Eq. \ref{eq:triangular_lattice}. If $\tau_a(x)=x+a$ with $a\in\Lambda$,
then $D\tau_a=I$, and therefore
\[
    \chi_D(\tau_a)=1=\chi_N(\tau_a).
\]
Using Eq. (\ref{eq:group_reflection_extension}), we obtain
\[
    U(x+a)=U(x),
    \qquad a\in\Lambda.
\]
Hence, $U$ descends to an eigenfunction on the flat torus
$\mathbb T_\Lambda^2$.
Finally, Fourier expansion on $\mathbb T_\Lambda^2$ gives
\[
    U(x)
    =
    \sum_{\substack{\xi\in\Lambda^*\\|\xi|^2=\lambda}}
        \widehat U(\xi)e^{i\xi\cdot x}.
\]
Therefore, $U\in\mathcal H_\lambda$. Since $U|_T=u$ and $u\neq0$, the
extension is nonzero.
\end{proof}

\subsection{A toral observation estimate}
\label{subsec:toral_observation}

The analytic input needed below is an observation estimate for
individual eigenspaces of a two-dimensional flat torus.

\begin{proposition}[Toral eigenspace observation estimate]
\label{prop:torus_observation}
Let $\mathbb{T}_{\Lambda}^{2}$ be the flat torus defined in
Eq. (\ref{eq:flat_torus}), and let $\omega\subset\mathbb{T}_{\Lambda}^{2}$
be a nonempty open set. Then, there exists a constant
\[
C_{\omega}>0
\]
such that, for every toral eigenvalue $\lambda$ and every
$U\in\mathcal{H}_{\lambda}$,
\begin{equation}
\|U\|_{L^{2}(\mathbb{T}_{\Lambda}^{2})}^{2}
\leq
C_{\omega}
\|U\|_{L^{2}(\omega)}^{2}.
\label{eq:torus_observation}
\end{equation}
The constant $C_{\omega}$ is independent of $\lambda$ and of the
multiplicity of $\mathcal{H}_{\lambda}$.
\end{proposition}

\begin{remark}
\label{rem:torus_observation_reference}
The essential feature of
Proposition~\ref{prop:torus_observation} is that its observation constant
is independent of the eigenvalue.  This is stronger than a spectral
inequality whose constant grows with the spectral parameter, and it is
this uniformity that permits the passage to uniform non-localization on
the triangle.
\end{remark}

We first apply Proposition~\ref{prop:torus_observation} to open
observation sets in the triangle.

\begin{lemma}[Observation after unfolding]
\label{lem:observation_after_unfolding}
Let $V\subset T$ be a nonempty open set. There exists a constant
$c_V>0$ such that, for $B\in\{D,N\}$, every eigenvalue $\lambda$ of
$-\Delta_B$, and every $u\in E_{\lambda}^{B}$,
\begin{equation}
\|u\|_{L^{2}(V)}^{2}
\geq
c_V\|u\|_{L^{2}(T)}^{2}.
\label{eq:triangle_open_observation}
\end{equation}
\end{lemma}

\begin{proof}
Let $u\in E_{\lambda}^{B}$ and let $U\in\mathcal{H}_{\lambda}$ be
the extension furnished by \cref{lem:reflection_extension}. Let
$\mathcal{F}$ be a fundamental domain for the lattice $\Lambda$.
By Eq. (\ref{eq:reflection_lattice_covolume}), the domain $\mathcal{F}$ is
tiled, up to sets of measure zero, by exactly six reflected copies of
$T$. Since reflection changes at most the sign of the extension,
\begin{equation}
\|U\|_{L^{2}(\mathcal{F})}^{2}
=
6\|u\|_{L^{2}(T)}^{2}.
\label{eq:global_norm_relation}
\end{equation}
Likewise, let $\omega\subset\mathcal{F}$ be the union of the
corresponding reflected copies of $V$. Then,
\begin{equation}
\|U\|_{L^{2}(\omega)}^{2}
=
6\|u\|_{L^{2}(V)}^{2}.
\label{eq:local_norm_relation}
\end{equation}
Since $V$ is nonempty and open, so is $\omega$ on the torus.
Applying Proposition \ref{prop:torus_observation} gives
\[
6\|u\|_{L^{2}(T)}^{2}
=
\|U\|_{L^{2}(\mathcal{F})}^{2}
\leq
C_{\omega}
\|U\|_{L^{2}(\omega)}^{2}
=
6C_{\omega}\|u\|_{L^{2}(V)}^{2}.
\]
Cancelling the factor $6$ yields
\[
\|u\|_{L^{2}(V)}^{2}
\geq
C_{\omega}^{-1}\|u\|_{L^{2}(T)}^{2}.
\]
Thus, the inequality (\ref{eq:triangle_open_observation}) holds with
$c_V=C_{\omega}^{-1}$.
\end{proof}

\subsection{Dirichlet and Neumann non-localization}
\label{subsec:DN_proofs}

We can now prove the first two main results, for open observation
sets, directly from the unfolding construction.

\begin{proof}[Proof of Theorem \ref{thm:dirichlet_unl} for open $V$]
Let $V\subset T$ be nonempty and open. By
Lemma \ref{lem:observation_after_unfolding}, there exists $c_V>0$,
independent of $\lambda$, such that
\[
\|u\|_{L^{2}(V)}^{2}
\geq
c_V\|u\|_{L^{2}(T)}^{2}
\]
for every $u\in E_{\lambda}^{D}$ and every Dirichlet eigenvalue
$\lambda$. Hence,
\[
\inf_{\lambda\in\operatorname{Spec}(-\Delta_D)}
\inf_{0\neq u\in E_{\lambda}^{D}}
\frac{\|u\|_{L^{2}(V)}^{2}}
     {\|u\|_{L^{2}(T)}^{2}}
\geq c_V>0.
\]
\end{proof}

\begin{proof}[Proof of Theorem \ref{thm:neumann_unl} for open $V$]
The argument is identical, using the even reflection extension in
Lemma \ref{lem:reflection_extension}. Thus, there exists $c_V>0$ such that
\[
\|u\|_{L^{2}(V)}^{2}
\geq
c_V\|u\|_{L^{2}(T)}^{2}
\]
for every $u\in E_{\lambda}^{N}$ and every Neumann eigenvalue
$\lambda$.
\end{proof}

\subsection{Positive-measure observation sets}
\label{subsec:positive_measure_sets}

Theorems~\ref{thm:dirichlet_unl} and
\ref{thm:neumann_unl} concern arbitrary measurable sets of positive
measure.  An estimate for open observation sets alone does not imply
this statement, since a measurable set of positive measure need not
contain a nonempty open subset.  We therefore establish the corresponding
toral observation estimate directly for measurable sets.

The proof uses the two-dimensional Fourier-cluster argument of
Burq, Germain, Sorella, and Zhu~\cite{BurqGermainSorellaZhu2026}.
The lattice-geometric input for the triangular torus is established
below.

\begin{lemma}[Jarn\'ik-type clustering on a fixed planar lattice]
\label{lem:lattice_jarnik}
Let $\Gamma\subset\mathbb R^2$ be a fixed full-rank lattice and let
\[
    S_\rho(\Gamma)
    :=
    \{\xi\in\Gamma:|\xi|=\rho\}.
\]
There exists a constant $c_\Gamma>0$ such that, for every $\rho>0$,
the finite set $S_\rho(\Gamma)$ admits a partition
\[
    S_\rho(\Gamma)=\bigcup_\alpha \Omega_\alpha
\]
with
\begin{equation}
    \#\Omega_\alpha\leq 2
\label{eq:lattice_cluster_size}
\end{equation}
and, whenever $\alpha\neq\beta$,
\begin{equation}
    \operatorname{dist}(\Omega_\alpha,\Omega_\beta)
    \geq
    c_\Gamma\rho^{1/3}.
\label{eq:lattice_cluster_separation}
\end{equation}
\end{lemma}

\begin{proof}
Let $A_\Gamma>0$ denote the covolume of $\Gamma$. If
$p,q,r\in\Gamma$ are distinct and noncollinear, then
\[
    |\triangle pqr|
    =
    \frac12|\det(q-p,r-p)|
    \geq
    \frac{A_\Gamma}{2},
\]
because the determinant of two lattice vectors is an integer multiple
of the determinant of a lattice basis.
Suppose now that $p,q,r\in S_\rho(\Gamma)$ are distinct and satisfy
\[
    |p-q|<t,
    \qquad
    |q-r|<t.
\]
Three distinct points on a circle are noncollinear, and the
circumradius formula gives
\[
    |\triangle pqr|
    =
    \frac{|p-q|\,|q-r|\,|r-p|}{4\rho}
    <
    \frac{t^3}{2\rho},
\]
since $|r-p|\leq |r-q|+|q-p|<2t$. Set
\[
    c_\Gamma:=\left(\frac{A_\Gamma}{2}\right)^{1/3}
\]
and let $t=c_\Gamma\rho^{1/3}$. Then,
\[
    \frac{t^3}{2\rho}=\frac{A_\Gamma}{4}
    <\frac{A_\Gamma}{2},
\]
so the preceding upper bound contradicts
$|\triangle pqr|\geq A_\Gamma/2$. Hence, there cannot exist
three distinct points of $S_\rho(\Gamma)$ joined by two successive
edges of length less than $c_\Gamma\rho^{1/3}$.

Form the graph on $S_\rho(\Gamma)$ in which two distinct vertices are
adjacent exactly when their distance is less than
$c_\Gamma\rho^{1/3}$, and let the sets $\Omega_\alpha$ be its connected
components. A connected component containing at least three vertices
contains a simple path with three distinct vertices, which was just
excluded. Thus, the inequality (\ref{eq:lattice_cluster_size}) holds. Distinct connected
components have no edge between them, which is precisely the inequality 
(\ref{eq:lattice_cluster_separation}).
\end{proof}

\begin{proposition}[Positive-measure observation on the triangular flat torus]
\label{prop:positive_measure_torus}
Let
$
    \mathbb T_\Lambda^2=\mathbb R^2/\Lambda
$
be the flat torus associated with the equilateral-triangle reflection
lattice. For every measurable set
$\omega\subset\mathbb T_\Lambda^2$ with $|\omega|>0$, there exists a
constant $C_{\omega,\Lambda}>0$ such that
\begin{equation}
    \|U\|_{L^2(\mathbb T_\Lambda^2)}^2
    \leq
    C_{\omega,\Lambda}
    \|U\|_{L^2(\omega)}^2
\label{eq:positive_measure_torus}
\end{equation}
for every eigenvalue $\lambda$ of
$-\Delta_{\mathbb T_\Lambda^2}$ and every
$U\in\mathcal H_\lambda$. The constant is independent of $\lambda$
and of the multiplicity of $\mathcal H_\lambda$.
\end{proposition}

\begin{proof}
Let $\Gamma=\Lambda^*$ be the dual lattice.  With the convention in
Eq. (\ref{eq:dual_lattice}), the toral characters are
$x\mapsto e^{i\langle\xi,x\rangle}$, $\xi\in\Gamma$.  For an
eigenvalue corresponding to radius $\rho$,
every $U\in\mathcal H_\lambda$ has an expansion
\begin{equation}
    U(x)
    =
    \sum_{\xi\in S_\rho(\Gamma)}
    a_\xi e^{i\langle\xi,x\rangle}.
\label{eq:torus_shell_expansion}
\end{equation}
Let $d\nu:=|\mathbb T_\Lambda^2|^{-1}dx$ be normalized Haar probability
measure on the torus. The characters
$x\mapsto e^{i\langle\xi,x\rangle}$, $\xi\in\Gamma$, are then an
orthonormal basis of $L^2(d\nu)$, and Parseval gives the exact identity
\begin{equation}
    \|U\|_{L^2(d\nu)}^2
    =
    \sum_{\xi\in S_\rho(\Gamma)}|a_\xi|^2.
\label{eq:normalized_parseval}
\end{equation}
If $\rho=0$, then $\mathcal H_0$ consists only of constants and is
included in the finite low-frequency sector treated below. We therefore
assume for the moment that $\rho>0$.
Set
\[
    d\mu
    :=
    \frac{\mathbf 1_\omega}{\nu(\omega)}\,d\nu.
\]
Then, $\mu$ is a probability measure and its density with respect to
$\nu$ belongs to $L^2(d\nu)$. Define
\[
    \widehat\mu(\eta)
    :=
    \int_{\mathbb T_\Lambda^2}
        e^{-i\langle\eta,x\rangle}\,d\mu(x),
    \qquad \eta\in\Gamma.
\]
By Parseval applied to the $L^2(d\nu)$ density of $\mu$,
\[
    (\widehat\mu(\eta))_{\eta\in\Gamma}\in\ell^2(\Gamma),
    \qquad
    \widehat\mu(\eta)\to0
    \quad\text{as }|\eta|\to\infty.
\]
Moreover,
\begin{equation}
    \gamma_\mu
    :=
    1-
    \sup_{\eta\in\Gamma\setminus\{0\}}
    |\widehat\mu(\eta)|
    >0.
\label{eq:gamma_mu_positive}
\end{equation}
Indeed, for every nonzero $\eta$ one has
$|\widehat\mu(\eta)|<1$: equality in the triangle inequality would
force the nonconstant character
$e^{-i\langle\eta,x\rangle}$ to be constant $\mu$-almost everywhere,
whereas each of its level sets has two-dimensional Haar measure zero.
Since $\widehat\mu(\eta)\to0$, only finitely many nonzero frequencies
can satisfy $|\widehat\mu(\eta)|\geq1/2$; the maximum over that finite
set is strictly smaller than one. This proves the inequality
(\ref{eq:gamma_mu_positive}).

We shall also use the following elementary geometric observation.  For
each nonzero $\eta\in\Gamma$, there are at most two ordered pairs
$(\xi,\zeta)\in S_\rho(\Gamma)^2$ such that
\[
    \zeta-\xi=\eta.
\]
Indeed, if $\zeta=\xi+\eta$, then
$|\xi|=|\zeta|=\rho$ implies
\[
    2\langle \xi,\eta\rangle+|\eta|^2=0.
\]
Thus $\xi$ belongs to a fixed affine line, whose intersection with the
circle $|\xi|=\rho$ contains at most two points.

Apply Lemma \ref{lem:lattice_jarnik} to
$S_\rho(\Gamma)=\bigcup_\alpha\Omega_\alpha$. Expanding against $\mu$
gives
\begin{equation}
    \int |U|^2\,d\mu
    =
    \sum_{\xi,\zeta\in S_\rho(\Gamma)}
    a_\xi\overline{a_\zeta}\,
    \widehat\mu(\zeta-\xi).
\label{eq:mu_quadratic_form}
\end{equation}
For a singleton cluster the corresponding diagonal contribution is
$|a_\xi|^2$. For a two-point cluster $\{\xi,\zeta\}$, the associated
$2\times2$ Hermitian form has diagonal entries $1$ and off-diagonal
entry $\widehat\mu(\zeta-\xi)$; its smallest eigenvalue is
$1-|\widehat\mu(\zeta-\xi)|\geq\gamma_\mu$. Hence, the sum of all
within-cluster contributions is at least
\begin{equation}
    \gamma_\mu
    \sum_{\xi\in S_\rho(\Gamma)}|a_\xi|^2
    =
    \gamma_\mu\|U\|_{L^2(d\nu)}^2.
\label{eq:within_cluster_lower}
\end{equation}
For pairs belonging to distinct clusters, the inequality 
(\ref{eq:lattice_cluster_separation}) gives
\[
    |\xi-\zeta|
    \geq
    c_\Gamma\rho^{1/3}.
\]
Let $\mathcal P_\rho$ denote the set of ordered pairs
$(\xi,\zeta)$ belonging to distinct clusters. The corresponding
cross-cluster contribution is
\[
    Q_{\mathrm{cross}}
    =
    \sum_{(\xi,\zeta)\in\mathcal P_\rho}
    a_\xi\overline{a_\zeta}\,
    \widehat\mu(\zeta-\xi).
\]
By Cauchy--Schwarz and the equation (\cref{eq:normalized_parseval}),
\begin{align}
    |Q_{\mathrm{cross}}|
    &\leq
    \left(
        \sum_{(\xi,\zeta)\in\mathcal P_\rho}
        |a_\xi|^2|a_\zeta|^2
    \right)^{1/2}
    \left(
        \sum_{(\xi,\zeta)\in\mathcal P_\rho}
        |\widehat\mu(\zeta-\xi)|^2
    \right)^{1/2}\\
    &\leq
    \|U\|_{L^2(d\nu)}^2
    \left(
        \sum_{(\xi,\zeta)\in\mathcal P_\rho}
        |\widehat\mu(\zeta-\xi)|^2
    \right)^{1/2},
\label{eq:cross_cluster_cauchy}
\end{align}
because
\[
    \sum_{(\xi,\zeta)\in\mathcal P_\rho}
        |a_\xi|^2|a_\zeta|^2
    \leq
    \left(\sum_\xi|a_\xi|^2\right)^2.
\]
For each nonzero $\eta\in\Gamma$, at most two ordered pairs on the
circle satisfy $\zeta-\xi=\eta$. Since every pair in
$\mathcal P_\rho$ also satisfies
$|\zeta-\xi|\geq c_\Gamma\rho^{1/3}$,
\begin{equation}
    \sum_{(\xi,\zeta)\in\mathcal P_\rho}
    |\widehat\mu(\zeta-\xi)|^2
    \leq
    2\sum_{\substack{\eta\in\Gamma\\
    |\eta|\geq c_\Gamma\rho^{1/3}}}
    |\widehat\mu(\eta)|^2.
\label{eq:cross_cluster_difference_count}
\end{equation}
Consequently,
\begin{equation}
    |Q_{\mathrm{cross}}|
    \leq
    \mathcal R_\mu(\rho)
    \|U\|_{L^2(d\nu)}^2,
    \qquad
    \mathcal R_\mu(\rho)
    :=
    \left(
        2\sum_{\substack{\eta\in\Gamma\\
        |\eta|\geq c_\Gamma\rho^{1/3}}}
        |\widehat\mu(\eta)|^2
    \right)^{1/2}.
\label{eq:cross_cluster_tail}
\end{equation}
Because $\widehat\mu\in\ell^2(\Gamma)$,
$\mathcal R_\mu(\rho)\to0$ as $\rho\to\infty$. Choose $\rho_0>0$
so that
\begin{equation}
    \mathcal R_\mu(\rho)\leq\frac{\gamma_\mu}{2}
    \qquad\text{for all }\rho\geq\rho_0.
\label{eq:rho0_tail_choice}
\end{equation}
Combining both (\ref{eq:within_cluster_lower}) and (\ref{eq:cross_cluster_tail}) yields
\begin{equation}
    \int |U|^2\,d\mu
    \geq
    \frac{\gamma_\mu}{2}
    \|U\|_{L^2(d\nu)}^2
\label{eq:high_frequency_mu_observation}
\end{equation}
for every toral eigenspace with $\rho\geq\rho_0$. Since
$d\mu=\nu(\omega)^{-1}\mathbf1_\omega d\nu$, this is equivalent to
\begin{equation}
    \|U\|_{L^2(d\nu)}^2
    \leq
    \frac{2}{\gamma_\mu\,\nu(\omega)}
    \|U\|_{L^2(\omega,d\nu)}^2.
\label{eq:high_frequency_nu_observation}
\end{equation}
It remains to consider the finitely many eigenspaces corresponding to
$\rho<\rho_0$.  Fix one such eigenspace $\mathcal H_\lambda$.  If
$0\ne U\in\mathcal H_\lambda$, then $U$ is a nonzero real-analytic
trigonometric polynomial on the torus, and hence its zero set has
two-dimensional Lebesgue measure zero.  Since $|\omega|>0$, it follows
that
\[
    \|U\|_{L^2(\omega)}>0.
\]
The continuous map
\[
    U\longmapsto \|U\|_{L^2(\omega)}^2
\]
therefore has a positive minimum on the $L^2$ unit sphere of
$\mathcal H_\lambda$.  Taking the minimum over the finitely many
eigenspaces with $\rho<\rho_0$ gives a constant
$c_{\mathrm{low}}>0$.

Combining this estimate with
\eqref{eq:high_frequency_nu_observation}, and taking the minimum of the
low- and high-frequency constants, gives a positive constant independent
of $\lambda$.  Finally, replacing $d\nu$ by Lebesgue measure multiplies
the local and global squared norms by the same fixed normalization
factor, and hence leaves their ratio unchanged.  This proves
\eqref{eq:positive_measure_torus}.
\end{proof}

The positive-measure triangular estimate now follows from
Proposition \ref{prop:positive_measure_torus} by the same unfolding argument.

\begin{proof}[Completion of the proofs of
Theorems \ref{thm:dirichlet_unl} and \ref{thm:neumann_unl}]
Let $V\subset T$ be measurable with $|V|>0$, and let
$u\in E_\lambda^B$, where $B\in\{D,N\}$. Let $U\in\mathcal H_\lambda$
be the reflection extension of $u$ furnished by
Lemma \ref{lem:reflection_extension}, and form the reflected observation set
$\omega$ in a fundamental domain of the torus exactly as in the proof
of Lemma \ref{lem:observation_after_unfolding}. Then, $|\omega|>0$, and
\cref{eq:global_norm_relation,eq:local_norm_relation} give
$$
    \|U\|_{L^2(\mathbb T_\Lambda^2)}^2
    =6\|u\|_{L^2(T)}^2,
    \qquad
    \|U\|_{L^2(\omega)}^2
    =6\|u\|_{L^2(V)}^2.
$$
Applying Proposition \ref{prop:positive_measure_torus} to $U$ yields
$$
    6\|u\|_{L^2(T)}^2
    \le
    C_{\omega,\Lambda}\,
    6\|u\|_{L^2(V)}^2.
$$
After cancelling the factor $6$, we obtain
$$
    \|u\|_{L^2(V)}^2
    \ge
    C_{\omega,\Lambda}^{-1}
    \|u\|_{L^2(T)}^2.
$$
The constant depends only on the fixed observation set $V$ and the
triangular reflection lattice, and is independent of $\lambda$, its
multiplicity, and the choice of $u$ in the corresponding eigenspace.
This completes the proofs of
Theorems \ref{thm:dirichlet_unl} and \ref{thm:neumann_unl}
for arbitrary measurable observation sets of positive measure.
\end{proof}

\begin{remark}
\label{rem:observation_set_hypothesis}
The preceding argument depends only on the geometry of a fixed
two-dimensional dual lattice.  In particular, the passage to measurable
observation sets of positive measure does not use the arithmetic
parametrization
$
    m^2-mn+n^2
$
of the triangular spectrum.
\end{remark}

\subsection{Passage to Robin boundary conditions}
\label{subsec:DN_to_robin}

The Neumann observation estimate is uniform over complete eigenspaces
and over the spectral parameter.  To transfer this estimate to Robin
boundary conditions as $\sigma\to0$, it is therefore necessary to
control the corresponding Robin modal subspaces uniformly in the
spectral level.  Convergence of individual eigenvalues or eigenfunctions
does not provide such control.

The relevant comparison is thus a subspace estimate rather than a
pointwise perturbation statement.  Each Robin modal class must remain
uniformly close to its Neumann counterpart as $\sigma\downarrow0$, with
an error bound independent of the modal indices.  This distinction is
essential because the spectral level is allowed to tend to infinity
simultaneously with the Robin parameter tending to zero.  A perturbation
estimate whose constant deteriorates with frequency would therefore not
yield a uniform observation inequality.

The small-Robin analysis below proves precisely this frequency-uniform
modal comparison.  Combined with the complete-eigenspace Neumann
estimate, it yields an observation constant that is uniform
simultaneously in the spectral level and for all sufficiently small
Robin parameters.

\section{Robin spectral structure and boundary perturbations}
\label{sec:robin_structure}

We consider the Robin problem
\begin{equation}
    -\Delta u=\Lambda u
    \qquad \text{in }T,
    \qquad
    \partial_\nu u+\sigma u=0
    \qquad \text{on }\partial T,
\label{eq:robin_problem_sec4}
\end{equation}
where $\sigma\geq0$ is constant along $\partial T$.

The Robin spectrum of the equilateral triangle admits an explicit
parametrization due to McCartin~\cite{McCartin2004Robin}.  Rudnick and
Wigman~\cite{RudnickWigman2022} subsequently obtained uniform estimates
for the displacement of the Robin spectrum from the Neumann spectrum
and analyzed its multiplicities.  We recall the parts of these results
needed below and supplement them with estimates for the corresponding
modal subspaces.

The latter estimates are required because the small-Robin observation
problem is uniform both in the spectral level and over complete
eigenspaces.  In particular, convergence of individual Robin
eigenvalues or eigenfunctions as $\sigma\to0$ is not sufficient; the
Robin modal subspaces must be compared with their Neumann limits
uniformly in the modal indices.

\subsection{McCartin's parametrization}
\label{subsec:mccartin_parametrization}

Let
$
    r=\frac{L}{2\sqrt{3}}
$
be the inradius of $T$.  We write $\lambda$ for a generic spectral
value and $\Lambda_{m,n}(\sigma)$ for the Robin eigenvalue associated
with the McCartin branch indexed by $0\leq m\leq n$.  To distinguish
the side length $L$ from the phase variables below, the latter will
always carry the indices $(m,n)$ when they are first introduced.\\
Following McCartin~\cite{McCartin2004Robin}, for each admissible pair
$0\leq m\leq n$ and $\sigma\geq0$, let
\[
    \bigl(
        L_{m,n}(\sigma),
        M_{m,n}(\sigma),
        N_{m,n}(\sigma)
    \bigr)
\]
denote the corresponding solution of the secular system, characterized
by
\begin{equation}
    L_{m,n}(\sigma)\in
    \left(-\frac{\pi}{2},0\right],
    \qquad
    M_{m,n}(\sigma),N_{m,n}(\sigma)\in
    \left[0,\frac{\pi}{2}\right),
\label{eq:LMN_ranges}
\end{equation}
and
\[
    \bigl(
        L_{m,n}(0),
        M_{m,n}(0),
        N_{m,n}(0)
    \bigr)
    =(0,0,0).
\]
The existence and uniqueness of this phase triple follow from
McCartin's parametrization.  The phases satisfy
\begin{align}
&
\bigl(
2L-M-N-(m+n)\pi
\bigr)\tan L
=
3r\sigma,
\label{eq:secular_L}
\\
&
\bigl(
2M-N-L+m\pi
\bigr)\tan M
=
3r\sigma,
\label{eq:secular_M}
\\
&
\bigl(
2N-L-M+n\pi
\bigr)\tan N
=
3r\sigma.
\label{eq:secular_N}
\end{align}
Here and below, the subscripts $(m,n)$ on $L,M,N$ are omitted whenever
no ambiguity arises.

Set
\begin{equation}
    \mu
    =
    m+\frac{2M-N-L}{\pi},
    \qquad
    \nu
    =
    n+\frac{2N-L-M}{\pi},
\label{eq:mu_nu}
\end{equation}
and
\begin{equation}
    \ell=-\mu-\nu.
\label{eq:ell_parameter}
\end{equation}
The corresponding Robin eigenvalue is
\begin{equation}
    \Lambda_{m,n}(\sigma)
    =
    \frac{4\pi^2}{27r^2}
    \bigl(
        \mu^2+\mu\nu+\nu^2
    \bigr).
\label{eq:robin_eigenvalue}
\end{equation}
At $\sigma=0$, the phase variables vanish, so that
\[
    \mu=m,
    \qquad
    \nu=n,
    \qquad
    \ell=-m-n.
\]
Consequently,
\begin{equation}
    \Lambda_{m,n}(0)
    =
    \frac{4\pi^2}{27r^2}
    \bigl(
        m^2+mn+n^2
    \bigr),
\label{eq:neumann_limit_eigenvalue}
\end{equation}
the corresponding Neumann eigenvalue.

\subsection{Robin modal functions}
\label{subsec:robin_modal_functions}

The Robin eigenfunctions decompose into symmetric and antisymmetric
modal families with respect to the altitude $x=\sqrt{3}r$.  In the
notation of the preceding subsection, McCartin's formula
\cite{McCartin2004Robin,RudnickWigman2022} takes the form
\begin{align}
T_{m,n,\sigma}^{s/a}(x,y)
={}&
\cos\!\left(
    \frac{\pi\ell}{3r}(3r-y)-\delta_1
\right)
\left\{\begin{matrix}\cos\\ \sin\end{matrix}\right\}
\!\left(
    \frac{\sqrt3\pi(\mu-\nu)}{9r}
    (x-\sqrt3 r)
\right)
\notag\\
&+
\cos\!\left(
    \frac{\pi\mu}{3r}(3r-y)-\delta_2
\right)
\left\{\begin{matrix}\cos\\ \sin\end{matrix}\right\}
\!\left(
    \frac{\sqrt3\pi(\nu-\ell)}{9r}
    (x-\sqrt3 r)
\right)
\notag\\
&+
\cos\!\left(
    \frac{\pi\nu}{3r}(3r-y)-\delta_3
\right)
\left\{\begin{matrix}\cos\\ \sin\end{matrix}\right\}
\!\left(
    \frac{\sqrt3\pi(\ell-\mu)}{9r}
    (x-\sqrt3 r)
\right),
\label{eq:mccartin-modal-formula}
\end{align}
where the upper and lower choices give the symmetric and antisymmetric
modes, respectively.  The coefficients of the three separated terms
are fixed, while their frequencies and phases depend on
$(m,n,\sigma)$ through $\ell,\mu,\nu$ and $\delta_1,\delta_2,\delta_3$.

\begin{remark}
\label{rem:modal-parity-degeneracy}
The dimensions of the modal spaces remain unchanged at the Neumann
endpoint, apart from the systematic diagonal degeneracy.  If $m=n$,
symmetry and uniqueness of the McCartin branch give $M=N$, and hence
$\mu=\nu$ and $\delta_2=\delta_3$.  The first term in the
antisymmetric formula vanishes, while the remaining two terms cancel.
Thus
\[
    T_{m,m,\sigma}^{a}\equiv0,
\]
whereas the symmetric mode is nonzero.
For $m=0<n$ and $\sigma=0$,
\[
    (\ell,\mu,\nu)=(-n,0,n).
\]
The three transverse frequency differences are then $-n$, $2n$, and
$-n$.  The term of frequency $2n$ cannot be canceled by the other two,
so both $T_{0,n,0}^{s}$ and $T_{0,n,0}^{a}$ are nonzero.  Finally,
$T_{0,0,0}^{s}\equiv3$ and $T_{0,0,0}^{a}\equiv0$.  Hence the modal
space has dimension one for $m=n$ and dimension two for $m<n$ near
$\sigma=0$.
\end{remark}
Expanding the trigonometric factors in
\eqref{eq:mccartin-modal-formula} into complex exponentials gives
\begin{equation}
    T_{m,n,\sigma}^{\varepsilon}(x)
    =
    \sum_{q=1}^{N_0}
    a_{q}^{\varepsilon}(m,n,\sigma)
    e^{i\xi_q(m,n,\sigma)\cdot x},
    \qquad
    \varepsilon\in\{s,a\},
\label{eq:robin_finite_exponential}
\end{equation}
where $N_0$ is independent of $m,n$, and $\sigma$.  Thus each modal
representative has uniformly bounded Fourier complexity.  A complete
Robin eigenspace, however, may contain contributions from several
distinct modal classes.

For a Robin eigenvalue $\Lambda$, let
\begin{equation}
    \mathcal I_\sigma(\Lambda)
    =
    \left\{
        (m,n):
        \Lambda_{m,n}(\sigma)=\Lambda
    \right\}/\!\sim ,
\label{eq:robin_index_classes}
\end{equation}
where $\sim$ identifies the systematic symmetry copies belonging to
the same modal class, and set
\begin{equation}
    \kappa_\sigma(\Lambda)
    =
    \#\mathcal I_\sigma(\Lambda).
\label{eq:robin_complexity_sec4}
\end{equation}
The quantity $\kappa_\sigma(\Lambda)$ counts the distinct modal classes
that contribute to the eigenspace associated with $\Lambda$.

\begin{lemma}[Uniform Fourier complexity of a Robin modal class]
\label{lem:robin-modal-class-complexity}
There exists an integer $N_{\mathrm{mod}}\geq1$, independent of
$m,n$, and $\sigma$, such that every desymmetrized Robin modal space
$\mathcal R_{\mathfrak c}(\sigma)$ is contained in the span of at most
$N_{\mathrm{mod}}$ plane waves.  Moreover, for every Robin eigenvalue
$\Lambda$,
\begin{equation}
    E_{\Lambda,\sigma}
    =
    \sum_{\mathfrak c\in\mathcal I_\sigma(\Lambda)}
    \mathcal R_{\mathfrak c}(\sigma).
\label{eq:fixed-robin-complete-modal-decomposition}
\end{equation}
Consequently, every $u\in E_{\Lambda,\sigma}$ admits a representation
\begin{equation}
    u(x)=\sum_{j=1}^{M}a_j e^{i\xi_j\cdot x},
    \qquad
    M\leq
    N_{\mathrm{mod}}\,\kappa_\sigma(\Lambda),
\label{eq:modal-class-frequency-count}
\end{equation}
after coincident frequencies are combined.
\end{lemma}

\begin{proof}
By McCartin's construction, the Robin eigenspaces are spanned by the
modal representatives associated with the indexed spectral classes.
Each desymmetrized class contains only a uniformly bounded number of
symmetry and parity representatives.  By
\eqref{eq:robin_finite_exponential}, each such representative is a
linear combination of at most $N_0$ plane waves.  This gives the
uniform bound $N_{\mathrm{mod}}$ for a single modal class and, by
completeness of the modal system,
\eqref{eq:fixed-robin-complete-modal-decomposition}.  Summing over the
$\kappa_\sigma(\Lambda)$ contributing classes and combining coincident
frequencies gives \eqref{eq:modal-class-frequency-count}.
\end{proof}

\subsection{Robin--Neumann spectral displacement}
\label{subsec:robin_neumann_gap}

We shall use the following uniform bound on the displacement of the
Robin spectrum from the Neumann spectrum, due to Rudnick and
Wigman~\cite{RudnickWigman2022}.

\begin{proposition}[Uniform Robin--Neumann spectral displacement]
\label{prop:robin_neumann_gap}
For every $\sigma>0$ and every admissible pair $(m,n)$,
\begin{equation}
    0
    <
    \Lambda_{m,n}(\sigma)-\Lambda_{m,n}(0)
    <
    \frac{4\sigma}{r}.
\label{eq:uniform_RN_gap}
\end{equation}
Consequently, for every $\sigma_1>0$,
\begin{equation}
    0\leq
    \Lambda_{m,n}(\sigma)-\Lambda_{m,n}(0)
    \leq
    \frac{4\sigma_1}{r},
    \qquad
    0\leq\sigma\leq\sigma_1,
\label{eq:uniform_RN_gap_compact}
\end{equation}
uniformly in $(m,n)$.
\end{proposition}

\begin{proof}
The strict estimate \eqref{eq:uniform_RN_gap} is proved in
Rudnick and Wigman~\cite{RudnickWigman2022};
\eqref{eq:uniform_RN_gap_compact} follows immediately.
\end{proof}

The spectral displacement estimate alone does not control the
corresponding eigenspaces.  For the small-Robin problem we shall also
need a comparison of the Robin and Neumann modal subspaces that is
uniform in the modal indices.

\subsection{A two-regime shell-splitting criterion}
\label{subsec:abstract_shell_splitting}

The shellwise ordering argument requires separate estimates in the
edge and bulk regimes.  In each regime, strict ordering follows once
the lower-order terms are dominated by the leading spectral gap.

\begin{lemma}[Comparison near a stationary point]
\label{lem:stationary_profile_comparison}
Let $I=[a,\theta_*]$ and let $f,g\in C^2(I)$.  Suppose that
\begin{equation}
    f'(\theta)<0
    \quad (a\leq\theta<\theta_*),
    \qquad
    f'(\theta_*)=g'(\theta_*)=0,
    \qquad
    f''(\theta_*)>0.
\label{eq:abstract_stationary_assumptions}
\end{equation}
Then, there exists $C>0$ such that
\begin{equation}
    |g(\theta')-g(\theta)|
    \leq
    C\bigl(f(\theta)-f(\theta')\bigr),
    \qquad
    a\leq\theta<\theta'\leq\theta_*.
\label{eq:abstract_stationary_comparison}
\end{equation}
\end{lemma}

\begin{proof}
Choose $\varepsilon>0$ such that
$f''\geq c_0>0$ on
$[\theta_*-\varepsilon,\theta_*]$.  Since
$f'(\theta_*)=g'(\theta_*)=0$,
\[
    -f'(t)\geq c_0(\theta_*-t),
    \qquad
    |g'(t)|\leq C_0(\theta_*-t)
\]
on this interval.  Hence, $|g'|\leq C_1(-f')$ near $\theta_*$.  On
$[a,\theta_*-\varepsilon]$, the same estimate follows from the strict
negativity of $f'$ and compactness.  Integration from $\theta$ to
$\theta'$ gives \eqref{eq:abstract_stationary_comparison}.
\end{proof}

\begin{proposition}[Two-regime shell splitting]
\label{prop:abstract_two_regime_splitting}
Fix $\tau$.  For each sufficiently large $R$, let $\mathcal A_R$ be a
finite totally ordered family of spectral branches, with order denoted
by $\alpha\prec\beta$, and write
\begin{equation}
    \mathcal A_R
    =
    \mathcal E_R\,\dot\cup\,\mathcal B_R,
\label{eq:abstract_edge_bulk_partition}
\end{equation}
where $\mathcal E_R$ and $\mathcal B_R$ are respectively an initial
and a final segment.  Let $\Lambda_\alpha(\tau)$ be the perturbed
eigenvalue associated with $\alpha\in\mathcal A_R$, and let
$H_R:\mathcal A_R\to\mathbb R$ be strictly decreasing in the branch
order.  For $\alpha\prec\beta$, set
\begin{equation}
    \Delta H_R(\alpha,\beta)
    :=
    H_R(\alpha)-H_R(\beta)>0.
\label{eq:abstract_H_gap_definition}
\end{equation}
Suppose that there exists $A_\tau>0$ such that the following estimates
hold.

In the edge regime, let
$s:\mathcal E_R\to[1,\infty)$ and assume that
\begin{align}
    \Lambda_\beta(\tau)-\Lambda_\alpha(\tau)
    &=
    A_\tau\Delta H_R(\alpha,\beta)
    +\mathcal E^{\rm e}_{\alpha\beta},
\label{eq:abstract_edge_difference}\\
    \Delta H_R(\alpha,\beta)
    &\geq
    d_\tau^{\rm e}(s(\alpha)),
    \qquad
    |\mathcal E^{\rm e}_{\alpha\beta}|
    \leq
    r_\tau^{\rm e}(s(\alpha)),
\label{eq:abstract_edge_gap}
\end{align}
for $\alpha\in\mathcal E_R$ and $\alpha\prec\beta$, where
$d_\tau^{\rm e},r_\tau^{\rm e}>0$ and
\begin{equation}
    \frac{r_\tau^{\rm e}(s)}
         {d_\tau^{\rm e}(s)}
    \longrightarrow0
    \qquad (s\to\infty).
\label{eq:abstract_edge_scale_condition}
\end{equation}
In the bulk regime, suppose that there are ordered coordinates
$\theta_\alpha\in I=[a,\theta_*]$, functions
$f,g_\tau\in C^2(I)$, positive scales $a_R,b_R,d_R$, and
$e_R\geq0$ such that, whenever
$\alpha\prec\beta$ belong to $\mathcal B_R$,
\begin{align}
    \Lambda_\beta(\tau)-\Lambda_\alpha(\tau)
    ={}&
    A_\tau a_R
    \bigl(
        f(\theta_\alpha)-f(\theta_\beta)
    \bigr)
\notag\\
    &+
    b_R
    \bigl(
        g_\tau(\theta_\beta)-g_\tau(\theta_\alpha)
    \bigr)
    +\mathcal E^{\rm b}_{\alpha\beta},
\label{eq:abstract_bulk_difference}
\end{align}
with
\begin{align}
    |g_\tau(\theta_\beta)-g_\tau(\theta_\alpha)|
    &\leq
    C_\tau
    \bigl(
        f(\theta_\alpha)-f(\theta_\beta)
    \bigr),
\label{eq:abstract_bulk_profile_control}\\
    f(\theta_\alpha)-f(\theta_\beta)
    &\geq d_R,
    \qquad
    |\mathcal E^{\rm b}_{\alpha\beta}|
    \leq e_R.
\label{eq:abstract_bulk_discrete_gap}
\end{align}
Assume further that
\begin{equation}
    \frac{b_R}{a_R}\longrightarrow0,
    \qquad
    \frac{e_R}{a_Rd_R}\longrightarrow0
    \qquad (R\to\infty).
\label{eq:abstract_scale_conditions}
\end{equation}
Then, there exist $S_\tau,R_\tau<\infty$ such that
\begin{equation}
    \Lambda_\alpha(\tau)<\Lambda_\beta(\tau)
\label{eq:abstract_eventual_order}
\end{equation}
for every ordered edge pair with
$s(\alpha)\geq S_\tau$ and every ordered bulk pair with
$R\geq R_\tau$.
\end{proposition}

\begin{proof}
For an edge pair,
\eqref{eq:abstract_edge_difference} and
\eqref{eq:abstract_edge_gap} give
\[
    \Lambda_\beta(\tau)-\Lambda_\alpha(\tau)
    \geq
    d_\tau^{\rm e}(s)
    \left(
        A_\tau
        -
        \frac{r_\tau^{\rm e}(s)}
             {d_\tau^{\rm e}(s)}
    \right),
\]
where $s=s(\alpha)$.  This is positive for all sufficiently large
$s$ by \eqref{eq:abstract_edge_scale_condition}.
For a bulk pair, set
\[
    \Delta f
    =
    f(\theta_\alpha)-f(\theta_\beta)>0.
\]
By \eqref{eq:abstract_bulk_difference}--%
\eqref{eq:abstract_bulk_discrete_gap},
\[
    \Lambda_\beta(\tau)-\Lambda_\alpha(\tau)
    \geq
    a_R\Delta f
    \left(
        A_\tau-C_\tau\frac{b_R}{a_R}
    \right)
    -e_R.
\]
Since $\Delta f\geq d_R$,
\[
    \frac{
        \Lambda_\beta(\tau)-\Lambda_\alpha(\tau)
    }{
        a_Rd_R
    }
    \geq
    A_\tau
    -
    C_\tau\frac{b_R}{a_R}
    -
    \frac{e_R}{a_Rd_R},
\]
which is positive for all sufficiently large $R$ by
\eqref{eq:abstract_scale_conditions}.
\end{proof}

\begin{corollary}[Bounded one-shell coincidence complexity]
\label{cor:abstract_exceptional_core}
Assume the hypotheses of
Proposition~\ref{prop:abstract_two_regime_splitting}.  Suppose, in
addition, that on every shell at most $L_\tau<\infty$ branches fail
the ordering thresholds of the proposition.  Then every spectral
value is attained by at most
\begin{equation}
    B_\tau:=L_\tau+1
\label{eq:abstract_one_shell_complexity}
\end{equation}
branch classes on a single shell.
\end{corollary}

\begin{proof}
Outside the exceptional set the perturbed eigenvalues are strictly
ordered, so a fixed spectral value can occur there at most once.
At most $L_\tau$ further occurrences can belong to the exceptional
set.
\end{proof}

\begin{corollary}[Global coincidence bound]
\label{cor:abstract_bounded_displacement}
Let $\mathcal A$ be a discrete family of branches with a
reference-shell map
\[
    Q:\mathcal A\longrightarrow\mathcal J
\]
and reference levels
$\{E_j:j\in\mathcal J\}\subset\mathbb R$.  For $L>0$, set
\begin{equation}
    \mathcal N_0(L)
    :=
    \sup_{x\in\mathbb R}
    \#\{
        j\in\mathcal J:
        E_j\in[x-L,x]
    \}.
\label{eq:abstract_reference_density}
\end{equation}
Fix $\tau$ and suppose that
$\mathcal N_0(D_\tau)<\infty$ for some $D_\tau<\infty$, that
\begin{equation}
    0<
    \Lambda_\alpha(\tau)-E_{Q(\alpha)}
    <D_\tau
    \qquad
    (\alpha\in\mathcal A),
\label{eq:abstract_one_sided_displacement}
\end{equation}
and that every reference shell contributes at most $B_\tau$ branch
classes to a single perturbed spectral value.  Then
\begin{equation}
    \sup_{\Lambda}
    \#\{
        \alpha:
        \Lambda_\alpha(\tau)=\Lambda
    \}
    \leq
    B_\tau\,\mathcal N_0(D_\tau)
    <\infty.
\label{eq:abstract_global_coincidence_bound}
\end{equation}
\end{corollary}

\begin{proof}
If $\Lambda_\alpha(\tau)=\Lambda$, then
\eqref{eq:abstract_one_sided_displacement} implies
\[
    E_{Q(\alpha)}
    \in
    (\Lambda-D_\tau,\Lambda).
\]
Hence, at most $\mathcal N_0(D_\tau)$ reference levels can contribute
to $\Lambda$.  Since each corresponding shell contributes at most
$B_\tau$ branch classes, the conclusion follows.
\end{proof}

\begin{remark}[Equilateral-triangle specialization]
\label{rem:abstract_criterion_triangle_map}
For the equilateral Robin spectrum, $\tau=\sigma$ and the leading
shell statistic is $H_R=F$.  In the edge regime, the leading discrete
gap is of order $m^{-3}$, while the remainder is
$O_\sigma(m^{-4})$.  In the bulk,
\[
    a_R=R^{-2},
    \qquad
    b_R=R^{-4},
    \qquad
    d_R\asymp R^{-2},
    \qquad
    e_R=O_\sigma(R^{-6}).
\]
Lemma~\ref{lem:stationary_profile_comparison} gives the required
profile comparison at the symmetric diagonal.

The exceptional branches lie in a bounded range of transverse
indices, while Proposition~\ref{prop:robin_neumann_gap} gives a
uniform Robin--Neumann displacement.  The Neumann reference levels
are
\[
    E_Q
    =
    \frac{4\pi^2}{27r^2}Q,
    \qquad
    Q\in\mathbb Z_{\geq0},
\]
and therefore,
\begin{equation}
    \mathcal N_0(L)
    \leq
    \left\lfloor
        \frac{27r^2}{4\pi^2}L
    \right\rfloor+1.
\label{eq:triangle_reference_density}
\end{equation}
Thus, the hypotheses of the preceding corollaries reduce to the
shellwise estimates established below.
\end{remark}

\subsection{Fixed-parameter angular expansion}
\label{subsec:fixed_sigma_angular_expansion}

We now keep $\sigma>0$ fixed and analyze the high-frequency geometry
of the exact secular system.  This is a different asymptotic regime
from the small-$\sigma$ perturbation used later.  Set
\begin{equation}
    h:=3r\sigma,
    \qquad
    \mathcal Q(m,n):=m^2+mn+n^2,
\label{eq:h_Q_def}
\end{equation}
and introduce the three-index notation
\[
    m_1=m,\qquad m_2=n,\qquad m_3=-(m+n),
\]
\[
    \mu_1=\mu,\qquad \mu_2=\nu,\qquad
    \mu_3=-\mu-\nu,
\]
with $M_1=M$, $M_2=N$, and $M_3=L$.  Then
\eqref{eq:secular_L}--\eqref{eq:secular_N} and
\eqref{eq:mu_nu} give the exact identities
\begin{equation}
    \mu_j\tan M_j=\frac{h}{\pi},
    \qquad
    \mu_j
    =m_j+\frac{2M_j-M_i-M_k}{\pi},
\label{eq:exact_three_index_system}
\end{equation}
for $\{i,j,k\}=\{1,2,3\}$.  Since the phases remain in the boxes
\eqref{eq:LMN_ranges}, this is equivalently
\begin{equation}
    M_j=\arctan\!\left(\frac{h}{\pi\mu_j}\right).
\label{eq:exact_arctan_phase}
\end{equation}
Thus, with
\[
    \Phi_h(t):=\frac1\pi
    \arctan\!\left(\frac{h}{\pi t}\right),
\]
the two positive coordinates satisfy the exact fixed-point system
\begin{align}
    \mu
    &=m+2\Phi_h(\mu)-\Phi_h(\nu)+\Phi_h(\mu+\nu),
\label{eq:fixed_point_mu}\\
    \nu
    &=n+2\Phi_h(\nu)-\Phi_h(\mu)+\Phi_h(\mu+\nu).
\label{eq:fixed_point_nu}
\end{align}
Define the angular statistic
\begin{equation}
    F(m,n)
    :=
    \frac1{m^2}+\frac1{n^2}+\frac1{(m+n)^2}
\label{eq:F_angular_statistic}
\end{equation}
for $m\geq1$, and
\begin{equation}
    A_\sigma
    :=
    \frac{4\sigma^2(1+r\sigma)}{\pi^2}.
\label{eq:A_sigma}
\end{equation}
In particular,
\begin{equation}
    A_\sigma>0
    \qquad\text{for every fixed }\sigma>0.
\label{eq:A_sigma_positive}
\end{equation}

\begin{lemma}[Fixed-$\sigma$ high-frequency angular expansion]
\label{lem:fixed_sigma_angular_expansion}
Fix $\sigma>0$.\\
\emph{(i) Bulk expansion.}
For every $\delta\in(0,1/\sqrt3)$, there exist
$R_{\sigma,\delta}$, a function
$g_{\sigma,\delta}\in C^2([\delta,1/\sqrt3])$, and a constant
$C_{\sigma,\delta}$ such that, whenever
\[
    R^2=\mathcal Q(m,n),\qquad
    \delta\leq x:=\frac mR\leq\frac1{\sqrt3},
    \qquad R\geq R_{\sigma,\delta},
\]
one has
\begin{equation}
\begin{split}
    \Lambda_{m,n}(\sigma)
    ={}&
    \frac{4\pi^2}{27r^2}R^2
    +\frac{4\sigma}{r}
    -A_\sigma R^{-2}f(x)
    +R^{-4}g_{\sigma,\delta}(x)
    +\mathcal E_{\sigma,R}(x),
\end{split}
\label{eq:bulk_fixed_sigma_expansion}
\end{equation}
where
\begin{equation}
    f(x)=\frac1{x^2(1-x^2)^2}
\label{eq:f_bulk}
\end{equation}
and
\begin{equation}
    \max_{0\leq j\leq2}
    \sup_{x\in[\delta,1/\sqrt3]}
    \left|\partial_x^j\mathcal E_{\sigma,R}(x)\right|
    \leq C_{\sigma,\delta}R^{-6}.
\label{eq:bulk_remainder_C2}
\end{equation}
The coefficient function is induced by a symmetric $C^2$ function of
$ (x,y)$ on the normalized shell.  In particular,
\begin{equation}
    g_{\sigma,\delta}'(1/\sqrt3)=0.
\label{eq:g_stationary_diagonal}
\end{equation}
\emph{(ii) Edge expansion.}
There exist $M_\sigma$ and $C_\sigma$ such that, uniformly for
$m\geq M_\sigma$ and $n\geq m$,
\begin{equation}
    \Lambda_{m,n}(\sigma)
    =
    \Lambda_{m,n}(0)
    +\frac{4\sigma}{r}
    -A_\sigma F(m,n)
    +\mathcal R_\sigma(m,n),
\label{eq:edge_fixed_sigma_expansion}
\end{equation}
with
\begin{equation}
    |\mathcal R_\sigma(m,n)|
    \leq C_\sigma m^{-4}.
\label{eq:edge_remainder}
\end{equation}
\end{lemma}

\begin{proof}
Throughout the proof $h=3r\sigma$ is fixed.  We stress that the
expansion parameter is the inverse frequency, not $\sigma$.  Introduce
\begin{equation}
    \Psi_h(z,\rho)
    :=\frac{\sqrt\rho}{\pi}
      \arctan\!\left(\frac{h\sqrt\rho}{\pi z}\right).
\label{eq:Psi_h_rho}
\end{equation}
Although written with $\sqrt\rho$, this function has the convergent
power series
\begin{equation}
    \Psi_h(z,\rho)
    =\sum_{k\geq0}
      \frac{(-1)^k h^{2k+1}}
      {(2k+1)\pi^{2k+2}z^{2k+1}}\,\rho^{k+1},
\label{eq:Psi_h_series}
\end{equation}
whenever $z$ stays in a compact subset of $(0,\infty)$ and $\rho$ is
sufficiently small.  Thus, $\Psi_h$ extends real-analytically across
$\rho=0$.  This elementary observation is what permits an analytic
implicit-function argument in $\rho=R^{-2}$ rather than in $R^{-1}$.\\
For the bulk regime, write
\[
    m=Rx,\qquad n=Ry,\qquad x^2+xy+y^2=1,
\]
where $y=y(x)>0$ and $x\in[\delta,1/\sqrt3]$.  Put
$\rho=R^{-2}$, $U=R^{-1}\mu$, and $V=R^{-1}\nu$.  Multiplying
\eqref{eq:fixed_point_mu}--\eqref{eq:fixed_point_nu} by $R^{-1}$ gives
the exact system
\begin{align}
 U={}&x+2\Psi_h(U,\rho)-\Psi_h(V,\rho)
        +\Psi_h(U+V,\rho),
\label{eq:bulk_scaled_U}\\
 V={}&y+2\Psi_h(V,\rho)-\Psi_h(U,\rho)
        +\Psi_h(U+V,\rho).
\label{eq:bulk_scaled_V}
\end{align}
Let $\mathcal F=(\mathcal F_1,\mathcal F_2)$ denote the left-hand
side minus the right-hand side of
\eqref{eq:bulk_scaled_U}--\eqref{eq:bulk_scaled_V}.  At $\rho=0$,
\[
    \mathcal F(x,y,0;x)=0,
    \qquad
    D_{(U,V)}\mathcal F(x,y,0;x)=I_2.
\]
Since $x$ ranges over a compact interval on which $x$, $y(x)$, and
$x+y(x)$ are bounded away from zero, the analytic implicit-function
theorem, followed by a finite-cover argument, gives a common
$\rho_{\sigma,\delta}>0$ for which $U$ and $V$ are jointly analytic in
$(\rho,x)$.  In particular, in $C^2([\delta,1/\sqrt3])$,
\begin{align}
 U(x,\rho)
 &=x+\rho u_1(x)+\rho^2u_2(x)+\rho^3u_3(x)
      +O_{\sigma,\delta}(\rho^4),
\label{eq:bulk_U_Taylor}\\
 V(x,\rho)
 &=y+\rho v_1(x)+\rho^2v_2(x)+\rho^3v_3(x)
      +O_{\sigma,\delta}(\rho^4).
\label{eq:bulk_V_Taylor}
\end{align}
The first coefficient pair obtained from
\eqref{eq:bulk_scaled_U}--\eqref{eq:bulk_scaled_V} is
\begin{align}
 u_1
 &=\frac{h}{\pi^2}
 \left(\frac2x-\frac1y+\frac1{x+y}\right),
\label{eq:u1_bulk}\\
 v_1
 &=\frac{h}{\pi^2}
 \left(\frac2y-\frac1x+\frac1{x+y}\right).
\label{eq:v1_bulk}
\end{align}
Set $S(U,V):=U^2+UV+V^2$.  Since
$S(x,y)=x^2+xy+y^2=1$, expansion of $S$ using
\eqref{eq:bulk_U_Taylor}--\eqref{eq:bulk_V_Taylor} gives
\begin{equation}
 S(U,V)
 =1+\rho S_1(x)+\rho^2S_2(x)+\rho^3S_3(x)
   +O_{\sigma,\delta}(\rho^4)
\label{eq:S_bulk_expansion}
\end{equation}
uniformly in $C^2$ with respect to $x$.  From
\eqref{eq:u1_bulk}--\eqref{eq:v1_bulk},
\begin{equation}
 S_1=(2x+y)u_1+(x+2y)v_1=\frac{9h}{\pi^2}.
\label{eq:S1_bulk}
\end{equation}
Equating the $\rho^2$ coefficients in
\eqref{eq:bulk_scaled_U}--\eqref{eq:bulk_scaled_V}, inserting
\eqref{eq:u1_bulk}--\eqref{eq:v1_bulk}, and then forming the combination
that occurs in $S_2$ yields the explicit identity
\begin{align}
 S_2
 &=(2x+y)u_2+(x+2y)v_2+u_1^2+u_1v_1+v_1^2
\notag\\
 &=-\frac{h^2(h+3)}{\pi^4}
   \frac{(x^2+xy+y^2)^2}{x^2y^2(x+y)^2}.
\label{eq:S2_bulk}
\end{align}
No expansion in $\sigma$ is involved in this coefficient calculation.
On the normalized shell $x^2+xy+y^2=1$ and
$1-x^2=y(x+y)$, so the rational factor in
\eqref{eq:S2_bulk} is precisely $f(x)$.
Since
\[
    \mu^2+\mu\nu+\nu^2
    =R^2S(U,V)=\rho^{-1}S(U,V),
\]
\eqref{eq:S_bulk_expansion}--\eqref{eq:S2_bulk} imply
\begin{equation}
\begin{split}
 \mu^2+\mu\nu+\nu^2
 ={}&R^2+\frac{9h}{\pi^2}
 -\frac{h^2(h+3)}{\pi^4}R^{-2}f(x)\\
 &+R^{-4}S_3(x)+O_{\sigma,\delta}(R^{-6}).
\end{split}
\label{eq:Q_fixed_sigma_expansion}
\end{equation}
Multiplication by $4\pi^2/(27r^2)$ and use of $h=3r\sigma$
give
\[
 \frac{4\pi^2}{27r^2}\frac{9h}{\pi^2}=\frac{4\sigma}{r},
 \qquad
 \frac{4\pi^2}{27r^2}
 \frac{h^2(h+3)}{\pi^4}
 =\frac{4\sigma^2(1+r\sigma)}{\pi^2}=A_\sigma.
\]
Thus, \eqref{eq:bulk_fixed_sigma_expansion} holds with
$g_{\sigma,\delta}=(4\pi^2/(27r^2))S_3$ and with the remainder
\eqref{eq:bulk_remainder_C2}.  The exact system is invariant under
interchange of the two positive coordinates.  Hence $S_3$ is the
restriction to the shell of a symmetric $C^2$ function
$G_\sigma(x,y)=G_\sigma(y,x)$.  At
$x_*=y_*=1/\sqrt3$ one has $y'(x_*)=-1$ and
$\partial_xG_\sigma(x_*,x_*)=\partial_yG_\sigma(x_*,x_*)$.  Therefore,
\[
 \frac{d}{dx}G_\sigma(x,y(x))\bigg|_{x=x_*}
 =\partial_xG_\sigma-\partial_yG_\sigma=0,
\]
which proves \eqref{eq:g_stationary_diagonal}.\\
 Next, we prove the edge estimate without losing uniformity when
$n/m\to\infty$.  Put
\[
    \rho=m^{-2},\qquad q=\frac mn\in(0,1],\qquad
    U=m^{-1}\mu,\qquad V=q m^{-1}\nu.
\]
For convenience, define
\begin{equation}
 \Psi_{h,q}(z,\rho)
 :=\frac{\sqrt\rho}{\pi}
   \arctan\!\left(\frac{hq\sqrt\rho}{\pi z}\right),
\label{eq:Psi_hq}
\end{equation}
which is jointly analytic in $(z,q,\rho)$ near
$z=1$, $q\in[0,1]$, $\rho=0$.  The exact fixed-point equations become
\begin{align}
 U={}&1+2\Psi_h(U,\rho)-\Psi_{h,q}(V,\rho)
       +\Psi_{h,q}(qU+V,\rho),
\label{eq:edge_scaled_U}\\
 V={}&1+2q\Psi_{h,q}(V,\rho)-q\Psi_h(U,\rho)
       +q\Psi_{h,q}(qU+V,\rho).
\label{eq:edge_scaled_V}
\end{align}
At $\rho=0$ the solution is $(U,V)=(1,1)$ and the Jacobian in
$(U,V)$ is $I_2$, uniformly for $q\in[0,1]$.  Compactness of the
$q$-interval and the analytic implicit-function theorem therefore give
solutions $U(q,\rho)$ and $V(q,\rho)$ analytic on a common neighborhood
of $[0,1]\times\{0\}$.  Equation \eqref{eq:edge_scaled_V} also shows
that $V-1$ is divisible by $q$ in this analytic ring; write
\begin{equation}
    V(q,\rho)=1+qW(q,\rho)
\label{eq:V_qW}
\end{equation}
with $W$ analytic.  Dividing \eqref{eq:edge_scaled_V} by $q$ and setting
$q=0$ gives
\[
    W(0,\rho)=-\Psi_h(U(0,\rho),\rho),
\]
whereas \eqref{eq:edge_scaled_U} at $q=0$ gives
$U(0,\rho)-1=2\Psi_h(U(0,\rho),\rho)$.  Consequently
\begin{equation}
    U(0,\rho)-1+2W(0,\rho)=0.
\label{eq:edge_divisibility_identity}
\end{equation}
The cancellation in \eqref{eq:edge_divisibility_identity} is exactly
what removes the apparent singularity of the quadratic form at
$q=0$.  Indeed, using \eqref{eq:V_qW},
\begin{align}
 &U^2+\frac{UV}{q}+\frac{V^2}{q^2}
 -\left(1+\frac1q+\frac1{q^2}\right)
\notag\\
 &\qquad=
 U^2-1+UW+W^2+
 \frac{U-1+2W}{q}.
\label{eq:edge_energy_regularization}
\end{align}
The last quotient extends analytically to $q=0$ by
\eqref{eq:edge_divisibility_identity}.  Hence the left-hand side of
\eqref{eq:edge_energy_regularization}, evaluated on the analytic
solution, has a uniform expansion
\begin{equation}
 \rho C_1(q)+\rho^2C_2(q)+O_\sigma(\rho^3),
 \qquad q\in[0,1].
\label{eq:edge_energy_rho_expansion}
\end{equation}
Direct coefficient matching in
\eqref{eq:edge_scaled_U}--\eqref{eq:edge_scaled_V} gives
\begin{equation}
 C_1(q)=\frac{9h}{\pi^2},
 \qquad
 C_2(q)=-\frac{h^2(h+3)}{\pi^4}
 \frac{(q^2+q+1)^2}{(q+1)^2}.
\label{eq:edge_C1_C2}
\end{equation}
Since $n=m/q$,
\begin{equation}
 m^2F(m,n)
 =1+q^2+\frac{q^2}{(1+q)^2}
 =\frac{(q^2+q+1)^2}{(q+1)^2}.
\label{eq:q_F_identity}
\end{equation}
Finally,
\[
 \mu^2+\mu\nu+\nu^2
 =m^2\left(U^2+\frac{UV}{q}+\frac{V^2}{q^2}\right).
\]
Multiplying \eqref{eq:edge_energy_rho_expansion} by $m^2=\rho^{-1}$,
using \eqref{eq:edge_C1_C2}--\eqref{eq:q_F_identity}, and then restoring
the factor $4\pi^2/(27r^2)$ proves
\eqref{eq:edge_fixed_sigma_expansion}--\eqref{eq:edge_remainder}, with
a constant independent of $q$, hence independent of $n/m$.
\end{proof}

\begin{remark}
\label{rem:fixed_sigma_not_small_sigma}
Lemma \ref{lem:fixed_sigma_angular_expansion} is an expansion in the
inverse frequency with $\sigma>0$ held fixed.  No smallness assumption
on $\sigma$ is used.  In particular, the positivity
\eqref{eq:A_sigma_positive} holds for every finite positive Robin
parameter.  We will use only the exact secular equations and this
fixed-parameter expansion; no third-order Taylor expansion in $\sigma$
is needed below.
\end{remark}

\subsection{Within-shell splitting at fixed Robin parameter}
\label{subsec:within_shell_splitting}

The first angular statistic in
\eqref{eq:edge_fixed_sigma_expansion} has a simple monotonicity along a
Neumann shell.  If
\[
    R^2=m^2+mn+n^2,
    \qquad x=\frac mR,
\]
then a direct elimination of $n/R$ gives
\begin{equation}
    F(m,n)=R^{-2}f(x),
    \qquad
    f(x)=\frac1{x^2(1-x^2)^2}.
\label{eq:F_shell_formula}
\end{equation}
Moreover,
\begin{equation}
    f'(x)
    =
    -\frac{2(1-3x^2)}{x^3(1-x^2)^3}<0,
    \qquad 0<x<\frac1{\sqrt3},
\label{eq:f_strict_monotonicity}
\end{equation}
and $f''(x)>0$ on this interval.  This is the angular ordering used by
Rudnick and Wigman in their analysis of equal Neumann shells; see
\cite[Proposition~7.2]{RudnickWigman2022}.

\begin{proposition}[Eventual within-shell simplicity]
\label{prop:eventual_within_shell_simplicity}
Fix $\sigma>0$. There exists an integer $M_\sigma^*\geq1$ such that,
whenever
\begin{equation}
    m^2+mn+n^2
    =
    {m'}^2+m'n'+{n'}^2,
\label{eq:same_neumann_shell}
\end{equation}
and
\begin{equation}
    M_\sigma^*\leq m<m'\leq n'<n,
\label{eq:ordered_same_shell_indices}
\end{equation}
one has
\begin{equation}
    \Lambda_{m,n}(\sigma)
    <
    \Lambda_{m',n'}(\sigma).
\label{eq:eventual_same_shell_ordering}
\end{equation}
\end{proposition}

\begin{proof}
Let $R^2$ denote the common value in
\eqref{eq:same_neumann_shell}, and put
$x=m/R$, $x'=m'/R$.  By
\eqref{eq:f_strict_monotonicity},
\begin{equation}
    \Delta F
    :=F(m,n)-F(m',n')
    =R^{-2}\bigl(f(x)-f(x')\bigr)>0.
\label{eq:DeltaF_positive}
\end{equation}
We verify the two regimes of Proposition
\ref{prop:abstract_two_regime_splitting} using the expansions in
Lemma \ref{lem:fixed_sigma_angular_expansion}.\\
First consider the edge region $x<\delta$, with $\delta>0$ chosen
small. If $m'\geq2m$, then
\[
    F(m,n)\geq m^{-2},
    \qquad
    F(m',n')\leq \frac{9}{4}{m'}^{-2},
\]
so $\Delta F\geq c m^{-2}$ for an absolute $c>0$.  If instead
$m'<2m$, then $x'<2\delta$.  Choosing $\delta$ so small that
$-f'(t)\geq c_0t^{-3}$ on $(0,2\delta]$, the mean-value theorem and
$m'-m\geq1$ give
\[
    \Delta F
    =R^{-2}[-f'(\xi)]\frac{m'-m}{R}
    \geq c_1m^{-3}
\]
for some $\xi\in(x,x')$.  In both edge subcases,
\eqref{eq:edge_fixed_sigma_expansion}--\eqref{eq:edge_remainder} yield
\[
\begin{split}
    \Lambda_{m',n'}(\sigma)-\Lambda_{m,n}(\sigma)
    &=A_\sigma\Delta F+O_\sigma(m^{-4})>0
\end{split}
\]
for all sufficiently large $m$.
We now treat the bulk region $x\geq\delta$.  By
\eqref{eq:bulk_fixed_sigma_expansion},
\begin{equation}
\begin{split}
    \Lambda_{m',n'}(\sigma)-\Lambda_{m,n}(\sigma)
    ={}&
    A_\sigma R^{-2}\bigl(f(x)-f(x')\bigr)\\
    &+R^{-4}\bigl(g_{\sigma,\delta}(x')
                   -g_{\sigma,\delta}(x)\bigr)
    +O_{\sigma,\delta}(R^{-6}).
\end{split}
\label{eq:bulk_difference_expansion}
\end{equation}
Set $x_*=1/\sqrt3$.  By
\eqref{eq:f_strict_monotonicity}, $f'<0$ on
$[\delta,x_*)$, while $f'(x_*)=0$ and $f''(x_*)>0$.
Together with \eqref{eq:g_stationary_diagonal}, Lemma
\ref{lem:stationary_profile_comparison} therefore gives
\begin{equation}
    |g_{\sigma,\delta}(x')-g_{\sigma,\delta}(x)|
    \leq
    C'_{\sigma,\delta}\bigl(f(x)-f(x')\bigr)
\label{eq:g_difference_vs_f_difference}
\end{equation}
throughout $\delta\leq x<x'\leq x_*$.  Finally,
Proposition~7.2 of Rudnick and Wigman gives, in this bulk regime, the
quantitative lattice separation
\begin{equation}
    \Delta F\geq c_\delta R^{-4}
\label{eq:bulk_DeltaF_lower}
\end{equation}
for distinct representations of the same shell.  Equivalently,
$f(x)-f(x')\geq c_\delta R^{-2}$.  Thus the bulk hypotheses of Proposition
\ref{prop:abstract_two_regime_splitting} hold with
\[
 a_R=R^{-2},\qquad b_R=R^{-4},\qquad
 d_R=c_\delta R^{-2},\qquad
 e_R=C_{\sigma,\delta}R^{-6},
\]
so $b_R/a_R\to0$ and $e_R/(a_Rd_R)\to0$.  In the edge regime the
preceding estimates give a uniform gap $\Delta F\geq c m^{-3}$ and the
remainder is $O_\sigma(m^{-4})$, hence
$r_\sigma^{\rm e}(m)/d_\sigma^{\rm e}(m)\to0$.
Proposition \ref{prop:abstract_two_regime_splitting}, together with
$A_\sigma>0$, therefore gives strict ordering whenever the transverse
index $m$ exceeds a fixed threshold $M_\sigma^*$.  The remaining bounded
range $0\leq m<M_\sigma^*$ is not part of the asymptotic ordering and is
handled by the one-shell coincidence count below.  This proves
\eqref{eq:eventual_same_shell_ordering}.
\end{proof}

\begin{corollary}[Uniform coincidence bound on one Neumann shell]
\label{cor:within_shell_coincidence_bound}
For every fixed $\sigma>0$, there exists a finite integer $B_\sigma$
such that, for every integer $Q\geq0$ and every Robin eigenvalue
$\Lambda$, the number of desymmetrized representatives $0\leq m\leq n$
on that shell contributing to $\Lambda$ satisfies
\begin{equation}
\#\left\{
    (m,n):0\leq m\leq n,\;
    \mathcal Q(m,n)=Q,\;
    \Lambda_{m,n}(\sigma)=\Lambda
\right\}
\leq B_\sigma.
\label{eq:within_shell_Bsigma}
\end{equation}
One may take $B_\sigma=M_\sigma^*+1$ after increasing
$M_\sigma^*$ if necessary.
\end{corollary}

\begin{proof}
By Proposition \ref{prop:eventual_within_shell_simplicity}, at most one
coincident class on a fixed shell can have $m\geq M_\sigma^*$.  All
remaining classes have $0\leq m<M_\sigma^*$.  This is precisely the
bounded-exceptional-core mechanism of Corollary
\ref{cor:abstract_exceptional_core}.  For fixed $Q$ and fixed $m$, the
quadratic equation
\[
    n^2+mn+m^2=Q
\]
has at most one solution with $n\geq m$.  The claim follows, including
the possible $m=0$ class.
\end{proof}

\subsection{Global Robin coincidence complexity}
\label{subsec:global_robin_complexity}

We now combine the within-shell estimate with the exact
Robin--Neumann displacement.  This is the only point where different
Neumann shells must be compared.

\begin{proof}[Proof of Theorem \ref{thm:global_robin_complexity}]
Fix $\sigma>0$ and a Robin eigenvalue $\Lambda$.  If a modal class
represented by $(m,n)$ contributes to $E_{\Lambda,\sigma}$, then
$\Lambda=\Lambda_{m,n}(\sigma)$.  By
Proposition \ref{prop:robin_neumann_gap} and
\eqref{eq:neumann_limit_eigenvalue},
\begin{equation}
    \Lambda-\frac{4\sigma}{r}
    <
    \frac{4\pi^2}{27r^2}\,\mathcal Q(m,n)
    <
    \Lambda.
\label{eq:cross_shell_window}
\end{equation}
This is precisely the one-sided displacement hypothesis of
Corollary \ref{cor:abstract_bounded_displacement} for the reference levels
\[
    E_Q=\frac{4\pi^2}{27r^2}Q,
    \qquad D_\sigma^{\rm disp}=\frac{4\sigma}{r}.
\]
By \eqref{eq:triangle_reference_density}, their local counting function
satisfies
\begin{equation}
    N_\sigma
    :=\mathcal N_0(D_\sigma^{\rm disp})
    \leq
    \left\lfloor\frac{27r\sigma}{\pi^2}\right\rfloor+1.
\label{eq:N_sigma_shell_count}
\end{equation}
Equivalently, after rescaling by the Neumann shell spacing, the admissible
integer shell values lie in an interval of length
\begin{equation}
    D_\sigma:=\frac{27r\sigma}{\pi^2}.
\label{eq:D_sigma_shell_window}
\end{equation}
Corollary \ref{cor:within_shell_coincidence_bound} provides the one-shell
bound $B_\sigma$, so Corollary \ref{cor:abstract_bounded_displacement}
gives
\begin{equation}
    \kappa_\sigma(\Lambda)
    \leq B_\sigma N_\sigma,
\label{eq:kappa_global_bound}
\end{equation}
and taking the supremum over $\Lambda$ gives
\begin{equation}
    K_\sigma
    \leq
    B_\sigma
    \left(
       \left\lfloor\frac{27r\sigma}{\pi^2}\right\rfloor+1
    \right)
    <\infty.
\label{eq:Ksigma_explicit_bound}
\end{equation}
This proves the theorem.
\end{proof}

\begin{proof}[Proof of Theorem \ref{thm:fixed_robin}]
Fix $\sigma>0$.  Theorem \ref{thm:global_robin_complexity} gives
$K_\sigma<\infty$, while Lemma
\ref{lem:robin-modal-class-complexity} gives
\[
    \mathfrak F_T\bigl(\mathcal R_{\mathfrak c}(\sigma)\bigr)
    \leq N_{\mathrm{mod}}
\]
for every desymmetrized Robin modal class.  Applying Corollary
\ref{cor:modal_complexity_observation} to the complete modal
decomposition
\eqref{eq:fixed-robin-complete-modal-decomposition} yields
\begin{equation}
    \mathfrak F_T(E_{\Lambda,\sigma})
    \leq
    N_{\mathrm{mod}}\kappa_\sigma(\Lambda)
    \leq N_{\mathrm{mod}}K_\sigma
\label{eq:fixed_robin_exp_complexity}
\end{equation}
for every Robin eigenvalue $\Lambda$, and hence
\[
    \|u\|_{L^2(V)}^2
    \geq
    \gamma_{T,V,N_{\mathrm{mod}}K_\sigma}
    \|u\|_{L^2(T)}^2
\]
for every $u\in E_{\Lambda,\sigma}$.  Therefore
\[
    c_V(\sigma)
    \geq
    \gamma_{T,V,N_{\mathrm{mod}}K_\sigma}>0,
\]
which proves Theorem \ref{thm:fixed_robin}.
\end{proof}

\subsection{Small-Robin multiplicity structure}
\label{subsec:small_robin_multiplicity}

For sufficiently small positive Robin parameter, the multiplicity
structure is determined by the symmetry of the triangle.  Rudnick and
Wigman~\cite[Theorem~1.5]{RudnickWigman2022} proved that, after
desymmetrization, the Robin spectrum is simple.

\begin{proposition}[Small-Robin modal separation]
\label{prop:small_robin_modal_separation}
There exists $\sigma_{\mathrm{sep}}>0$ such that, for every
$0<\sigma<\sigma_{\mathrm{sep}}$, distinct desymmetrized Robin index
classes have distinct eigenvalues.  Consequently, each Robin
eigenvalue is associated with a unique desymmetrized modal class.
\end{proposition}

\begin{lemma}[Complete eigenspaces for small Robin parameter]
\label{lem:small-robin-complete-eigenspace}
Let $0<\sigma<\sigma_{\mathrm{sep}}$ and
$\Lambda\in\operatorname{Spec}(-\Delta_\sigma)$.  Then, there is a
unique desymmetrized index class, represented by an admissible pair
$0\leq m\leq n$, such that
\[
    \Lambda=\Lambda_{m,n}(\sigma).
\]
The corresponding complete eigenspace is
\begin{equation}
    E_{\Lambda,\sigma}
    =
    \mathcal R_{m,n}(\sigma)
    =
    \begin{cases}
    \operatorname{span}\{T_{m,m,\sigma}^{s}\},
        & m=n,\\[1mm]
    \operatorname{span}
    \{T_{m,n,\sigma}^{s},T_{m,n,\sigma}^{a}\},
        & m<n.
    \end{cases}
\label{eq:complete-small-robin-eigenspace}
\end{equation}
\end{lemma}

\begin{proof}
McCartin's construction gives a complete orthogonal system consisting
of the symmetric modes $T_{m,n,\sigma}^{s}$ for $0\leq m\leq n$ and
the antisymmetric modes $T_{m,n,\sigma}^{a}$ for $0\leq m<n$; for
$m=n$ the antisymmetric mode vanishes.  By
Proposition~\ref{prop:small_robin_modal_separation}, a Robin eigenvalue
with $0<\sigma<\sigma_{\mathrm{sep}}$ belongs to a unique
desymmetrized index class.  Completeness of McCartin's system then
gives \eqref{eq:complete-small-robin-eigenspace}.
\end{proof}

\begin{remark}
\label{rem:neumann_endpoint}
Proposition~\ref{prop:small_robin_modal_separation} applies only for
$\sigma>0$.  At $\sigma=0$, the Neumann spectrum may have unbounded
arithmetic multiplicity.  Uniform observation on
$[0,\sigma_0]$ therefore uses the Neumann observation theorem at the
endpoint rather than a uniform multiplicity bound.
\end{remark}

\subsection{Frequency perturbations}
\label{subsec:frequency_perturbation}

For each Robin modal class, let
\[
    \Xi_{m,n}(\sigma)
    =
    \{
        \xi_q(m,n,\sigma)
    \}_{q=1}^{N_0}
\]
denote the finite collection of frequencies in
(\ref{eq:robin_finite_exponential}), and let $
    \Xi_{m,n}(0)$ 
be the corresponding Neumann frequency orbit.

The secular equations imply that the Robin phases remain bounded in
their fixed intervals, as shown in (\ref{eq:LMN_ranges}). Consequently,
\begin{equation}
    \mu-m=O(1),
    \qquad
    \nu-n=O(1),
\label{eq:rough_parameter_bound}
\end{equation}
uniformly in the spectral indices for $\sigma$ in a fixed compact
interval. For the stability theorem, however, we require a stronger
estimate in the small-$\sigma$ regime.

\begin{lemma}[Uniform modal-frequency perturbation]
\label{lem:uniform_frequency_perturbation}
There exist $\sigma_*>0$ and $C>0$ such that, for every admissible
index pair $(m,n)$ and every
\[
    0\leq\sigma\leq\sigma_*,
\]
the Robin frequencies can be labeled so that
\begin{equation}
    \max_{1\leq q\leq N_0}
    \left|
        \xi_q(m,n,\sigma)
        -
        \xi_q(m,n,0)
    \right|
    \leq
    C\sqrt{\sigma}.
\label{eq:uniform_frequency_shift}
\end{equation}
The constant $C$ is independent of $(m,n)$. If $m\geq1$, the stronger
estimate
\begin{equation}
    \max_{1\leq q\leq N_0}
    \left|
        \xi_q(m,n,\sigma)
        -
        \xi_q(m,n,0)
    \right|
    \leq
    C\sigma
\label{eq:bulk_frequency_shift}
\end{equation}
holds uniformly in $m,n$.
\end{lemma}

\begin{remark}
\label{rem:key_robin_lemma}
Lemma~\ref{lem:uniform_frequency_perturbation} is substantially
stronger than convergence of each fixed Robin mode to its Neumann
counterpart. Its content is precisely the uniformity with respect to
the spectral indices. This is one of the key estimates needed for the
small-Robin theorem.
\end{remark}

\subsection{A uniform modal stability lemma}
\label{subsec:modal_stability}

Frequency control alone does not immediately give an observation
estimate, since the coefficients and phases in
\cref{eq:robin_finite_exponential} also depend on $\sigma$. We
therefore isolate the functional statement required in the proof.

Let
$
    \mathcal R_{m,n}(\sigma)
$
denote the Robin modal space generated by the symmetry-related
functions associated with the index class $(m,n)$, and let
$\mathcal N_{m,n}
    =
    \mathcal R_{m,n}(0)
$
denote its Neumann limit.
For two closed subspaces $E,F\subset L^2(T)$, define
\begin{equation}
    d(E,F)
    =
    \max
    \left\{
        \sup_{\substack{u\in E\\ \|u\|_2=1}}
        \operatorname{dist}(u,F),
        \;
        \sup_{\substack{v\in F\\ \|v\|_2=1}}
        \operatorname{dist}(v,E)
    \right\}.
\label{eq:subspace_gap}
\end{equation}

\begin{lemma}[Uniform modal-space stability]
\label{lem:uniform_modal_stability}
There exist $\sigma_0>0$ and a function
\[
    \varepsilon:[0,\sigma_0]\to[0,\infty),
    \qquad
    \varepsilon(\sigma)\longrightarrow0
    \quad\text{as }\sigma\to0,
\]
such that
\begin{equation}
    \sup_{m,n}
    d\bigl(
        \mathcal R_{m,n}(\sigma),
        \mathcal N_{m,n}
    \bigr)
    \leq
    \varepsilon(\sigma)
\label{eq:uniform_modal_gap}
\end{equation}
for every $0\leq\sigma\leq\sigma_0$.
\end{lemma}

The proof of Lemma 
\ref{lem:uniform_modal_stability} is deferred to
Section \ref{subsec:proof_modal_stability}. It is the main technical step in
the perturbative Robin analysis.

\subsection{Transfer of observation inequalities}
\label{subsec:observation_transfer}

We next consider the following lemma.

\begin{lemma}[Stability of an observation lower bound]
\label{lem:observation_stability}
Let $V\subset T$ be measurable, and let $E,F\subset L^2(T)$ be
finite-dimensional subspaces. Suppose that
\begin{equation}
    \|v\|_{L^2(V)}^2
    \geq
    c\,\|v\|_{L^2(T)}^2
    \qquad
    \text{, for all }v\in F,
\label{eq:F_observation}
\end{equation}
for some $c>0$. If
\[
    d(E,F)\leq\delta<1,
\]
then every $u\in E$ satisfies
\begin{equation}
    \|u\|_{L^2(V)}
    \geq
    \bigl[\sqrt{c}(1-\delta)-\delta\bigr]
    \|u\|_{L^2(T)}.
\label{eq:safe_stability_bound}
\end{equation}
In particular, the lower bound is positive whenever
\begin{equation}
    \delta<\frac{\sqrt{c}}{1+\sqrt{c}}.
\label{eq:stability_positive_condition}
\end{equation}
\end{lemma}

\begin{proof}
By homogeneity it suffices to consider $\|u\|_{L^2(T)}=1$. By the
definition of the gap, there exists $v\in F$ such that
\[
    \|u-v\|_{L^2(T)}\leq\delta.
\]
Hence, $\|v\|_{L^2(T)}\geq1-\delta$. Using the contraction
\[
    \|u-v\|_{L^2(V)}\leq\|u-v\|_{L^2(T)}
\]
and \cref{eq:F_observation}, we obtain
\[
\begin{aligned}
    \|u\|_{L^2(V)}
    &\geq \|v\|_{L^2(V)}-\|u-v\|_{L^2(V)}\\
    &\geq \sqrt{c}\,\|v\|_{L^2(T)}-\delta\\
    &\geq \sqrt{c}(1-\delta)-\delta.
\end{aligned}
\]
Rescaling proves Eqs. (\ref{eq:safe_stability_bound}) and (\ref{eq:stability_positive_condition}) follows immediately.
\end{proof}

\subsection{The perturbative argument}
\label{subsec:robin_mechanism}

Let $\mathcal N_{m,n}$ denote the Neumann modal space associated with
the index class $(m,n)$.  By Theorem~\ref{thm:neumann_unl}, for every
measurable $V\subset T$ with $|V|>0$ there exists $c_V>0$,
independent of $(m,n)$, such that
\begin{equation}
    \|v\|_{L^2(V)}^2
    \geq
    c_V\|v\|_{L^2(T)}^2,
    \qquad
    v\in\mathcal N_{m,n}.
\label{eq:neumann_modal_obs}
\end{equation}
Lemma~\ref{lem:uniform_modal_stability} compares
$\mathcal R_{m,n}(\sigma)$ with $\mathcal N_{m,n}$ uniformly in the
modal indices.  The observation estimate is therefore stable for
sufficiently small $\sigma$: there exist $\sigma_0>0$ and $c_V'>0$
such that
\begin{equation}
    \|u\|_{L^2(V)}^2
    \geq
    c_V'\|u\|_{L^2(T)}^2,
    \qquad
    u\in\mathcal R_{m,n}(\sigma),
    \quad
    0\leq\sigma\leq\sigma_0,
\label{eq:robin_modal_obs}
\end{equation}
uniformly in $(m,n)$.

\begin{proposition}[Small-Robin observation]
\label{prop:small_robin_reduction}
There exists $\sigma_0>0$ such that
\begin{equation}
    \inf_{0<\sigma\leq\sigma_0}
    \inf_{\Lambda\in\operatorname{Spec}(-\Delta_\sigma)}
    \inf_{0\neq u\in E_{\Lambda,\sigma}}
    \frac{\|u\|_{L^2(V)}^2}
         {\|u\|_{L^2(T)}^2}
    >0.
\label{eq:small_robin_positive_sigma}
\end{equation}
Together with the Neumann estimate at $\sigma=0$, this proves
Theorem~\ref{thm:small_robin}.
\end{proposition}

\begin{proof}
Choose $\sigma_0>0$ smaller than the separation threshold in
Proposition~\ref{prop:small_robin_modal_separation} and sufficiently
small for \eqref{eq:robin_modal_obs} to hold.  For
$0<\sigma\leq\sigma_0$, Lemma~\ref{lem:small-robin-complete-eigenspace}
identifies every complete Robin eigenspace with a single
desymmetrized modal space:
\[
    E_{\Lambda,\sigma}
    =
    \mathcal R_{m,n}(\sigma)
\]
for the unique index class associated with $\Lambda$.  Applying
\eqref{eq:robin_modal_obs} gives
\eqref{eq:small_robin_positive_sigma}.  The case $\sigma=0$ follows
from Theorem~\ref{thm:neumann_unl}.
\end{proof}

\subsection{Uniform estimates from the secular system}
\label{subsec:proof_modal_stability}

It remains to prove the uniform modal-space convergence required in
Lemma~\ref{lem:uniform_modal_stability}:
\begin{equation}
    \sup_{m,n}
    d\bigl(
        \mathcal R_{m,n}(\sigma),
        \mathcal N_{m,n}
    \bigr)
    \longrightarrow0
    \qquad
    \text{as }\sigma\downarrow0.
\label{eq:central_uniform_limit}
\end{equation}
In particular, convergence for each fixed pair $(m,n)$ is not
sufficient.

The secular equations
\eqref{eq:secular_L}--\eqref{eq:secular_N} give estimates for
$L,M,N$ that are uniform in the modal indices.  These estimates control
the displacements $\mu-m$ and $\nu-n$ and, through McCartin's
representation \eqref{eq:mccartin-modal-formula}, the corresponding
Robin modal functions.  We begin with the secular parameters.

\subsection{Uniform estimates for the secular parameters}
\label{subsec:uniform-secular-estimates}

Set
\begin{equation}
    h:=3r\sigma,
\label{eq:h-robin}
\end{equation}
and
\begin{equation}
    a:=-L,\qquad b:=M,\qquad c:=N.
\label{eq:abc-secular}
\end{equation}
Then, $0\leq a,b,c<\pi/2$, and
\eqref{eq:secular_L}--\eqref{eq:secular_N} become
\begin{align}
    \bigl(2a+b+c+(m+n)\pi\bigr)\tan a &= h,
\label{eq:secular-a}\\
    \bigl(2b-c+a+m\pi\bigr)\tan b &= h,
\label{eq:secular-b}\\
    \bigl(2c+a-b+n\pi\bigr)\tan c &= h.
\label{eq:secular-c}
\end{align}

\begin{lemma}[Uniform secular-parameter bounds]
\label{lem:uniform-secular-parameters}
For every admissible pair $0\leq m\leq n$ and every $\sigma\geq0$,
\begin{equation}
    0\leq a\leq\sqrt{\frac h2},
    \qquad
    0\leq b,c\leq\sqrt h.
\label{eq:abc-sqrt-bound}
\end{equation}
Consequently,
\begin{equation}
    |L|+|M|+|N|
    \leq
    \left(2+\frac1{\sqrt2}\right)\sqrt h,
\label{eq:LMN-uniform-bound}
\end{equation}
and
\begin{align}
    |\mu-m| &\leq \frac4\pi\sqrt h,
\label{eq:mu-uniform-shift}\\
    |\nu-n| &\leq \frac4\pi\sqrt h,
\label{eq:nu-uniform-shift}\\
    |\ell+(m+n)| &\leq \frac8\pi\sqrt h.
\label{eq:ell-uniform-shift}
\end{align}
The bounds are uniform in the modal indices.
\end{lemma}

\begin{proof}
Since $\tan t\geq t$ on $[0,\pi/2)$,
\eqref{eq:secular-a} gives
\[
    h\geq2a\tan a\geq2a^2,
\]
and hence, $a\leq\sqrt{h/2}$.
Set $d=\max\{b,c\}$.  If $d=b$, then $c\leq b$ and
\[
    2b-c+a+m\pi
    =
    b+(b-c)+a+m\pi
    \geq b.
\]
It follows from \eqref{eq:secular-b} that
\[
    h
    \geq b\tan b
    \geq b^2.
\]
If $d=c$, then $b\leq c$ and similarly
\[
    2c+a-b+n\pi
    =
    c+(c-b)+a+n\pi
    \geq c,
\]
so that
\[
    h
    \geq c\tan c
    \geq c^2
\]
by \eqref{eq:secular-c}.  Thus, $b,c\leq\sqrt h$, proving
\eqref{eq:abc-sqrt-bound}, and
\eqref{eq:LMN-uniform-bound} follows.\\
By \eqref{eq:mu_nu},
\[
    \mu-m=\frac{2b-c+a}{\pi},
    \qquad
    \nu-n=\frac{2c+a-b}{\pi}.
\]
Therefore,
\[
    |\mu-m|
    \leq
    \frac{2b+c+a}{\pi}
    \leq
    \frac4\pi\sqrt h,
\]
and the same estimate holds for $|\nu-n|$.  Since
$\ell=-\mu-\nu$,
\[
    |\ell+(m+n)|
    \leq
    |\mu-m|+|\nu-n|
    \leq
    \frac8\pi\sqrt h.
\]
\end{proof}

The square-root bound is uniform over all modal indices, including the
edge family $m=0$.  Away from this family the secular parameters satisfy
the stronger linear estimate below.

\begin{lemma}[Linear secular bounds away from the edge]
\label{lem:bulk-secular-linear}
If $m\geq1$, then
\begin{equation}
    |L|+M+N\leq C_r\sigma
\label{eq:bulk-LMN-linear}
\end{equation}
uniformly in $m,n$.  Consequently,
\begin{equation}
    |\mu-m|+|\nu-n|+|\ell+(m+n)|
    \leq C_r'\sigma.
\label{eq:bulk-parameter-linear}
\end{equation}
\end{lemma}

\begin{proof}
Since $m\geq1$ and $n\geq m$,
\eqref{eq:secular-a} gives
\[
    h
    \geq
    (m+n)\pi\tan a
    \geq
    2\pi a,
\]
and hence, $a\leq h/(2\pi)$.  Moreover, since $c<\pi/2$,
\[
    2b-c+a+m\pi
    \geq
    m\pi-c
    \geq
    \frac{\pi}{2}.
\]
Thus, \eqref{eq:secular-b} implies
\[
    h\geq\frac{\pi}{2}\tan b
    \geq\frac{\pi}{2}b,
\]
so that $b\leq2h/\pi$.  Similarly,
$c\leq2h/\pi$.  Since $h=3r\sigma$,
\eqref{eq:bulk-LMN-linear} follows, and
\eqref{eq:bulk-parameter-linear} follows from
\eqref{eq:mu_nu} and \eqref{eq:ell_parameter}.
\end{proof}

\begin{remark}
\label{rem:sqrt-robin-scale}
The square-root rate in
Lemma~\ref{lem:uniform-secular-parameters} is sharp for the full family
of modal indices.  For $(m,n)=(0,0)$, symmetry gives $b=c$, and the
secular equations reduce to
\[
    2(a+b)\tan a=h,
    \qquad
    (a+b)\tan b=h.
\]
Hence,
\[
    \tan b=2\tan a.
\]
As $h\downarrow0$,
\[
    b=2a+O(a^3).
\]
Substitution into the first equation gives
\[
    6a^2+O(a^4)=h,
\]
and therefore,
\begin{equation}
    a\sim\sqrt{\frac h6},
    \qquad
    b=c\sim\sqrt{\frac{2h}{3}}.
\label{eq:ground-branch-sqrt-asymptotics}
\end{equation}
Thus, an $O(\sigma)$ estimate cannot hold uniformly over all modal
indices.
\end{remark}

The phase parameters in McCartin's modal representation are
\begin{equation}
    \delta_1=L-M-N,\qquad
    \delta_2=-L+M-N,\qquad
    \delta_3=-L-M+N.
\label{eq:robin-phase-parameters}
\end{equation}

\begin{corollary}[Uniform Robin phase bounds]
\label{cor:uniform-robin-phases}
For $\sigma$ in a fixed bounded interval, there exists $C_r>0$ such
that
\begin{equation}
    \max_{1\leq j\leq3}|\delta_j|
    \leq C_r\sqrt\sigma
\label{eq:delta-sqrt-bound}
\end{equation}
uniformly in $(m,n)$.
\end{corollary}

\begin{proof}
By \eqref{eq:robin-phase-parameters},
\[
    |\delta_j|\leq |L|+M+N.
\]
The conclusion follows from \eqref{eq:LMN-uniform-bound} and
$h=3r\sigma$.
\end{proof}

\begin{proof}[Proof of Lemma~\ref{lem:uniform_frequency_perturbation}]
Every wave vector in \eqref{eq:robin_finite_exponential} is a fixed
linear function, determined by the geometry of $T$, of
$(\ell,\mu,\nu)$.  At $\sigma=0$ these parameters are
\[
    (-(m+n),m,n).
\]
Lemma~\ref{lem:uniform-secular-parameters} gives
\[
    |\ell+(m+n)|
    +|\mu-m|
    +|\nu-n|
    \leq C_r\sqrt\sigma
\]
uniformly in $(m,n)$.  Since the number of frequencies in each modal
family is uniformly bounded, the corresponding frequency displacement
is $O(\sqrt\sigma)$.  If $m\geq1$,
Lemma~\ref{lem:bulk-secular-linear} gives the stronger
$O(\sigma)$ bound.
\end{proof}

\subsection{Proof of uniform modal-space stability}
\label{subsec:proof-uniform-modal-stability}

\begin{lemma}[Uniform comparison of modal functions]
\label{lem:uniform-modal-function-comparison}
There exist $\sigma_1>0$ and $C>0$ such that, for every admissible
$(m,n)$, every nonzero allowed parity $\varepsilon\in\{s,a\}$, and
$0\leq\sigma\leq\sigma_1$,
\begin{equation}
    \left\|
        T_{m,n,\sigma}^{\varepsilon}
        -T_{m,n,0}^{\varepsilon}
    \right\|_{L^2(T)}
    \leq C\sqrt\sigma.
\label{eq:uniform-mode-L2-comparison}
\end{equation}
The constant is independent of $(m,n)$.
\end{lemma}

\begin{proof}
Use the explicit representation, as described in (\ref{eq:mccartin-modal-formula}).
Every coefficient of $x$ or $y$ occurring in a trigonometric argument
is a fixed linear function of $(\ell,\mu,\nu)$, while the additive
phases are $\delta_1,\delta_2,\delta_3$. By Lemma
\ref{lem:uniform-secular-parameters} and Corollary \ref{cor:uniform-robin-phases}, the
change in each such coefficient and phase from its Neumann value is
$O(\sqrt\sigma)$ uniformly in $(m,n)$. To make the frequency
uniformity explicit, consider the first separated term. Writing
$\ell_0=-(m+n)$, one has uniformly for $(x,y)\in T$,
\begin{align*}
&\left|
\frac{\pi\ell}{3r}(3r-y)-\delta_1
-
\frac{\pi\ell_0}{3r}(3r-y)
\right|\\
&\qquad\leq
\frac{\pi}{3r}|3r-y|\,|\ell-\ell_0|+|\delta_1|
\leq C_T\sqrt\sigma,
\end{align*}
and, writing $\mu_0=m$, $\nu_0=n$,
\[
\left|
\frac{\sqrt3\pi}{9r}
\bigl[(\mu-\nu)-(\mu_0-\nu_0)\bigr]
(x-\sqrt3r)
\right|
\leq C_T\sqrt\sigma.
\]
The remaining four trigonometric arguments are treated identically,
since their frequency coefficients are fixed linear combinations of
$\ell,\mu,\nu$. In particular, no factor depending on the size of the
base frequencies $m,n$ appears in these argument differences. Because
$T$ is bounded, every complete trigonometric argument in
(\ref{eq:mccartin-modal-formula}) therefore changes by
$O(\sqrt\sigma)$ uniformly on $T$ and uniformly in the modal indices.
Using
\[
    |\sin x-\sin y|\leq|x-y|,
    \qquad
    |\cos x-\cos y|\leq|x-y|,
\]
and, for products,
\[
    |f_1f_2-g_1g_2|
    \leq |f_1-g_1|+|f_2-g_2|
\]
whenever $|f_j|,|g_j|\leq1$, each of the three separated terms differs
from its Neumann counterpart by at most $C_T\sqrt\sigma$. Hence,
\[
    \sup_{x\in T}
    \left|T_{m,n,\sigma}^{\varepsilon}(x)
          -T_{m,n,0}^{\varepsilon}(x)\right|
    \leq C_T\sqrt\sigma
\]
uniformly in $(m,n)$ and in the allowed parity. Multiplication by
$|T|^{1/2}$ proves the inequality (\ref{eq:uniform-mode-L2-comparison}).
\end{proof}

\begin{lemma}[Uniform norm and normalization bounds]
\label{lem:uniform-modal-normalization}
Let $T_{m,n,0}^{\varepsilon}$ denote the standard Neumann
representatives obtained from Eq. (\ref{eq:mccartin-modal-formula}) at
$\sigma=0$. Then, every nonzero representative satisfies
\begin{equation}
    \|T_{m,n,0}^{\varepsilon}\|_{L^2(T)}^2
    \geq c_T^2,
    \qquad
    c_T^2:=\frac{9\sqrt3\,r^2}{4}.
\label{eq:uniform-neumann-norm-lower}
\end{equation}
Moreover, after reducing $\sigma_1>0$ if necessary,
\begin{equation}
    \|T_{m,n,\sigma}^{\varepsilon}\|_{L^2(T)}
    \geq \frac{c_T}{2}
\label{eq:uniform-robin-norm-lower}
\end{equation}
for all admissible indices, all nonzero allowed parities, and
$0\leq\sigma\leq\sigma_1$. If
\[
    e_{m,n,\sigma}^{\varepsilon}
    :=
    \frac{T_{m,n,\sigma}^{\varepsilon}}
         {\|T_{m,n,\sigma}^{\varepsilon}\|_{L^2(T)}},
\]
then
\begin{equation}
    \left\|
        e_{m,n,\sigma}^{\varepsilon}
        -e_{m,n,0}^{\varepsilon}
    \right\|_{L^2(T)}
    \leq C_T'\sqrt\sigma
\label{eq:normalized-mode-closeness}
\end{equation}
uniformly in the indices.
\end{lemma}

\begin{proof}
For $m<n$, the standard symmetric and antisymmetric Neumann modes
satisfy
\[
    \|T_{m,n,0}^{s}\|_2^2
    =
    \|T_{m,n,0}^{a}\|_2^2
    =
    \frac{9\sqrt3\,r^2}{4},
\]
while for $m=n>0$,
\[
    \|T_{m,m,0}^{s}\|_2^2
    =
    \frac{9\sqrt3\,r^2}{2};
\]
see McCartin~\cite{McCartin2002Neumann} and the summary in
Rudnick--Wigman~\cite{RudnickWigman2022}.  For $(m,n)=(0,0)$,
\cref{eq:mccartin-modal-formula} reduces to
$T_{0,0,0}^{s}\equiv3$, and since $|T|=3\sqrt3\,r^2$,
\[
    \|T_{0,0,0}^{s}\|_2^2=27\sqrt3\,r^2.
\]
This proves (\ref{eq:uniform-neumann-norm-lower}) for every nonzero
allowed representative.
Let
\[
    \eta_{m,n}^{\varepsilon}(\sigma)
    :=
    \|T_{m,n,\sigma}^{\varepsilon}
      -T_{m,n,0}^{\varepsilon}\|_2.
\]
By Lemma \ref{lem:uniform-modal-function-comparison},
$\eta_{m,n}^{\varepsilon}(\sigma)\leq C\sqrt\sigma$ uniformly. Choose
$\sigma_1$ so that $C\sqrt{\sigma_1}\leq c_T/2$. The reverse triangle
inequality then yields the inequality (\ref{eq:uniform-robin-norm-lower}).\\
For nonzero $f,g$, write instead
\[
\frac{f}{\|f\|_2}-\frac{g}{\|g\|_2}
=
\frac{f-g}{\|f\|_2}
+
g\left(\frac1{\|f\|_2}-\frac1{\|g\|_2}\right).
\]
Since
\[
\|g\|_2
\left|\frac1{\|f\|_2}-\frac1{\|g\|_2}\right|
=
\frac{|\|g\|_2-\|f\|_2|}{\|f\|_2}
\leq
\frac{\|f-g\|_2}{\|f\|_2},
\]
we obtain the denominator-safe estimate
\begin{equation}
\left\|\frac{f}{\|f\|_2}-\frac{g}{\|g\|_2}\right\|_2
\leq
\frac{2\|f-g\|_2}{\|f\|_2}.
\label{eq:normalization-safe}
\end{equation}
Applying this with
$f=T_{m,n,\sigma}^{\varepsilon}$ and
$g=T_{m,n,0}^{\varepsilon}$, and using
\cref{eq:uniform-robin-norm-lower}, gives
\[
    \left\|
        e_{m,n,\sigma}^{\varepsilon}
        -e_{m,n,0}^{\varepsilon}
    \right\|_2
    \leq
    \frac{4C}{c_T}\sqrt\sigma,
\]
which is the inequality (\ref{eq:normalized-mode-closeness}).
\end{proof}

\begin{proof}[Proof of Lemma \ref{lem:uniform_modal_stability}]
For $m=n$, only the symmetric representative is nonzero, so both
$\mathcal R_{m,m}(\sigma)$ and $\mathcal N_{m,m}$ are one-dimensional.
For $m<n$, both modal spaces are two-dimensional, with one symmetric
and one antisymmetric representative.  Reflection $\mathsf R$ across
the altitude is unitary on $L^2(T)$ and satisfies
\[
    \mathsf R e_{m,n,\sigma}^{s}=e_{m,n,\sigma}^{s},
    \qquad
    \mathsf R e_{m,n,\sigma}^{a}=-e_{m,n,\sigma}^{a}.
\]
Therefore,
\[
    \langle e_{m,n,\sigma}^{s},e_{m,n,\sigma}^{a}\rangle
    =
    \langle \mathsf R e_{m,n,\sigma}^{s},
            \mathsf R e_{m,n,\sigma}^{a}\rangle
    =-
    \langle e_{m,n,\sigma}^{s},e_{m,n,\sigma}^{a}\rangle,
\]
so the normalized symmetric and antisymmetric representatives are
orthogonal for every $\sigma$.\\
For a one-dimensional modal space,
the inequality (\ref{eq:normalized-mode-closeness}) immediately bounds both directed
subspace distances.  If $m<n$, write
\[
    u=\alpha_s e_{m,n,\sigma}^{s}
      +\alpha_a e_{m,n,\sigma}^{a},
    \qquad \|u\|_2=1.
\]
Orthogonality gives $|\alpha_s|^2+|\alpha_a|^2=1$. Setting
\[
    v=\alpha_s e_{m,n,0}^{s}
      +\alpha_a e_{m,n,0}^{a}
\]
and using the inequality (\ref{eq:normalized-mode-closeness}) gives
\[
    \|u-v\|_2
    \leq
    (|\alpha_s|+|\alpha_a|)C_T'\sqrt\sigma
    \leq
    \sqrt2\,C_T'\sqrt\sigma.
\]
The reverse comparison, starting with a unit vector in
$\mathcal N_{m,n}$, is identical. Consequently,
\begin{equation}
    \sup_{0\leq m\leq n}
    d\bigl(\mathcal R_{m,n}(\sigma),\mathcal N_{m,n}\bigr)
    \leq C_*\sqrt\sigma,
\label{eq:uniform-modal-gap-sqrt}
\end{equation}
which proves Lemma \ref{lem:uniform_modal_stability} with
$\varepsilon(\sigma)=C_*\sqrt\sigma$.
\end{proof}

\subsection{Completion of the small-Robin theorem}
\label{subsec:completion-small-robin}

\begin{proof}[Proof of Theorem \ref{thm:small_robin}]
Let $c_{V,N}>0$ denote the Neumann observation constant from
Theorem \ref{thm:neumann_unl}. By the inequality (\ref{eq:uniform-modal-gap-sqrt}), after
possibly reducing the modal-separation threshold, there exists
$C_*>0$ such that
\[
    \sup_{m,n}
    d\bigl(\mathcal R_{m,n}(\sigma),\mathcal N_{m,n}\bigr)
    \leq C_*\sqrt\sigma.
\]
Choose $\sigma_0>0$ below the threshold in
Proposition \ref{prop:small_robin_modal_separation} and so small that
\begin{equation}
    C_*\sqrt{\sigma_0}
    \leq
    \frac{\sqrt{c_{V,N}}}{2(1+\sqrt{c_{V,N}})}.
\label{eq:sigma0-choice}
\end{equation}
For $0<\sigma\leq\sigma_0$,
Lemma \ref{lem:small-robin-complete-eigenspace} shows that every Robin
eigenspace coincides with the complete McCartin modal space associated
with a unique desymmetrized index class. Its Neumann limiting modal
space is contained in a Neumann eigenspace and therefore satisfies
\[
    \|v\|_{L^2(V)}^2\geq c_{V,N}\|v\|_{L^2(T)}^2.
\]
Applying Lemma \ref{lem:observation_stability} with
$\delta=C_*\sqrt\sigma$ and the inequality (\ref{eq:sigma0-choice}) gives
\[
    \|u\|_{L^2(V)}
    \geq\frac12\sqrt{c_{V,N}}\,\|u\|_{L^2(T)}.
\]
Thus,
\begin{equation}
    \|u\|_{L^2(V)}^2
    \geq\frac{c_{V,N}}4\|u\|_{L^2(T)}^2
\label{eq:explicit-small-robin-bound}
\end{equation}
uniformly for $0<\sigma\leq\sigma_0$. The endpoint $\sigma=0$ follows
directly from Theorem \ref{thm:neumann_unl}, completing the proof.
\end{proof}


\section{The Robin problem on the square}
\label{sec:square_robin}

We next consider the unit square
\[
    S=(0,1)^2.
\]
Separation of variables reduces the Robin spectrum to the
one-dimensional Robin frequencies.  For fixed positive Robin
parameter, their high-frequency expansion gives a strict ordering
within the sum-of-two-squares shells.

\subsection{Separated Robin frequencies}
\label{subsec:square_1d_expansion}

Let $k_n(\sigma)$ denote the nonnegative one-dimensional Robin
frequency on $(0,1)$ associated with the Neumann branch
$n\in\mathbb Z_{\geq0}$.  For $n\geq1$, it is the solution near
$n\pi$ of
\begin{equation}
    \tan k
    =
    \frac{2\sigma k}{k^2-\sigma^2}.
\label{eq:square_1d_secular}
\end{equation}
The corresponding eigenvalues on $S$ are
\begin{equation}
    \Lambda_{n,m}^{\square}(\sigma)
    =
    k_n(\sigma)^2+k_m(\sigma)^2,
    \qquad
    0\leq n\leq m.
\label{eq:square_eigenvalues}
\end{equation}
The exchange $(n,m)\leftrightarrow(m,n)$ is accounted for by regarding
the unordered pair $\{n,m\}$ as a single modal class.

\begin{lemma}[Fixed-$\sigma$ expansion]
\label{lem:square_fixed_sigma_expansion}
Fix $\sigma>0$.  As $n\to\infty$,
\begin{equation}
\begin{split}
    k_n(\sigma)
    ={}&
    n\pi+\frac{2\sigma}{n\pi}
    -\frac{2\sigma^2(\sigma+6)}{3(n\pi)^3}\\
    &+
    \frac{2\sigma^3(3\sigma^2+40\sigma+120)}
         {15(n\pi)^5}
    +O_\sigma(n^{-7}),
\end{split}
\label{eq:square_kn_expansion}
\end{equation}
and
\begin{equation}
    k_n(\sigma)^2
    =
    \pi^2n^2+4\sigma
    -\frac{A_\sigma^{\square}}{n^2}
    +\frac{B_\sigma^{\square}}{n^4}
    +O_\sigma(n^{-6}),
\label{eq:square_kn2_expansion}
\end{equation}
where
\begin{equation}
    A_\sigma^{\square}
    =
    \frac{4\sigma^2(\sigma+3)}{3\pi^2}>0,
    \qquad
    B_\sigma^{\square}
    =
    \frac{4\sigma^3(\sigma^2+10\sigma+20)}{5\pi^4}>0.
\label{eq:square_AB}
\end{equation}
\end{lemma}

\begin{proof}
Set $N=n\pi$, $t=N^{-1}$, and write
\[
    k_n=N+\delta_n,
    \qquad
    \delta_n=t z(t^2).
\]
Since $\tan(N+\delta_n)=\tan\delta_n$,
\eqref{eq:square_1d_secular} becomes
\begin{equation}
    \frac{\tan(tz)}{t}
    =
    \frac{2\sigma(1+t^2z)}
         {(1+t^2z)^2-\sigma^2t^2}.
\label{eq:square_scaled_secular}
\end{equation}
With $\rho=t^2$, both sides are analytic near
$(z,\rho)=(2\sigma,0)$, and the derivative with respect to $z$ of
their difference is equal to one at this point.  The analytic implicit
function theorem gives
\[
    z(\rho)
    =
    a_1+a_3\rho+a_5\rho^2+O_\sigma(\rho^3).
\]
Substitution into \eqref{eq:square_scaled_secular} gives
\[
    a_1=2\sigma,
    \qquad
    a_3=-\frac{2\sigma^2(\sigma+6)}3,
    \qquad
    a_5=
    \frac{2\sigma^3(3\sigma^2+40\sigma+120)}{15}.
\]
This proves \eqref{eq:square_kn_expansion}; squaring gives
\eqref{eq:square_kn2_expansion} and \eqref{eq:square_AB}.
\end{proof}
Consider a Neumann shell
\begin{equation}
    n^2+m^2=R^2.
\label{eq:square_neumann_shell}
\end{equation}
For $n\geq1$, set
\begin{equation}
    F_{\square}(n,m)
    =
    \frac1{n^2}+\frac1{m^2},
    \qquad
    G_{\square}(n,m)
    =
    \frac1{n^4}+\frac1{m^4}.
\label{eq:square_FG}
\end{equation}
Lemma~\ref{lem:square_fixed_sigma_expansion} gives
\begin{equation}
\begin{split}
    \Lambda_{n,m}^{\square}(\sigma)
    ={}&
    \pi^2R^2+8\sigma
    -A_\sigma^{\square}F_{\square}(n,m)
    +B_\sigma^{\square}G_{\square}(n,m)\\
    &+
    O_\sigma(n^{-6}+m^{-6}).
\end{split}
\label{eq:square_shell_expansion}
\end{equation}
In particular, the remainder is $O_\sigma(n^{-6})$ uniformly for
$m\geq n\to\infty$.  On each shell,
\begin{equation}
    G_{\square}
    =
    F_{\square}^2-\frac{2}{R^2}F_{\square}.
\label{eq:square_G_as_F}
\end{equation}
Thus, the first two nonconstant terms in
\eqref{eq:square_shell_expansion} are functions of the same shell
statistic $F_{\square}$.

\subsection{Same-shell splitting}
\label{subsec:square_shell_splitting}

\begin{proposition}[Same-shell splitting on the square]
\label{prop:square_same_shell_splitting}
For every fixed $\sigma>0$, there exists
$N_\sigma^{\square}\geq1$ such that, whenever
\[
    n^2+m^2={n'}^2+{m'}^2,
    \qquad
    N_\sigma^{\square}\leq n<n'\leq m'<m,
\]
one has
\begin{equation}
    \Lambda_{n,m}^{\square}(\sigma)
    <
    \Lambda_{n',m'}^{\square}(\sigma).
\label{eq:square_eventual_order}
\end{equation}
Consequently, on each Neumann shell a Robin spectral value occurs in
at most
\begin{equation}
    B_\sigma^{\square}:=N_\sigma^{\square}+1
\label{eq:square_one_shell_bound}
\end{equation}
unordered modal classes.
\end{proposition}

\begin{proof}
Consider two ordered representations of the same shell and set
\[
    P=n^2m^2,
    \qquad
    P'={n'}^2{m'}^2.
\]
Since
\[
    m^2=R^2-n^2,
    \qquad
    {m'}^2=R^2-{n'}^2,
\]
we have
\begin{equation}
    P'-P
    =
    ({n'}^2-n^2)(m^2-{n'}^2)>0.
\label{eq:square_product_gap}
\end{equation}
Moreover,
\[
    F_{\square}(n,m)=\frac{R^2}{P},
\]
and hence,
\begin{equation}
    \Delta F
    :=
    F_{\square}(n,m)-F_{\square}(n',m')
    >0.
\label{eq:square_DeltaF_positive}
\end{equation}
By \eqref{eq:square_G_as_F},
\begin{equation}
    \Delta G
    =
    \Delta F
    \left(
        F_{\square}(n,m)
        +F_{\square}(n',m')
        -\frac{2}{R^2}
    \right).
\label{eq:square_DeltaG_factor}
\end{equation}
Subtracting the two expansions in
\eqref{eq:square_shell_expansion} gives
\begin{equation}
\begin{split}
    \Lambda_{n',m'}^{\square}(\sigma)
    -\Lambda_{n,m}^{\square}(\sigma)
    ={}&
    \Delta F
    \bigg[
        A_\sigma^{\square}
        -
        B_\sigma^{\square}
        \left(
            F_{\square}(n,m)
            +F_{\square}(n',m')
            -\frac{2}{R^2}
        \right)
    \bigg] \\
    &+O_\sigma(n^{-6}).
\end{split}
\label{eq:square_difference_expansion}
\end{equation}
Since $m,m'\geq n$,
\[
    F_{\square}(n,m)
    +F_{\square}(n',m')
    =O(n^{-2}),
\]
so the expression in brackets is
$A_\sigma^{\square}+O_\sigma(n^{-2})$ and is bounded below by
$A_\sigma^{\square}/2$ for all sufficiently large $n$.\\
We next estimate $\Delta F$.  Since
\[
    \Delta F
    =
    \frac{R^2(P'-P)}{PP'},
\]
suppose first that $n\geq\delta R$ for some fixed
$\delta>0$.  Then $P,P'=O(R^4)$, while
\eqref{eq:square_product_gap} and the integrality of the indices give
$P'-P\geq1$.  Hence
\begin{equation}
    \Delta F\geq c_\delta R^{-6}.
\label{eq:square_bulk_gap}
\end{equation}
For the comparison with the remainder, however, one uses the stronger
spacing furnished by \eqref{eq:square_product_gap}.  Since
$n'<m$ and all indices are integers,
\[
    {n'}^2-n^2\geq 2n+1,
    \qquad
    m^2-{n'}^2\geq 2n'+1.
\]
Thus, in the bulk,
\begin{equation}
    P'-P\geq c_\delta R^2,
    \qquad
    \Delta F\geq c_\delta R^{-4}.
\label{eq:square_bulk_gap_strong}
\end{equation}

Suppose next that $n<\delta R$, with $\delta>0$ chosen sufficiently
small.  If $n'\geq2n$, then
\[
    \Delta F\geq c n^{-2}.
\]
If $n<n'<2n$, then
$m^2-{n'}^2\geq cR^2$, and
\eqref{eq:square_product_gap} gives
\begin{equation}
    \Delta F\geq c n^{-3}.
\label{eq:square_edge_gap}
\end{equation}
In both cases the remainder
$O_\sigma(n^{-6})$ in
\eqref{eq:square_difference_expansion} is of lower order than
$\Delta F$.  After increasing $N_\sigma^{\square}$ if necessary,
\eqref{eq:square_eventual_order} follows.

It remains to count the indices below this threshold.  On a fixed
shell, each $0\leq n<N_\sigma^{\square}$ determines at most one
$m\geq n$.  Hence there are at most $N_\sigma^{\square}$ such modal
classes.  Among the remaining classes,
\eqref{eq:square_eventual_order} allows at most one occurrence of a
fixed Robin spectral value.  This proves
\eqref{eq:square_one_shell_bound}.
\end{proof}

\subsection{Global coincidence bound and observation}
\label{subsec:square_global_closure}

For fixed $\sigma>0$, set
\begin{equation}
    d_n(\sigma)
    :=
    k_n(\sigma)^2-\pi^2n^2.
\label{eq:square_dn}
\end{equation}
The one-dimensional Robin eigenvalue branches are strictly increasing
in $\sigma$.  Indeed, the Hellmann--Feynman formula gives
\[
    \frac{d}{d\sigma}k_n(\sigma)^2
    =
    |\phi_n(0)|^2+|\phi_n(1)|^2>0
\]
for an $L^2$-normalized eigenfunction $\phi_n$.  Hence
$d_n(\sigma)>0$ for $\sigma>0$.  Moreover,
Lemma~\ref{lem:square_fixed_sigma_expansion} gives
\[
    d_n(\sigma)\longrightarrow4\sigma
    \qquad (n\to\infty).
\]
It follows that
\begin{equation}
    D_\sigma^{(1)}
    :=
    \sup_{n\geq0}d_n(\sigma)
    <\infty.
\label{eq:square_1d_displacement_bound}
\end{equation}
Thus,
\begin{equation}
    0<
    \Lambda_{n,m}^{\square}(\sigma)
    -\pi^2(n^2+m^2)
    <
    D_\sigma^{\square},
    \qquad
    D_\sigma^{\square}
    :=
    2D_\sigma^{(1)}+1.
\label{eq:square_2d_displacement}
\end{equation}
The Neumann reference levels have the form
\[
    E_Q=\pi^2Q,
\]
with $Q$ representable as a sum of two squares.  In particular,
\begin{equation}
    \mathcal N_{0,\square}(L)
    \leq
    \left\lfloor\frac{L}{\pi^2}\right\rfloor+1.
\label{eq:square_reference_density}
\end{equation}
Corollary~\ref{cor:abstract_bounded_displacement}, together with
\eqref{eq:square_one_shell_bound} and
\eqref{eq:square_2d_displacement}, now gives
\begin{equation}
\begin{split}
    K_\sigma^{\square}
    &\leq
    B_\sigma^{\square}
    \mathcal N_{0,\square}(D_\sigma^{\square})\\
    &\leq
    (N_\sigma^{\square}+1)
    \left(
        \left\lfloor
            \frac{D_\sigma^{\square}}{\pi^2}
        \right\rfloor+1
    \right)
    <\infty.
\end{split}
\label{eq:square_global_K_bound}
\end{equation}
A one-dimensional Robin eigenfunction is a linear combination of
$e^{ik_nx}$ and $e^{-ik_nx}$.  Hence
$\phi_n(x)\phi_m(y)$ is a sum of at most four plane waves.  If $n<m$,
the unordered modal class also contains the exchanged product
$\phi_m(x)\phi_n(y)$, and therefore
\begin{equation}
    \mathfrak F_S
    \bigl(
        \mathcal R_{\{n,m\}}^{\square}(\sigma)
    \bigr)
    \leq8.
\label{eq:square_modal_fourier_complexity}
\end{equation}
Since at most $K_\sigma^{\square}$ unordered modal classes contribute
to a single Robin eigenvalue,
\begin{equation}
    \mathfrak F_S(E_{\Lambda,\sigma}^{\square})
    \leq
    8K_\sigma^{\square}.
\label{eq:square_eigenspace_fourier_complexity}
\end{equation}
Corollary~\ref{cor:modal_complexity_observation} then gives, for every
measurable $V\subset S$ with $|V|>0$,
\begin{equation}
    \|u\|_{L^2(V)}^2
    \geq
    \gamma_{S,V,8K_\sigma^{\square}}
    \|u\|_{L^2(S)}^2,
    \qquad
    u\in E_{\Lambda,\sigma}^{\square}.
\label{eq:square_final_observation}
\end{equation}
This proves Theorem~\ref{thm:square_fixed_robin}.

\begin{remark}[Scaling]
\label{rem:square_scaling}
If $S_\ell=(0,\ell)^2$, dilation to the unit square replaces the Robin
parameter $\sigma$ by $\ell\sigma$.  Hence
Theorem~\ref{thm:square_fixed_robin} holds, with the corresponding
scaled observation constant, on every square of fixed side length.
\end{remark}


\section{The Robin problem on rational rectangles}
\label{sec:rational_rectangles}

Let
\[
    R_{a,b}=(0,a)\times(0,b),
    \qquad
    a,b>0,
    \qquad
    a\ne b,
\]
and impose the Robin condition
\[
    \partial_\nu u+\sigma u=0
    \qquad\text{on }\partial R_{a,b},
\]
with $\sigma>0$ fixed.  Suppose that the squared aspect ratio is
rational:
\begin{equation}
    \frac{a^2}{b^2}
    =
    \frac{p}{q},
    \qquad
    p,q\in\mathbb N,
    \qquad
    (p,q)=1.
\label{eq:rect_rational_ratio}
\end{equation}
Writing
\begin{equation}
    a^2=p\ell^2,
    \qquad
    b^2=q\ell^2
\label{eq:rect_scale_pq}
\end{equation}
for some $\ell>0$, the Neumann eigenvalues are indexed by the weighted
quadratic form
\[
    qn^2+pm^2.
\]
The fixed-$\sigma$ Robin correction will be compared on the
corresponding weighted quadratic shells.
\subsection{Separated frequencies and fixed-parameter expansion}
\label{subsec:rect_frequency_expansion}

Let $k_n(\tau)$ denote the unit-interval Robin frequency used in
Section~\ref{sec:square_robin}.  Scaling the interval to length $a$ shows that
the physical frequency in the first coordinate is $a^{-1}k_n(a\sigma)$.
Lemma~\ref{lem:square_fixed_sigma_expansion}, with the parameter replaced by
$a\sigma$ and the eigenvalue divided by $a^2$, therefore gives
\begin{equation}
 \frac{k_n(a\sigma)^2}{a^2}
 =\frac{\pi^2n^2}{a^2}+\frac{4\sigma}{a}
 -\frac{A_{\sigma,a}}{n^2}
 +\frac{B_{\sigma,a}}{n^4}
 +O_{\sigma,a}(n^{-6}),
\label{eq:rect_1d_expansion_a}
\end{equation}
where
\begin{equation}
 A_{\sigma,a}
 =\frac{4\sigma^2(a\sigma+3)}{3\pi^2}>0,
 \qquad
 B_{\sigma,a}
 =\frac{4a\sigma^3((a\sigma)^2+10a\sigma+20)}{5\pi^4}>0.
\label{eq:rect_AB_a}
\end{equation}
The analogous formulas with $a,n$ replaced by $b,m$ will be denoted by
$A_{\sigma,b}$ and $B_{\sigma,b}$.  Separation of variables yields
\begin{equation}
\begin{split}
 \Lambda_{n,m}^{a,b}(\sigma)
 ={}&\pi^2\left(\frac{n^2}{a^2}+\frac{m^2}{b^2}\right)
 +4\sigma\left(\frac1a+\frac1b\right)\\
 &-\frac{A_{\sigma,a}}{n^2}-\frac{A_{\sigma,b}}{m^2}
 +\frac{B_{\sigma,a}}{n^4}+\frac{B_{\sigma,b}}{m^4}
 +O_{\sigma,a,b}(n^{-6}+m^{-6})
\end{split}
\label{eq:rect_2d_expansion}
\end{equation}
when $n,m\to\infty$.  Indices equal to zero, and more generally a fixed
finite set of low indices, will be absorbed into the exceptional core below.\\
By \eqref{eq:rect_scale_pq}, the Neumann reference energy can be written as
\begin{equation}
 \pi^2\left(\frac{n^2}{a^2}+\frac{m^2}{b^2}\right)
 =\frac{\pi^2}{pq\ell^2}J,
 \qquad
 J:=qn^2+pm^2.
\label{eq:rect_reference_shell}
\end{equation}
Thus, a fixed reference shell is a level set of the positive integral
quadratic form $qn^2+pm^2$.

\subsection{Same-shell splitting}
\label{subsec:rect_stationary_splitting}

Fix a reference shell $J=R^2$ and set
\begin{equation}
    x=\frac{n}{R},
    \qquad
    m^2=\frac{R^2}{p}(1-qx^2),
    \qquad
    0<x<q^{-1/2}.
\label{eq:rect_shell_coordinate}
\end{equation}
The leading nonconstant term in
\eqref{eq:rect_2d_expansion} is $-R^{-2}h(x)$, where
\begin{equation}
    h(x)
    =
    \frac{A_{\sigma,a}}{x^2}
    +
    \frac{pA_{\sigma,b}}{1-qx^2}.
\label{eq:rect_leading_profile}
\end{equation}
The next term is determined by
\begin{equation}
    g(x)
    =
    \frac{B_{\sigma,a}}{x^4}
    +
    \frac{p^2B_{\sigma,b}}{(1-qx^2)^2}.
\label{eq:rect_next_profile}
\end{equation}
On every compact subinterval of $(0,q^{-1/2})$,
\eqref{eq:rect_2d_expansion} takes the form
\begin{equation}
    \Lambda_{n,m}^{a,b}(\sigma)
    =
    E_J+C_{\sigma,a,b}
    -R^{-2}h(x)
    +R^{-4}g(x)
    +O_{\sigma,a,b}(R^{-6}),
\label{eq:rect_bulk_profile_expansion}
\end{equation}
where
\[
    E_J=\frac{\pi^2}{pq\ell^2}J,
    \qquad
    C_{\sigma,a,b}
    =
    4\sigma\left(\frac1a+\frac1b\right).
\]

\begin{lemma}[Convexity of the leading profile]
\label{lem:rect_strong_convexity}
The function $h$ is strictly convex on $(0,q^{-1/2})$ and has a
unique critical point $x_*$, which is its global minimum.  More
precisely,
\begin{equation}
    h''(x)
    =
    \frac{6A_{\sigma,a}}{x^4}
    +
    \frac{2pqA_{\sigma,b}(1+3qx^2)}
         {(1-qx^2)^3}
    >0,
\label{eq:rect_h_second}
\end{equation}
and
\begin{equation}
    qx_*^2
    =
    \frac{\sqrt{A_{\sigma,a}}}
    {\sqrt{A_{\sigma,a}}
     +\sqrt{pA_{\sigma,b}/q}}.
\label{eq:rect_stationary_point}
\end{equation}
Moreover,
\begin{equation}
    \mu_{\sigma,a,b}
    :=
    \inf_{0<x<q^{-1/2}}h''(x)
    >0.
\label{eq:rect_uniform_convexity}
\end{equation}
\end{lemma}

\begin{proof}
Differentiation gives
\[
    h'(x)
    =
    -\frac{2A_{\sigma,a}}{x^3}
    +
    \frac{2pqA_{\sigma,b}x}{(1-qx^2)^2},
\]
and \eqref{eq:rect_h_second} follows by differentiating once more.
The equation $h'(x)=0$ is equivalent to
\[
    pqA_{\sigma,b}x^4
    =
    A_{\sigma,a}(1-qx^2)^2.
\]
Taking positive square roots gives
\eqref{eq:rect_stationary_point}, and hence the critical point is
unique.  Since $h(x)\to+\infty$ at both endpoints, $x_*$ is the
global minimum.  Finally, $h''$ is positive and continuous on
$(0,q^{-1/2})$ and tends to $+\infty$ at both endpoints, which
implies \eqref{eq:rect_uniform_convexity}.
\end{proof}

\begin{proposition}[Same-shell splitting on rational rectangles]
\label{prop:rect_discrete_sector_splitting}
For every fixed $\sigma>0$ and every nonsquare rational rectangle
$R_{a,b}$, there exists $N_{\sigma,a,b}\geq1$ such that, on every
reference shell \eqref{eq:rect_reference_shell}, the Robin eigenvalues
associated with modal pairs satisfying
\[
    \min\{n,m\}\geq N_{\sigma,a,b}
\]
are strictly ordered separately in the two sectors
\begin{equation}
    x\leq x_*,
    \qquad
    x>x_*.
\label{eq:rect_two_sectors}
\end{equation}
In particular, a fixed Robin spectral value is attained by at most two
such modal pairs on a single reference shell.
\end{proposition}

\begin{proof}
Write
\[
    y=y(x):=\frac{m}{R}
    =
    \sqrt{\frac{1-qx^2}{p}},
    \qquad
    qx^2+py^2=1.
\]
Let $y_*=y(x_*)$.  Choose $\eta>0$ sufficiently small that
\begin{equation}
    4\eta<\min\{x_*,y_*\},
\label{eq:rect_eta_choice}
\end{equation}
and such that the edge estimates below hold.  We use the overlapping
regions
\begin{equation}
    \mathcal E_L:=\{x<2\eta\},
    \qquad
    \mathcal B:=\{x\geq\eta,\ y\geq\eta\},
    \qquad
    \mathcal E_R:=\{y<2\eta\}.
\label{eq:rect_three_region_cover}
\end{equation}
They cover the shell; moreover, $x_*\in\mathcal B$,
$\mathcal E_L$ lies to the left of $x_*$, and
$\mathcal E_R$ lies to its right.

\medskip
\noindent\emph{Bulk.}
On $\mathcal B$, \eqref{eq:rect_bulk_profile_expansion} holds with a
uniform $O_{\sigma,a,b,\eta}(R^{-6})$ remainder, and $g'$ is uniformly
bounded.  If $x<x'$ are two admissible shell coordinates lying on the
same side of $x_*$, then
\begin{equation}
    x'-x
    =
    \frac{n'-n}{R}
    \geq R^{-1}.
\label{eq:rect_lattice_spacing}
\end{equation}
By \eqref{eq:rect_uniform_convexity} and $h'(x_*)=0$,
\begin{equation}
    |h(x')-h(x)|
    \geq
    \frac{\mu_{\sigma,a,b}}{2}|x'-x|^2.
\label{eq:rect_discrete_h_gap}
\end{equation}
Indeed, if $x<x'\leq x_*$, then
\[
    -h'(t)
    \geq
    \mu_{\sigma,a,b}(x_*-t),
\]
and therefore,
\[
\begin{split}
    h(x)-h(x')
    &\geq
    \mu_{\sigma,a,b}
    \int_x^{x'}(x_*-t)\,dt\\
    &=
    \mu_{\sigma,a,b}
    \left(
        (x_*-x')(x'-x)
        +\frac{(x'-x)^2}{2}
    \right).
\end{split}
\]
The argument on the right of $x_*$ is identical.  Since $g'$ is
bounded on $\mathcal B$,
\begin{equation}
    |g(x')-g(x)|
    \leq
    C_{\sigma,a,b,\eta}|x'-x|.
\label{eq:rect_g_lipschitz}
\end{equation}
Subtracting the expansions
\eqref{eq:rect_bulk_profile_expansion}, the leading difference has
magnitude at least
\[
    cR^{-2}|x'-x|^2,
\]
while the fourth-order term is bounded by
\[
    CR^{-4}|x'-x|
\]
and the remainder by $CR^{-6}$.  By
\eqref{eq:rect_lattice_spacing}, the ratios of these two errors to the
leading term are $O(R^{-1})$ and $O(R^{-2})$, respectively.  Thus the
Robin eigenvalues are strictly ordered in each sector on
$\mathcal B$ for all sufficiently large $R$.

\medskip
\noindent\emph{Left edge.}
Along a fixed shell, define
\begin{equation}
    H_L(t)
    :=
    \frac{A_{\sigma,a}}{t^2}
    +
    \frac{pA_{\sigma,b}}{R^2-qt^2},
    \qquad
    0<t<2\eta R.
\label{eq:rect_left_H_def}
\end{equation}
Then,
\[
    H_L'(t)
    =
    -\frac{2A_{\sigma,a}}{t^3}
    +
    \frac{2pqA_{\sigma,b}t}
         {(R^2-qt^2)^2}.
\]
The ratio of the second term to the magnitude of the first satisfies
\[
    \frac{pqA_{\sigma,b}}{A_{\sigma,a}}
    \frac{t^4}{(R^2-qt^2)^2}
    \leq
    C_{\sigma,a,b}\eta^4.
\]
After choosing $\eta$ sufficiently small,
\begin{equation}
    H_L'(t)
    \leq
    -\frac{A_{\sigma,a}}{t^3},
    \qquad
    0<t<2\eta R.
\label{eq:rect_left_H_derivative}
\end{equation}
Hence, for integers $N\leq n<n'<2\eta R$,
\begin{equation}
    H_L(n)-H_L(n')
    \geq
    \begin{cases}
        c n^{-2},
            & n'\geq2n,\\
        c n^{-3}(n'-n),
            & n<n'<2n.
    \end{cases}
\label{eq:rect_left_edge_gap}
\end{equation}
The fourth-order term along the left edge is
\[
    K_L(t)
    :=
    \frac{B_{\sigma,a}}{t^4}
    +
    \frac{p^2B_{\sigma,b}}
         {(R^2-qt^2)^2}.
\]
For $t<2\eta R$,
\[
    |K_L'(t)|
    \leq
    C_{\sigma,a,b,\eta}t^{-5}.
\]
Its difference is therefore $O(n^{-4})$ if $n'\geq2n$ and
$O(n^{-5}(n'-n))$ if $n<n'<2n$.  Moreover,
\begin{equation}
    m^2
    =
    \frac{R^2-qn^2}{p}
    \geq
    \frac{1-4q\eta^2}{p}R^2,
\label{eq:rect_left_m_lower}
\end{equation}
so that $m\geq c_{p,q,\eta}R$.  Since $n<2\eta R$,
\[
    m^{-6}\leq C_{p,q,\eta}n^{-6}.
\]
The remainder in \eqref{eq:rect_2d_expansion} is therefore
$O_{\sigma,a,b,\eta}(n^{-6})$ throughout $\mathcal E_L$.  The
fourth-order term and the remainder are lower order than the gaps in
\eqref{eq:rect_left_edge_gap} once $n$ is sufficiently large.  This
gives strict ordering for all modal pairs contained in
$\mathcal E_L$ with sufficiently large indices.

We also compare points separated by the overlap.  Suppose that
$x<\eta$, $x'\geq2\eta$, and both coordinates lie in the left sector.
Then,
\[
    H_L(n)\geq\frac{A_{\sigma,a}}{n^2},
    \qquad
    H_L(n')\leq\frac{C_{\sigma,a,b,\eta}}{R^2}.
\]
Since $n<\eta R$, choosing $\eta$ sufficiently small gives
\begin{equation}
    H_L(n)-H_L(n')
    \geq
    c_{\sigma,a,b}n^{-2}.
\label{eq:rect_left_transition_gap}
\end{equation}
The fourth-order difference is $O(n^{-4})$, while the remainder
difference is
\[
    O(n^{-6})+O(R^{-6})
    =
    O(n^{-6}).
\]
Thus, \eqref{eq:rect_left_transition_gap} gives the same ordering for
such pairs.  The region $\eta\leq x<2\eta$ is already contained in
both $\mathcal E_L$ and $\mathcal B$.

\medskip
\noindent\emph{Right edge.}
The right edge is treated with $m$ as the edge variable.  Since
\[
    n^2=\frac{R^2-pm^2}{q},
\]
set
\[
    H_R(t)
    :=
    \frac{A_{\sigma,b}}{t^2}
    +
    \frac{qA_{\sigma,a}}{R^2-pt^2}.
\]
With the same choice of $\eta$,
\[
    H_R'(t)
    \leq
    -\frac{A_{\sigma,b}}{t^3},
    \qquad
    0<t<2\eta R.
\]
The corresponding fourth-order term has derivative $O(t^{-5})$, and
the other coordinate satisfies $n\geq c_{p,q,\eta}R$.  The analogues
of \eqref{eq:rect_left_edge_gap} and
\eqref{eq:rect_left_transition_gap} therefore give strict ordering on
the right sector.

Finally, only finitely many shells lie below the large-$R$ threshold.
Increasing the index threshold beyond all indices occurring on these
shells makes the high-frequency assertion vacuous there.  We obtain a
single $N_{\sigma,a,b}$ valid for every reference shell.  Since the
Robin eigenvalues are strictly ordered in each of the two sectors, a
fixed spectral value is attained by at most one high-frequency modal
pair in each sector.
\end{proof}

\subsection{One-shell and global coincidence complexity}
\label{subsec:rect_global_complexity}

For a shell $J$, let the exceptional core consist of pairs satisfying
\[
 \min\{n,m\}<N_{\sigma,a,b}.
\]
For each fixed $n$, the shell equation $qn^2+pm^2=J$ determines at most one
$m\ge0$, and similarly after interchanging the variables.  Therefore, the
exceptional core contains at most $2N_{\sigma,a,b}$ pairs.  Combining this
with Proposition~\ref{prop:rect_discrete_sector_splitting} gives the uniform
one-shell bound
\begin{equation}
 B_{\sigma,a,b}^{\mathrm{rect}}
 :=2N_{\sigma,a,b}+2.
\label{eq:rect_one_shell_bound}
\end{equation}

We next control the number of reference shells that can meet at one perturbed
spectral level.  Define the one-dimensional physical displacements
\begin{equation}
 d_n^{(a)}(\sigma)
 :=\frac{k_n(a\sigma)^2}{a^2}-\frac{\pi^2n^2}{a^2}.
\label{eq:rect_1d_displacement}
\end{equation}
The Robin quadratic form is increasing in $\sigma$.  As in the square
case, the one-dimensional Robin eigenvalue branches are strictly increasing:
the Hellmann--Feynman derivative is the positive boundary mass of the
normalized interval eigenfunction.  Hence $d_n^{(a)}(\sigma)>0$.
Equation~\eqref{eq:rect_1d_expansion_a} gives
$d_n^{(a)}(\sigma)\to4\sigma/a$, and hence
\begin{equation}
 D_{\sigma,a}^{(1)}:=\sup_{n\ge0}d_n^{(a)}(\sigma)<\infty.
\label{eq:rect_Da}
\end{equation}
Define $D_{\sigma,b}^{(1)}$ analogously and set
\begin{equation}
 D_{\sigma,a,b}^{\mathrm{rect}}
 :=D_{\sigma,a}^{(1)}+D_{\sigma,b}^{(1)}+1.
\label{eq:rect_Dab}
\end{equation}
Then, every separated modal pair satisfies the strict displacement estimate
\begin{equation}
 0<\Lambda_{n,m}^{a,b}(\sigma)-E_{J(n,m)}
 <D_{\sigma,a,b}^{\mathrm{rect}}.
\label{eq:rect_strict_displacement}
\end{equation}
The reference levels $E_J=(\pi^2/(pq\ell^2))J$ lie in an arithmetic lattice,
so their local counting function satisfies the elementary estimate
\begin{equation}
 \mathcal N_{0,a,b}(L)
 \le
 \left\lfloor\frac{pq\ell^2}{\pi^2}L\right\rfloor+1.
\label{eq:rect_reference_density}
\end{equation}
Corollary~\ref{cor:abstract_bounded_displacement}, together with
\eqref{eq:rect_one_shell_bound} and \eqref{eq:rect_strict_displacement}, now
gives
\begin{equation}
\begin{split}
 K_{\sigma,a,b}^{\mathrm{rect}}
 &\le
 B_{\sigma,a,b}^{\mathrm{rect}}
 \mathcal N_{0,a,b}(D_{\sigma,a,b}^{\mathrm{rect}})\\
 &\le
 (2N_{\sigma,a,b}+2)
 \left(
 \left\lfloor
 \frac{pq\ell^2}{\pi^2}D_{\sigma,a,b}^{\mathrm{rect}}
 \right\rfloor+1
 \right)<\infty.
\end{split}
\label{eq:rect_global_K}
\end{equation}
This proves the spectral-complexity assertion in
Theorem~\ref{thm:rational_rectangle_fixed_robin} for $a\ne b$.

\subsection{Observation closure}
\label{subsec:rect_observation_closure}

A one-dimensional Robin eigenfunction on an interval is a linear combination
of two exponentials with frequencies $\pm k$.  Hence, each separated product
mode on $R_{a,b}$ is a sum of at most four plane waves.  Because $a\ne b$,
we count ordered modal pairs and no systematic exchange partner needs to be
included in one modal class.  Thus,
\begin{equation}
 \mathfrak F_{R_{a,b}}(\mathcal R_{n,m}^{a,b}(\sigma))\le4.
\label{eq:rect_modal_fourier_complexity}
\end{equation}
The complete eigenspace is a sum of at most
$K_{\sigma,a,b}^{\mathrm{rect}}$ such modal spaces, so
Corollary~\ref{cor:modal_complexity_observation} yields
\begin{equation}
 \mathfrak F_{R_{a,b}}(E_{\Lambda,\sigma}^{a,b})
 \le4K_{\sigma,a,b}^{\mathrm{rect}}
\label{eq:rect_eigenspace_fourier_complexity}
\end{equation}
and, for every measurable $V\subset R_{a,b}$ with $|V|>0$,
\begin{equation}
 \|u\|_{L^2(V)}^2
 \ge
 \gamma_{R_{a,b},V,4K_{\sigma,a,b}^{\mathrm{rect}}}
 \|u\|_{L^2(R_{a,b})}^2
\label{eq:rect_final_observation}
\end{equation}
for every Robin eigenvalue $\Lambda$ and every
$u\in E_{\Lambda,\sigma}^{a,b}$.  This proves
Theorem~\ref{thm:rational_rectangle_fixed_robin} in the nonsquare case; the
square case was proved in Section~\ref{sec:square_robin}.


\section{Conclusion}
\label{sec:conclusion}

We have proved uniform complete-eigenspace non-localization for the Robin
Laplacian on the equilateral triangle and on rectangles with rational
squared aspect ratio.  On the triangle, the estimate is uniform for
$\sigma$ in a nontrivial interval containing the Neumann endpoint.  For
every fixed $\sigma>0$, the conclusion holds uniformly over the entire
spectrum, both on the triangle and on each fixed rational rectangle.

The main fixed-parameter spectral phenomenon is a spectrum-wide bound on
coincidence complexity.  At high frequency, Robin perturbations split the
arithmetic reference shells up to a bounded exceptional set.  Combined with
bounded spectral displacement, this gives a uniform bound on the number of
modal classes contributing to a single Robin eigenspace.  The two geometries
realize this mechanism differently: through the coupled McCartin secular
system on the equilateral triangle and through weighted quadratic shells and
discrete same-shell splitting on rational rectangles.

Spectral simplicity is therefore not necessary for non-localization of
complete eigenspaces.  What is relevant here is a spectrum-wide bound on the
complexity of spectral coincidences.  Since each modal class has uniformly
bounded plane-wave complexity, the complete eigenspaces inherit a uniform
Fourier-complexity bound, and the multidimensional Tur\'an--Nazarov inequality
converts it into observation on every measurable set of positive measure.

The parameter dependence is different in the perturbative and
fixed-parameter regimes.  The small-Robin theorem on the triangle is uniform
in both the spectral level and $0\leq\sigma\leq\sigma_0$.  For an arbitrary
fixed $\sigma>0$, the observation constant may depend on $\sigma$, and in the
rectangular case also on the fixed aspect ratio.  No uniform lower bound is
asserted over all $\sigma>0$ or over the family of rational aspect ratios.


\end{document}